\documentclass[twoside,11pt]{article}

\usepackage{jmlr2e}

\usepackage{amsmath}
\usepackage[mathscr]{eucal}
\usepackage{graphics}
\usepackage{multicol}
\usepackage[ruled]{algorithm2e}
\usepackage{xcolor}
\usepackage{enumerate}
\usepackage[pagewise,displaymath, mathlines]{lineno}
    \usepackage[mathscr]{eucal}
    \usepackage{verbatim}
    \usepackage{cases}
    \usepackage{graphicx}
    \usepackage{url}
    \usepackage{tensor}
    \usepackage{float}
    \usepackage{url}
    \usepackage{setspace} 
    \usepackage{stackengine}
    \usepackage[font=small,format=plain,labelfont=bf,up,textfont=it,up]{caption}
    \usepackage{accents}

\DeclareFontShape{OMX}{cmex}{m}{n}{
  <-7.5> cmex7
  <7.5-8.5> cmex8
  <8.5-9.5> cmex9
  <9.5-> cmex10
}{}
\SetSymbolFont{largesymbols}{normal}{OMX}{cmex}{m}{n}
\SetSymbolFont{largesymbols}{bold}  {OMX}{cmex}{m}{n}

\makeatletter
\newcommand*\rel@kern[1]{\kern#1\dimexpr\macc@kerna}
\newcommand*\widebar[1]{%
  \begingroup
  \def\mathaccent##1##2{%
    \rel@kern{0.8}%
    \overline{\rel@kern{-0.8}\macc@nucleus\rel@kern{0.2}}%
    \rel@kern{-0.2}%
  }%
  \macc@depth\@ne
  \let\math@bgroup\@empty \let\math@egroup\macc@set@skewchar
  \mathsurround\z@ \frozen@everymath{\mathgroup\macc@group\relax}%
  \macc@set@skewchar\relax
  \let\mathaccentV\macc@nested@a
  \macc@nested@a\relax111{#1}%
  \endgroup
}
\makeatother

\newcommand{\ubar}[1]{\underaccent{\bar}{#1}}
\newcommand{\dbar}{d\hspace*{-0.08em}\bar{}\hspace*{0.1em}}

    \newcommand{\wh}{\widehat}
    \newcommand{\X}{\mathcal{X}}

    \newcommand{\wt}{\widetilde}

\allowdisplaybreaks

\ShortHeadings{Filament Confidence Regions}{Qiao}
\firstpageno{1}

\begin{document}
\title{Confidence Regions for Filamentary Structures}
%
\author{\name Wanli Qiao \email wqiao@gmu.edu\\
    \addr Department of Statistics\\
    George Mason University\\
    4400 University Drive, MS 4A7\\
    Fairfax, VA 22030, USA}
%
%
\maketitle
%
\begin{abstract} \noindent 
Filamentary structures, also called ridges, generalize the concept of modes of density functions and provide low-dimensional representations of point clouds. Using kernel type plug-in estimators, we give asymptotic confidence regions for filamentary structures based on two bootstrap approaches: multiplier bootstrap and empirical bootstrap. Our theoretical framework respects the topological structure of ridges by allowing the possible existence of self-intersections. Different asymptotic behaviors of the estimators are analyzed depending on how flat the ridges are, and our confidence regions are shown to be asymptotically valid in different scenarios in a unified form. As a critical step in the derivation, we approximate the suprema of the relevant empirical processes by those of Gaussian processes, which are degenerate in our problem and are handled by anti-concentration inequalities for Gaussian processes that do not require positive infimum variance. 
\end{abstract}
%
%



\section{Introduction}\label{introsection}

Filaments, or filamentary structures, are low-dimensional geometric features in multivariate data, which exhibit network-like topological patterns. Such structures are often characterized by the intersections of different branching pieces. Filamentary structures are often observed in, for example, road maps based on GPS data, earthquake fault lines in remote sensing data, blood vessels in medical imaging data. In particular, the distribution of galaxies in the universe show remarkable filamentary structures called Cosmic Web and it has been of great research interests among cosmologists to extract these structures and describe their topological characteristics \citep{sousbie20083d}. 

One of the mathematical models for filamentary structures is called ridges, which is a lower dimensional set where the density is locally highest in some constrained subspace. By allowing different dimensionality of the constrained subspace, ridges extend the concept of local modes. In this paper we develop rigorous theory for the construction of confidence regions for density ridges using bootstrap approaches, which respects the topological structures (potentially with self-intersections) of the ridges in the population model. 

While ridges can be used to represent the filamentary structures and are estimable from multivariate data, the topological structure of ridges is sensitive to perturbations and unstable when self-intersections exist. Its evidence based on simulations and real data examples has been observed in \cite{pulkkinen2015ridge}. Also see Figure~\ref{fig: toy}. In practice using plug-in estimators for ridges (based on kernel density estimation, for example) can hardly capture any self-intersections of ridges from data. Even though the ridge estimator and its ground truth can be asymptotically close measured by the Hausdorff distance, merely using ridge estimators is often not sufficient to represent the filamentary structures in the data, when their topology is concerned. In this difficult situation, confidence regions for ridges can be useful, because they aim to cover the true ridges.

\begin{figure}[ht]
\begin{center}
\includegraphics[scale=0.39]{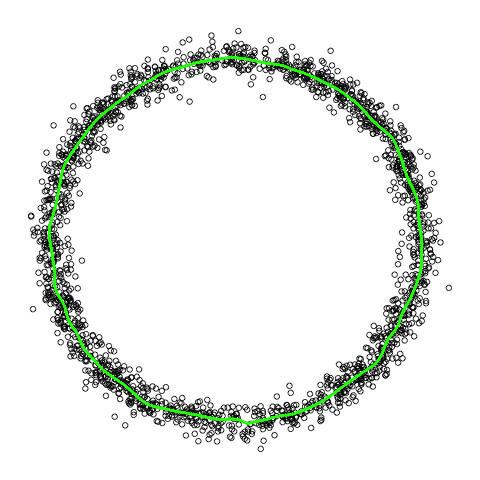}
\includegraphics[scale=0.39]{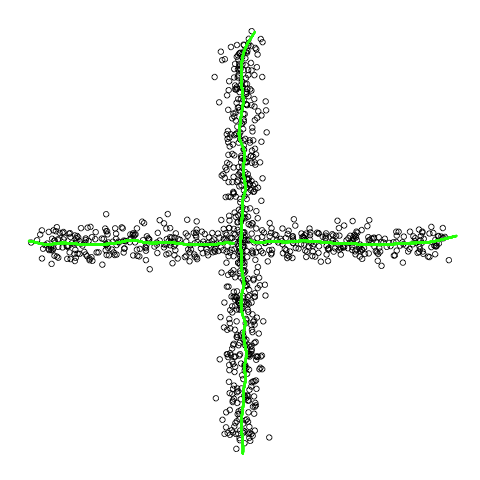}
\caption{Examples of ridge estimation for which the green curves are the ridge estimates using the data points (black dots). Left graph (circle): the estimated ridge is a closed curve without self-intersections. Right graph (cross): the ridge estimates are broken pieces near the intersection.} 
\label{fig: toy}
\end{center}
\end{figure}

Ridges with self-intersections can be described as stratified spaces, which consists of different pieces (called strata) glued together. Intersections of ridges are unstable geometric features, in the following sense. When there exists a self-intersection of the ridge induced by the underlying density function $f$, the self-intersection can disappear if $f$ is perturbed by any arbitrarily small amount measured in the $L_2$ norm. The structural instability of self-intersections is called a generic property of ridges \citep{damon1998generic,miller1998relative}. If the perturbation is considered a continuous procedure depending on a parameter (say time), then the change of structures in ridges can be characterized by a \emph{Morse transition} \citep{damon1998generic}, during which different components of ridges get merged and split as the parameter varies. This is analogous to estimating a level set which includes saddle points --- in fact, ridges are (subsets of) level sets of a special function which we call nonridgeness function in this paper. Such a level set has self-intersections at saddle points, but small perturbations of the density (by adding a small constant to the function, for example) can cause broken pieces near the intersections. See Figure~\ref{illustration}. 

\begin{figure}[ht]
\begin{center}
\includegraphics[height=4cm]{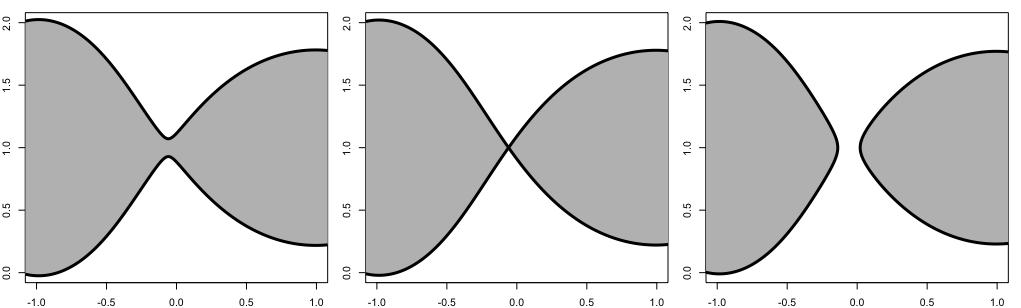}
\caption{Three plots showing the topological changes in the level sets of a Morse function near a saddle point at three levels $t_1<t_2<t_3$. The grey areas are the upper level sets, and the solid dark lines are the boundaries (level sets). It is clear that the topology of the level sets is unstable near the saddle point when the levels are perturbed. }
\label{illustration}
\end{center}
\end{figure}

In addition to the unstable topological structures, another challenge in studying the statistical inference for density ridges arises from the interplay between the first and second derivatives in the definition of ridges, as we describe now. The following definition of ridges has been widely used in the literature (see \citealt{eberly1996ridges}, for example). Let $f$ be a twice differentiable probability density function on $\mathbb{R}^d$. Its gradient and Hessian at $x$ are denoted by $\nabla f(x)$ and $\nabla^2 f(x)$, respectively. Suppose $\lambda_1(x)\geq\cdots\geq\lambda_d(x)$ are the eigenvalues of $\nabla^2 f(x)$ with associated orthonormal eigenvectors $v_1(x),\cdots,v_d(x)$. Let $r$ be an integer such that $1\leq r<d$. The $r$-ridge of $f$ is the set 
\begin{align*}
\textsf{ridge}_r(f) = \{x: v_i(x)^\top \nabla f(x)=0,i=r+1,\cdots,d, \lambda_{r+1}(x)<0\}.
\end{align*}
%
We fix $r$ in this paper and will write $\textsf{ridge}(f)=\textsf{ridge}_r(f)$ in what follows. Note that this definition involves both the first and second derivatives (through the eigenvalues and eigenvectors of the Hessian) of the density functions. Typically the rate of estimating the second derivatives is slower than that of the first derivatives (for example, using kernel density estimators as we do in this paper). This leads to different pointwise rates of convergence for ridge estimation depending on whether a ridge point is near a critical point (where the gradient is zero) or not. 
%
The asymptotic distributional theory of ridges estimation involves Gaussian approximation to the suprema of the deviation in the estimation. A consequence of different rates of convergence at different types of ridge points is the degeneracy of the approximating Gaussian processes, which have zero variance at some points. Due to such degeneracy, typical anti-concentration inequalities for Gaussian processes that require positive infimum variance (see, e.g., Lemma A.1 in \citealt{chernozhukov2014gaussian}) cannot be applied to bound the Kolmogorov distance between the distribution of the suprema of the deviation and that of the approximating Gaussian processes, which is a critical step towards the asymptotic distributional theory for ridge estimation. This problem is solved by a new type of anti-concentration inequality in the theoretical analysis of our confidence regions for ridges. 

The main idea in our approach is briefly described as follows. We denote $V(x)=(v_{r+1}(x),\cdots,v_d(x))$ and call $p(x)=\|V(x)^\top \nabla f(x)\|$ the ``nonridgeness'' function. This function measures how $x$ is different from a ridge point if $\lambda_{r+1}(x)<0$. A ridge point satisfies $p(x)=0$, which achieves the global minimum of the nonridgeness function. Using a kernel density estimator $\wh f$ for $f$ based on a random sample of size $n$, $\textsf{ridge}(f)$ can be estimated by $\textsf{ridge}(\wh f\,)$. The confidence region for $\textsf{ridge}(f)$ we consider is in the form of 
\begin{align}\label{ciform}
C_n(\epsilon_n):=\{x: \wh p(x) \leq \epsilon_n, \wh\lambda_{r+1}(x)<0\},
\end{align}
where $\wh p(x)$ and $\wh\lambda_{r+1}(x)$ are plug-in kernel estimators of $p(x)$ and $\lambda_{r+1}(x)$, respectively, and $\epsilon_n>0$ is a sequence determined by bootstrap procedures. More specifically, for an asymptotic $(1-\alpha)$ confidence region $C_n(\epsilon_n)$, $\epsilon_n$ is set to be the $(1-\alpha)$-quantile of $\sup_{x\in \mathcal{S}_n} \wh p^{\text{bs}}(x)$,  where $\wh p^{\text{bs}}(x)$ is a bootstrap version of the nonridgeness function estimate, using either the multiplier bootstrap or empirical bootstrap, and $\mathcal{S}_n$ is a small neighborhood of $\textsf{ridge}(\wh f\,)$. With such $\epsilon_n$ capturing the variation of the estimated nonridgeness around its truth, the set $C_n(\epsilon_n)$ is envisioned as a small tube around $\textsf{ridge}(\wh f\,)$. Our main result Theorem~\ref{multiplierquantile} give theoretical guarantee for the asymptotic validity of $C_n(\epsilon_n)$ as a confidence region for $\textsf{ridge}(f)$, namely, $\mathbb{P}(\textsf{ridge}(f)\subset C_n(\epsilon_n)) = 1-\alpha+o(1)$. 

{\em Related work.} Besides using ridges, there are various existing approaches to modeling local-dimensional structures (usually manifolds) embedded in higher-dimensional data, including principal curves \citep{hastie1989principal}, multiscale multianisotropic strips \citep{arias2006adaptive}, path density \citep{genovese2009path}, and medial axis of data distribution's support \citep{genovese2012geometry}. Despite its importance and usefulness in characterizing the shape of data, the estimation of stratified spaces has received much less development than manifold estimation, as pointed out in a recent review on Topological Data Analysis (TDA) by~\cite{wasserman2018topological}. Among the limited studies on sets involving self-intersections in the literature,~\cite{arias2017spectral} use local PCA to deal with intersections of manifolds, but only for the purpose of clustering instead of structure extraction or dimensionality reduction. Metric graphs \citep{accggm12, lrw14} and data skeletonization via Reeb graphs \citep{ge2011data, chazal2015gromov} can be used to extract branching structures of data. However, these methods are known to be sensitive to noise. 

In the statistical literature, early work on ridge estimation includes \cite{hall1992ridge, cheng2004estimating}. More recently, \cite{ozertem2011locally} develop an algorithm called the subspace constrained mean shift (SCMS) algorithm to extract ridges from data. \cite{qiao2021algorithms} propose two alternative ridge estimation algorithms that have consistency guarantees. \cite{genovese2014nonparametric} obtain the rate of convergence in estimating ridges by using plug-in kernel estimators. \cite{qiao2016theoretical} derive distributional results for ridge estimation on $\mathbb{R}^2$ using trajectory-wise distance. An empirical bootstrap method of constructing confidence regions for $\textsf{ridge}(\mathbb{E} \wh f)$ in the case of $r=1$ using Hausdorff distance appears in \cite{chen2015asymptotic}. Recently \cite{qiao2021asymptotic} develops confidence regions for $\textsf{ridge}(f)$ using extreme value distributions of $\chi^2$-fields on manifolds. The study of ridge estimation is also related to level set estimation, which has rich statistical literature. Confidence regions for level sets have been developed in \cite{mammen2013confidence, chen2017density, qiao2019nonparametric}, which all focus on a level set at a non-critical level where self-intersections do not exist. 

Below we highlight two major aspects that distinguish the contributions of our work from the ridge confidence region methods in \cite{chen2015asymptotic} and \cite{qiao2021asymptotic}. (1) Again we emphasize that we work under more general conditions that allow stratified structures in $\textsf{ridge}(f)$, which implies that $\textsf{ridge}(f)$ may not necessarily be manifolds. This is in fact one of the main motivations that ridges are widely used as a model of filamentary structures. The results in \cite{chen2015asymptotic} and \cite{qiao2021asymptotic} rely on assumptions that restrict the class of ridges considered in their works to be manifolds with {\em positive reach} (see \citealt{federer1959curvature}), which excludes the possible existence of self-intersections. (2) Our theoretical analysis distinguishes two cases (see Section~\ref{Twocases}) depending on how critical points are related to ridge points, for which different asymptotic behaviors of the estimators are established. Special dedication is given to solve the issue caused by the degeneracy of the approximate Gaussian processes, which further requires new anti-concentration inequalities. Our bootstrap confidence regions, in a unified form, are shown to be valid for both of these cases. Note that \cite{chen2015asymptotic} only consider a scenario when the estimation of ridges is asymptotically determined by the rate of the first derivatives. The analysis in \cite{qiao2021asymptotic} first excludes the existence of critical points on ridges and then adopts a different approach to dealing with them.

We organize the paper as follows. In Section~\ref{backgroundsec} we introduce the relevant concepts, notation and assumptions used in our results. Section~\ref{mainresult} includes the main results. In Section~\ref{gaussapprox} we present the results of Gaussian approximation to the distribution of the suprema of the estimated nonridgeness (Theorem~\ref{gaussianappanti}) by using an anti-concentration inequality for the suprema of Gaussian processes without requiring positive infimum variance (Lemma~\ref{newanticoncentration}). In Section~\ref{bootstrapsec}, we develop the bootstrap confidence regions for ridges. The validity of these confidence regions are shown in Theorem~\ref{multiplierquantile}.  All the proofs are given in Section~\ref{proofsec}.
%

\section{Preliminaries, notation and assumptions}\label{backgroundsec}

We suppose that the object $\textsf{ridge}(f)$ to be estimated is contained in a compact subset of $\mathbb{R}^d$, denoted by $\mathcal{H}$. Without loss of generality, we take $\mathcal{H}=[0,1]^d$ in this paper. Recall that $V(x)$ is a matrix with $v_{r+1}(x),\cdots,v_d(x)$ as its columns and $p(x)=\|L(x)\nabla f(x)\|$ is the nonridgeness function, where $L(x) = V(x)V(x)^\top $. Then the set of ridge points in $\mathcal{H}$ is
$$\mathcal{S}:=\textsf{ridge}(f)=\{x\in\mathcal{H}: p(x)=0, \lambda_{r+1}(x)<0\}.$$
For any $\epsilon\geq0$, let $\mathcal{S}(\epsilon)=\{x\in\mathcal{H}: p(x)\leq \epsilon, \lambda_{r+1}(x)<0\}$ such that $\mathcal{S}=\mathcal{S}(0)$.

Suppose that we have an i.i.d sample $X_1,\cdots,X_n$ from a density function $f$ on $\mathbb{R}^d$. Denote the kernel density estimator $$\wh f(x)=\frac{1}{nh^d}\sum_{i=1}^n K\Big(\frac{x-X_i}{h}\Big),$$
where $h>0$ is a bandwidth and $K$ is a kernel function. Let $\wh\lambda_{1}(x)\geq\cdots\geq \wh\lambda_d(x)$ be the eigenvalues of $\nabla^2 \wh f(x)$ with associated orthonormal eigenvectors $\wh v_1(x),\cdots,\wh v_d(x)$. Let $\wh V(x)=(\wh v_{r+1}(x),\cdots,\wh v_d(x))$ and $\wh L(x)=\wh V(x)\wh V(x)^\top$. Denote $\wh p(x)=\|\wh L(x)\nabla \wh f(x)\|$ and $\wh{\mathcal{S}}(\epsilon)=\{x\in\mathcal{H}: \wh p(x)\leq \epsilon, \wh\lambda_{r+1}(x)<0\}$ for any $\epsilon\geq0$. We write $\wh{\mathcal{S}} \equiv \wh{\mathcal{S}}(0)=\textsf{ridge}(\wh f)$, which is the plug-in ridge estimator using the kernel density estimator.

%

\subsection{Critical points on ridges}
\label{Twocases}
We will consider the following two cases in our analysis of ridge estimation.
\begin{itemize}
\item[\textbf{(a)}] There exists at least one point $x_0\in\mathcal{S}$ such that $\nabla f(x_0)\neq 0$.
\item[\textbf{(b)}] All points in $\mathcal{S}$ are critical points of $f$, that is, $\nabla f(x)=0$ for all $x\in\mathcal{S}$.
\end{itemize}
%
Case (a) is a typical scenario for ridges, where critical points of $f$ may or may not exist on $\mathcal{S}$, and for the former, the critical points may be isolated or non-isolated. Case (b) corresponds to a scenario when the ridge is completely flat. For case (b) we can equivalently write 
\begin{align}\label{caseb}
\mathcal{S} = \{x\in\mathcal{H}: \nabla f(x)=0, \lambda_{r+1}(x)<0\},
\end{align}
which is a subset of all the critical points of $f$ on $\mathcal{H}$. We also note that case (b) includes an interesting setting when $\mathcal{S}$ consists of all the local modes of $f$. Local mode estimation~\citep[e.g.,][]{chen2016comprehensive, qiao2022space} is usually considered for Morse functions, which are $C^2$ functions whose critical points are all non-degenerate~\citep{milnor1963morse}, but $f$ is not necessarily Morse in our setting. 

We distinguish these two cases because the transition from case (a) to case (b) causes some qualitative change when we use $\textsf{ridge}(\wh f\,)$ to estimate $\mathcal{S}$. It is known that using kernel estimators, the estimation of the first derivatives of the density function enjoys a faster rate of convergence than that of the second derivatives under standard assumptions, if the bias is ignored. When a ridge point is locally flat (near a critical point), the asymptotic behavior of estimating the ridge point is mainly determined by the estimation of the first derivatives; otherwise for non-flat ridge points, the slower rate of estimating the second derivatives dominates the overall estimation. As a consequence, in case (b) the uniform rate of convergence in ridge estimation only depends on the estimation of the first derivatives, while in case (a) the uniform rate is determined by the estimation of the second derivatives. This distinction is reflected in our assumption (H) below. 

\subsection{Notation and assumptions}\label{assumptionlist}
We introduce some notation used in this paper. For $x\in\mathbb{R}^d$ and a positive number $\delta$, let $\mathcal{B}(x,\delta)=\{y\in\mathbb{R}^d: \|x-y\|\leq\delta\}$ be the closed ball centered at $x$ with radius $\delta$. For any subset $\mathcal{A}\subset\mathbb{R}^d$, let $\mathcal{A}\oplus\delta=\cup_{x\in A} \mathcal{B}(x,\delta)$, $\mathcal{A}^\circ$ be the interior of $\mathcal{A}$, and $\mathcal{A}^\complement$ be its complement. For any  point $x\in\mathbb{R}^d$, let $d(x,\mathcal{A})=\inf_{y\in\mathcal{A}}\|x-y\|$ be the distance from $x$ to $\mathcal{A}$. For a (semi)metric space $(T,\dbar)$, Let $N(T,\dbar,\epsilon)$ and $N_p(T,\dbar,\epsilon)$ be the $\epsilon$-covering number and $\epsilon$-packing number of $T$, respectively. See Section~\ref{proofsec} for their precise definitions.

Let $S\mathbb{R}^{d\times d}$ be the class of all $d\times d$ symmetric real matrices. For any $A\in S\mathbb{R}^{d\times d}$, the vectorization operator $\textsf{vec}$ converts $A$ by stacking its columns into a $d^2$-dimensional column vector, while the half-vectorization operator $\textsf{vech}$ vectorizes the lower triangular part of $A$ into a $d(d+1)/2$-dimensional column vector. The duplication matrix $\mathbf{D}$, which is a $d^2\times d(d+1)/2$ constant matrix consisting of only 0's and 1's, relate the matrix operators $\textsf{vec}$ and $\textsf{vech}$ through $\textsf{vec} A = \mathbf{D} \textsf{vech}A$. Note that the operator $\textsf{vec}$ can in fact be similarly defined for matrices of any dimensions. Let $\lambda_{\min}(A)$ and $\lambda_{\max}(A)$ be the smallest and largest eigenvalues of $A$, respectively. For two matrices $A$ and $B$, let $A\otimes B$ be their Kronecker product. 
For any matrix $A$, let $A^+$ be the Moore-Penrose pseudoinverse of $A$, and $\|A\|_F$ and $\|A\|_{\text{op}}$ be the Frobenius norm and operator norm of $A$, respectively. 
For two sequences $a_n$ and $b_n$, we write $a_n\lesssim b_n$ if there exists a positive constant $C$ such that $a_n\leq Cb_n$. We also write $a_n\gtrsim b_n$ if $b_n\lesssim a_n$, and $a_n\simeq b_n$ if $a_n\lesssim b_n$ and $b_n\lesssim a_n$. 

For $\alpha=(\alpha_1,\cdots,\alpha_d)\in\mathbb{N}^d$, let $|\alpha| = \alpha_1+\cdots+\alpha_d$. For an $|\alpha|$ times differentiable function $g:\mathbb{R}^d\mapsto\mathbb{R}$, denote $\partial^{\alpha} g(x) = \frac{\partial^{|\alpha|} }{\partial^{\alpha_1}x_1\cdots \partial^{\alpha_d}x_d} g(x),\; x\in\mathbb{R}^d.$ Also we denote $d^2 g = \textsf{vech}\nabla^2 g$. For a positive integer $m$, let $\mathbb{S}_m=\{x\in\mathbb{R}^{m+1}: \|x\|=1\}$ be the unit $m$-sphere. For $k\ge0$, and $h>0$, let $\gamma_{n,h}^{(k)}=\sqrt{\log n/(nh^{d+2k})}$, which are the uniform rates of convergence of kernel estimators of the $k$th derivatives of the density centered at their expectations. For any $\gamma>0$ and $L>0$, denote by $\Sigma(\gamma,L_0,\mathcal{H})$ the the H\"{o}lder class of functions  $p:\mathbb{R}^d\rightarrow\mathbb{R}$, which are $\lfloor \gamma\rfloor$-times continuously differentiable on $\mathcal{H}$ and satisfy $|\partial^\alpha p(x) - \partial^\alpha p(y) | \leq L_0 \|x-y\|^{|\alpha|-\gamma}$ for all $x,y\in\mathcal{H}$, and $\alpha\in\mathbb{N}^d$ with $|\alpha|=\lfloor \gamma\rfloor$. Below is the definition of $\beta$-valid kernels, which can be also found in, e.g., \cite{rigollet2009optimal}.

\begin{definition}[$\beta$-valid kernel]
For a fixed $\beta>0$, a function $K: \mathbb{R}^d\rightarrow \mathbb{R}$ is called a $\beta$-valid kernel, if it has support $[0,1]^d$, and satisfies $\int_{\mathbb{R}^d} K(x)dx=1$, $\int_{\mathbb{R}^d} |K(x)|^pdx=1$ for all $p\geq 1$, $\int_{\mathbb{R}^d} \|x\|^\beta K(x)dx<\infty$. When $\lfloor \beta\rfloor \geq 1$, we further require $\int_{\mathbb{R}^d} x^s K(x)dx=0,$ for all $s=(s_1,\cdots,s_d)\in\mathbb{N}^d$ such that $1\leq |s|\leq \lfloor \beta\rfloor $, where for any $x=(x_1,\cdots,x_d)\in\mathbb{R}^d$, $x^s=\prod_{i=1}^d x_i^{s_i}$. 
\end{definition}

We use the following assumptions in this paper.\\
%
%
(${\bf F1}$) $f\in\Sigma(2+\beta,L_0,\mathcal{H})$ for some $\beta>0$ and $L_0>0$. Also there exists a constant $c_0>0$ such that $\inf_{x\in\mathcal{H}} f(x)\geq c_0$.\\[3pt]
%
%
%
%
(${\bf F2}$) There exists $\delta_0>0$ such that $\mathcal{S}\oplus \delta_0\subset\mathcal{H}$ and $\mathcal{P}_0\oplus\delta_0\subset\mathcal{H}$, where $\mathcal{P}_0:=\{y\in\mathcal{H}: p(y)=0\}$. There exists $e_0>0$ such that $\lambda_{r}(x) - \lambda_{r+1}(x)>e_0$ for all $x\in\mathcal{H}$ and $|\lambda_{r+1}(x)|>e_0$ for all $x\in\mathcal{P}_0\oplus\delta_0$.\\[3pt]
(${\bf F3}$) There exists $0<\beta^\prime<\beta$ such that for all $x\in\mathcal{S}\oplus \delta_0$
\begin{align}
\label{betaprime1}
p(x) \ge C_L \cdot d(x,\text{ridge}(f))^{\beta^\prime}, \text{ for some } C_L>0.
\end{align}
There exists $x_*\in\text{ridge}(f)$ such that for all $x\in \mathcal{B}(x_*,\delta_0)$,
\begin{align}
\label{betaprime2}
p(x) \leq C_U\cdot d(x,\text{ridge}(f))^{\beta^\prime}, \text{ for some } C_U>0.
\end{align}
%
(${\bf F4}$) There exists at least one point $x_0\in \mathcal{S}$ with some $0<r_0\le \delta_0$, $\eta_0>0$, $\phi_0>0$, and $C_0>0$ such that $N_p(\mathcal{S}_{x_0,r_0},\|\cdot\|,\eta)\geq C_0 \eta^{-\phi_0}$ for all $\eta\in(0,\eta_0]$, where $\mathcal{S}_{x_0,r_0}:=\mathcal{S}\cap \mathcal{B}(x_0,r_0)$. For case (a), we require that $\|\nabla f(x)\|>0$ for all $x\in\mathcal{S}_{x_0,r_0}$. \\[3pt]
%
(${\bf K1}$) The function $K$ is $\beta$-valid kernel, with Lipschitz continuous second partial derivatives.\\[3pt]
%
(${\bf K2}$) There exists a ball $\mathcal{B}\subset\mathbb{R}^d$ such that the component functions of both $\mathbf{1}_\mathcal{B} (s)d^2 K(s)$ and $\mathbf{1}_\mathcal{B}(s) \nabla K(s)$ are linearly independent, where $\mathbf{1}_\mathcal{B}$ in the indicator function on $\mathcal{B}$.\\ [3pt]
(${\bf H}$) We assume that there exist sequences $\rho_n\to0$ and $h=h_n\to0$ such that\\
\indent $\;$ for case (a): $\gamma_{n,h}^{(2)}=o(\rho_n)$ and $(\log n)^{3/2} (h^\beta(\gamma_{n,h}^{(2)} )^{-1} + h^{-1}\rho_n^{1/\beta^\prime}) \rightarrow 0$ as $n\rightarrow\infty$;\\
\indent $\;$ for case (b): $\gamma_{n,h}^{(1)}=o(\rho_n)$ and $(\log n)^{3/2} (h^\beta(\gamma_{n,h}^{(1)} )^{-1} + \gamma_{n,h}^{(2)} +h^{-1}\rho_n^{1/\beta^\prime})\rightarrow 0$ as $n\rightarrow\infty$.

\begin{remark}
\label{assumptionremark}
{\em
1. (F1) is a differentiability assumption. In the literature of statistical inference for ridges (see \citealt{chen2015asymptotic,qiao2016theoretical,qiao2021asymptotic}, it is often assumed that $\beta=2$, that is, $f$ is four times continuous differentiable.

2. (F2) is a set of global assumptions. They guarantee the continuity of the nonridgeness function and preclude the existence of a point $x$ such that $V(x)^\top \nabla f(x)=0$ and $\lambda_{r+1}(x) = 0$, which can be viewed as a degenerate ridge point. A similar assumption is used in \cite{genovese2014nonparametric,qiao2016theoretical,qiao2021asymptotic}.

3. (F3) is a set of local conditions which are assumed to hold in a neighborhood of the ridge $\mathcal{S}$. The conditions in \eqref{betaprime1} and \eqref{betaprime2} make the rate $\beta^\prime$ sharp. The parameter $\beta^\prime$ is used to describe the global speed of deviation of the non-ridgeness function from level 0 as the power of distance from the ridge. If we focus on small areas near the ridge, then compared with the value at regular ridge points, $\beta^\prime$ is larger near the self-intersections of ridges where the level sets of $p$ change more abruptly when the levels are perturbed. In some cases, $\beta^\prime$ can be explicitly given. For example, $\beta^\prime  = 1$ under the assumptions in \cite{genovese2014nonparametric} or \cite{qiao2021algorithms}. In an example given in~\cite{qiao2021algorithms} where $f(u,v)=\frac{3}{8}(1-u^2)v$, $u\in(-1,1), v\in(0,2)$, $\textsf{ridge}(f)$ consists of two self-intersected curves and $\beta^\prime=2$. 

Similar assumptions have been used in the literature of density level set estimation, where we replace $p$ and $\text{ridge}(f)$ by the density $f$ and one of its level sets, respectively. See, for example, \cite{singh2009adaptive,steinwart2015fully,jiang2017density,wang2019dbscan}. When the Morse theory~\citep{milnor1963morse} is applied, it is known that $\beta^\prime=1$ at a regular level and $\beta^\prime=2$ at a critical level for which self-intersections can occur near saddle points as we show in Figure~\ref{illustration}. 


4. Assumption (F4) is very mild and states that the ridge cannot be arbitrarily small like finitely many isolated points. Similar assumptions for level set estimation have been used in e.g., \cite{singh2009adaptive,wang2019dbscan}. 

5. Assumptions (K1) and (K2) on the kernel function $K$ are regular ones for ridge estimation. See \cite{qiao2021asymptotic} for some relevant discussions.

6. Our analysis (see Proposition~\ref{ridgeapptheorem} below) shows that the rate of convergence for the ridge estimation is equivalent to that of the second derivative estimation for case (a) and to that of the first derivative estimation for case (b), which is why we have different requirements for the turning parameters for these two cases in assumption (H). In this assumption, an undersmoothing bandwidth is used, which makes the bias in the kernel estimation asymptotically negligible.
}\end{remark}

\section{Main results}\label{mainresult}

The form of the confidence regions for $\mathcal{S}$ we study in this paper is $C_n(\epsilon_n)$ as given in (\ref{ciform}). Notice that $\mathcal{S}\subset C_n(\epsilon_n)$ is equivalent to (i) $\sup_{x\in\mathcal{S}} \wh p(x) \leq \epsilon_n$ and (ii) $\sup_{x\in\mathcal{S}} \wh\lambda_{r+1}(x) <0$. We can show that condition (ii) holds with a probability tending to one under the assumptions listed in Section~\ref{assumptionlist} (see the proof of Corollary~\ref{theoryci} below). This is mainly because of the strong consistency of $\wh \lambda_{r+1}$ as an estimator of $\lambda_{r+1}$, which is assumed to be strictly negative on $\mathcal{S}$. To make $C_n(\epsilon_n)$ a valid asymptotic $(1-\alpha)$ confidence region for $\mathcal{S}$, $\epsilon_n$ can be chosen to be the $(1-\alpha)$-quantile of the (asymptotic) distribution of $\sup_{x\in\mathcal{S}} \wh p(x)$. We consider bootstrap methods to approximate this distribution: $\epsilon_n$ is determined as the $(1-\alpha)$-quantile of the supremum of the bootstrap version of $\wh p$ over $\wh{\mathcal{S}}(\rho_n)$, for a sequence $\rho_n\rightarrow0$ as $n\rightarrow\infty$ satisfying assumption (H). We note that under certain conditions $\epsilon_n$ can also be obtained by bootstrapping $\wh p$ over $\wh{\mathcal{S}}=\wh{\mathcal{S}}(0)$. See Section~\ref{DirectBoot}.
%

We will first consider a multiplier bootstrap method. Let $\X_n=\{X_1,\cdots,X_n\}$ and $\sigma(\X_n)$ be the sigma-algebra generated by $\X_n$. Let $e_1,\cdots,e_n$ be independent standard normal random variables, which are independent of $\X_n$. Define
\begin{align*}
\wh f^e(x) = \wh f(x) + \frac{1}{n} \sum_{i=1}^n e_i \Big( \frac{1}{h^d} K\Big(\frac{x-X_i}{h}\Big) - \wh f(x)\Big).
\end{align*}
Let $\wh\lambda_{1}^e(x)\geq\cdots\geq \wh\lambda_d^e(x)$ be the eigenvalues of $\nabla^2 \wh f^e(x)$ with associated orthonormal eigenvectors $\wh v_1^e(x),\cdots,\wh v_d^e(x)$. Denote $\wh V^e(x)=(\wh v_{r+1}^e(x),\cdots,\wh v_d^e(x))$ and $\wh L^e (x)=\wh V^e(x)\wh V^e(x)^\top $. Let $\wh p^e (x) = \|\wh L^e (x) \nabla \wh f^e (x)\|$ and $t_{1-\alpha}^e$ be the $(1-\alpha)$-quantile of $\sup_{x\in\wh{\mathcal{S}}(\rho_n)} |\wh p^e (x) - \wh p (x)|$ conditional on $\X_n$ for $\alpha\in(0,1)$. Our $(1-\alpha)$ confidence region for $\mathcal{S}$ based on multiplier bootstrap is given by $ \wh{\mathcal{S}}(t_{1-\alpha}^e)$, which can be numerically computed based on Monte Carlo simulations. 
%
%
We note that the construction of $ \wh{\mathcal{S}}(t_{1-\alpha}^e)$ does not rely on any ridge extraction algorithm (see, e.g., \citealt{ozertem2011locally,qiao2021algorithms}), which has a computational advantage over bootstrap methods that have such a requirement. 

The same idea of bootstrapping $\wh p$ over $\wh{\mathcal{S}}(\rho_n)$ also applies to the empirical bootstrap method. See Section~\ref{empiricalboot}.


%
%

\subsection{Gaussian approximation of suprema of estimated nonridgeness}\label{gaussapprox}

In this section we study the Gaussian approximation of the distribution of $\sup_{x\in\mathcal{S}} \wh p(x)$, which will be shown to have different forms for cases (a) and (b) specified in Section~\ref{Twocases}. The approximation relies on finding the leading term in the asymptotic expansion of 
\begin{align}
\wh p(x) - p(x)
=&\|\wh L(x)\nabla \wh f(x)\| - \|L(x)\nabla f(x)\|,\label{lfexpress}
\end{align} 
which in turn involves the Taylor expansion of $\wh L(x) - L(x)$ as matrix-valued functions of the Hessian of the density. 
We need to introduce more notation. Suppose that $\nabla^2 f(x)$ has $q$ distinct eigenvalues: $\mu_1(x)>\cdots>\mu_q(x)$, each with multiplicities $\ell_1,\cdots,\ell_q$.  Here $q$ is a positive integer with $1\leq q\leq d$. Note that $q$ as well as $\ell_1,\cdots,\ell_q$ may depend on $x$ but we have suppressed this dependence in the notation for simplicity. Let $s_0=0$, $s_1=\ell_1,\cdots,s_q=\ell_1+\cdots+\ell_q$. Then $\mu_j(x)=\lambda_{s_{j-1}+1}(x)=\cdots=\lambda_{s_j}(x)$. Let $E_j(x)$ be a $d\times \ell_j$ matrix with orthonormal eigenvectors $v_{s_{j-1}+1}(x),\cdots,v_{s_j}(x)$ as its columns. Define the eigenprojection matrix $P_j(x) = E_j(x)E_j(x)^\top ,$ $j=1,\cdots,q$, which orthogonally projects any $d$-dimensional vector onto the space spanned by the columns of $E_i(x)$. Even though $E_j(x)$ is not uniquely defined when $\ell_j\geq 2$, $P_j(x)$ is independent of any particular choice of $E_j(x)$. Suppose that $\lambda_r(x)\neq \lambda_{r+1}(x)$ as in assumption (F2) and there are $q_r=q_r(x)$ distinct eigenvalues larger than $\lambda_{r+1}(x)$, that is, $\mu_{q_r}(x)=\lambda_r(x)$ and $\mu_{q_r+1}(x)=\lambda_{r+1}(x)$. Then note that by $V(x)=(E_{q_r+1}(x),\cdots,E_q(x))$ we can write
\begin{align}\label{LPexpression}
L(x) = \sum_{j=q_r+1}^q E_j(x)E_j(x)^\top =\sum_{j=q_r+1}^q P_j(x).
\end{align}
%
%
For $j,k=1,\cdots,q$, when $j\neq k$, let $\nu_{jk}(x)=[\mu_j(x) - \mu_k(x)]^{-1} $ and
\begin{align}\label{sjxexpress}
S_j(x) = \sum_{\substack{k=1 \\ k\neq j}}^q \nu_{jk}(x)P_k(x).
\end{align}
It is easy to verify that $S_j(x)  = (\mu_j(x)\mathbf{I}_d - \nabla^2 f(x))^+ $ by using the definition of pseudoinverse of matrices and the following properties: $S_j(x)P_j(x)=P_j(x)S_j(x)=0$ and $(\nabla^2 f(x) - \mu_j(x)\mathbf{I}_d) S_j = S_j (\nabla^2 f(x) - \mu_j(x)\mathbf{I}_d) = P_j(x) - \mathbf{I}_d$. 

Recall that $\mathbf{D}$ is the duplication matrix (see Section~\ref{backgroundsec}). Denote 
\begin{align}\label{MhTfullexpress}
M(x)^\top  =  \sum_{j=q_r+1}^q \{  P_j(x) \otimes [S_j(x) \nabla f(x)]^\top  + S_j(x) \otimes [P_j(x) \nabla f(x)]^\top  \} \mathbf{D}.
\end{align}
The above expression can be simplified for $x\in\mathcal{S}$. For $x\in\mathcal{S}$, since $P_j(x)\nabla f(x)=0$ for $q_r+1\leq j\leq q$ by the definition of ridges, we can write
\begin{align}\label{MhTexpress}
M(x)^\top  =  \mathop{\sum\sum}\limits_{q\geq j\geq q_r+1>k\geq 1} \nu_{jk}(x) \{  P_j(x) \otimes [P_k(x) \nabla f(x)]^\top  \} \mathbf{D}.
\end{align}
Also we denote 
\begin{align}
&\varepsilon_n^{(0)} = \sup_{x\in\mathcal{H}} |\wh f(x) - f(x)|, \label{varepsilon0}\\
&\varepsilon_n^{(1)} = \sup_{x\in\mathcal{H}}\|\nabla \wh f(x)- \nabla f(x)\|, \label{varepsilon1}\\
&\varepsilon_n^{(2)} = \sup_{x\in\mathcal{H}}\|\nabla^2 \wh f(x)-\nabla^2 f(x)\|_F^2. \label{varepsilon2}
\end{align}
%
%
Let $D_{n,1}(x) = \nabla \wh f(x)- \mathbb{E}\nabla \wh f(x)$ and $D_{n,2}(x) = d^2 \wh f(x)-\mathbb{E} d^2 \wh f(x)$, which are vectors of the kernel estimators of the first and second derivatives of the density, respectively, centered at their expectations. The following proposition provides the leading terms in the asymptotic expansion of $\sup_{x\in\mathcal{S}}\wh p(x)$, when the bias is made asymptotically negligible.
\begin{proposition}\label{ridgeapptheorem}
Assume that (F1), (F2), and (K1) hold. Suppose $n\rightarrow\infty$ and $h\rightarrow0$ such that $\max\{\varepsilon_n^{(0)},\varepsilon_n^{(1)},\varepsilon_n^{(2)}\}\rightarrow0$. We have 
\begin{align}\label{ridgenessapprox}
 \sup_{x\in\mathcal{S}}\wh p(x) = 
 \begin{cases}
 \sup_{x\in\mathcal{S}} \|M(x)^\top D_{n,2}(x) \| + O(h^\beta +\varepsilon_n^{(1)} + (\varepsilon_n^{(2)} )^2) & \text{ for case (a)}\\
\sup_{x\in\mathcal{S}} \|L(x)^\top D_{n,1}(x) \| + O(h^\beta + \varepsilon_n^{(1)} \varepsilon_n^{(2)}) & \text{ for case (b)}.
 \end{cases}
\end{align}
 .
\end{proposition}

\begin{remark}\label{ridgeleading}
{\em We give the rates of the terms on the right-hand side of of (\ref{ridgenessapprox}) in order to understand why the first terms have dominant rates in the expansions. Standard calculations for the uniform rates of convergence for the density and its derivative estimation can be used to show that $\varepsilon_n^{(0)}= O_p(\gamma_{n,h}^{(0)}) + O(h^\beta)$, $\varepsilon_n^{(1)}= O_p(\gamma_{n,h}^{(1)}) + O(h^\beta),$ and $\varepsilon_n^{(2)}= O_p(\gamma_{n,h}^{(2)}) + O(h^\beta).$ See \cite{gine2002rates, einmahl2005uniform, arias2016estimation}. Similarly, it can be shown that $\sup_{x\in\mathcal{S}} \|M(x)^\top D_{n,2}(x) \| = O_p(\gamma_{n,h}^{(2)})$ and $\sup_{x\in\mathcal{S}} \|M(x)^\top D_{n,1}(x) \| = O_p(\gamma_{n,h}^{(1)})$. Note that all the $O(h^\beta)$ terms (including those on the right-hand side of (\ref{ridgenessapprox}) are due to the bias in the kernel type estimators. If the $O(h^\beta)$ terms are all asymptotically negligible (under assumption (H), for example), then the first terms on the right-hand side of (\ref{ridgenessapprox}) are indeed the leading terms.
%
%
}\end{remark}

Next we treat $\sup_{x\in\mathcal{S}} \|L(x)^\top D_{n,1}(x) \|$ and $\sup_{x\in\mathcal{S}} \|M(x)^\top D_{n,2}(x) \|$ on the right-hand side of (\ref{ridgenessapprox}) as the suprema of empirical processes indexed by classes of functions. For $x\in\mathcal{S},$ and $z\in\mathbb{S}_{d-1}$, define 
\begin{align}
g_{x,z}(y) = 
\begin{cases}
h^{-d/2}z^\top  M(x)^\top  d^2K(\frac{x-y}{h}) & \text{ for case (a)}\\
h^{-d/2}z^\top  L(x)^\top  \nabla K(\frac{x-y}{h}) & \text{ for case (b)},
\end{cases}
\end{align}
 and the class of functions 
\begin{align*}
\mathcal{F} = \{g_{x,z}(\cdot): x\in\mathcal{S},z\in\mathbb{S}_{d-1} \}.
\end{align*}
Let $\mathbb{G}_n(g) = \frac{1}{\sqrt{n}} \sum_{i=1}^n [g(X_i) - \mathbb{E} g(X_i)]=\sqrt{n}[\mathbb{P}_n(g) - \mathbb{P}(g)]$, where $\mathbb{P}_n$ is the empirical measure based on $\X_n$ such that $\mathbb{P}_n (g) = n^{-1}\sum_{i=1}^n g(X_i)$, and $\mathbb{P}(g)=\mathbb{E}g(X)$. Then by the definition of Euclidean norm we can write
%
\begin{align*}
Z : = \sup_{g\in \mathcal{F}} \mathbb{G}_n(g) = 
\begin{cases}
\omega_n\sup_{x\in\mathcal{S}} \|M(x)^\top D_{n,2}(x) \| & \text{ for case (a)}\\
\omega_n\sup_{x\in\mathcal{S}} \|L(x)^\top D_{n,1}(x)  \| & \text{ for case (b)},
\end{cases}
\end{align*}
%
%
%
where
\begin{align*}
\omega_n = 
\begin{cases}
\sqrt{nh^{d+4}} & \text{ for case (a)}\\
\sqrt{nh^{d+2}} & \text{ for case (b)}.
\end{cases}
\end{align*}
In Proposition~\ref{ridgeapptheorem} we have shown the approximations $ \omega_n \sup_{x\in\mathcal{S}}\wh p(x) \approx Z$. In the following corollary we provide the rates of the errors in these approximations.
\begin{corollary}\label{gaussianappcorollary}
Assume that (F1), (F2) and (H) hold. Then there exists a constant $C_1>0$ such that for $n$ large enough,
\begin{align}\label{firstapproximation}
\mathbb{P} \Big(  | \omega_n \sup_{x\in\mathcal{S}}\wh p(x) - Z | > C_1 \delta_n\Big)  \leq n^{-1},
\end{align}
where 
\begin{align}
\delta_n = 
\begin{cases}
\sqrt{\log n} (h^\beta(\gamma_{n,h}^{(2)})^{-1} + h + \gamma_{n,h}^{(2)})  & \text{ for case (a)}\\
\sqrt{\log n} (h^\beta(\gamma_{n,h}^{(1)})^{-1} +\gamma_{n,h}^{(2)}) & \text{ for case (b)}.
\end{cases}
\end{align}
\end{corollary}

The suprema of empirical process $Z$ can be approximated by their Gaussian counterparts, as we show in the next proposition. Define a centered Gaussian process $G_P$ indexed by $\mathcal{F}$, which has covariance function
\begin{align*}
\mathbb{E} (G_P(g_1),G_P(g_2)) = \text{Cov} (g_1(X),g_2(X)),\; g_1,g_2\in\mathcal{F}.
\end{align*}
%
%
For two random variables $A$ and $B$ having the same distribution, we denote $A\stackrel{d}{=}B$. Let $\delta_n^+ = \delta_n + (\log n)^{1/8} (\gamma_{n,h}^{(0)})^{1/4})$.
%
%

\begin{proposition}\label{2ndapproxtheorem}
Assume that (F1), (F2), (K1) and (H) hold. Then there exists a random variable $\wt Z \stackrel{d}{=} \sup_{g\in\mathcal{F}} G_P(g)$ such that for some constant $C_1>0$
\begin{align}\label{secondapproximation}
\mathbb{P} \Big(  \Big| \omega_n \sup_{x\in\mathcal{S}}\wh p(x) - \wt Z \Big| > C_1 \delta_n^+   \Big)  = O((\log n)^{9/8} (\gamma_{n,h}^{(0)})^{1/4}).
\end{align}

\end{proposition}

The above result gives an upper bound in probability for the difference of the suprema between the estimated nonridgeness over $\mathcal{S}$ and its Gaussian approximation. We need a comparison between their distributions as a further step towards developing confidence regions for ridges. In other words, we would like to utilize Proposition~\ref{2ndapproxtheorem} to obtain results in the form of 
\begin{align}
\sup_{t \in \mathcal{A}} \Big| \mathbb{P}\Big( \omega_n \sup_{x\in\mathcal{S}}\wh p(x) \leq t\Big) - \mathbb{P}(\wt Z \leq t) \Big| = o(1), \; n\rightarrow\infty, \label{caseatarget}
%
\end{align}
for some subset $\mathcal{A}\subset \mathbb{R}$ that covers the quantiles of the distribution of $\omega_n \sup_{x\in\mathcal{S}}\wh p(x)$. To this end, we need to use an anti-concentration inequality for the suprema of Gaussian processes. In general, for a real valued random variable $\xi$, a positive real number $\epsilon>0$ and a set $\mathcal{A}\subset \mathbb{R}$, we call 
\begin{align*}
\eta(\xi,\epsilon,\mathcal{A}) = \sup_{x\in \mathcal{A}} \mathbb{P}(|\xi - t|\leq \epsilon)
\end{align*}
the concentration function. Anti-concentration inequalities refer to inequalities bounding $\eta(\xi,\epsilon,\mathcal{A})$, which reflects how likely $\xi$ concentrates around a quantity in $\mathcal{A}$. In our problem we focus on the concentration functions for $\xi=\wt Z$. The two types of comparisons between the supremum of the estimated nonridgeness and the extreme value of the Gaussian processes in (\ref{secondapproximation}) and (\ref{caseatarget}) can be linked through the follow lemma, which slightly generalizes Lemma 2.1 in \cite{chernozhukov2016empirical}. Its proof is omitted.
\begin{lemma}\label{Kolmogrovgeneral}
Let $U$, $W$ be real-valued random variables such that $\mathbb{P}(|U-W|>s_1)\leq s_2$ for some constants $s_1,s_2>0$. Then we have for any subset $\mathcal{A}\subset\mathbb{R}$,
\begin{align*}
\sup_{t\in \mathcal{A}} |\mathbb{P}(U\leq t) - \mathbb{P}(W\leq t)| \leq \eta(W,s_1,\mathcal{A})+ s_2.
\end{align*}
\end{lemma}

To get the comparison of distributions in (\ref{caseatarget}), we take $U=\omega_n\sup_{x\in\mathcal{S}}\wh p(x)$ and $W=\wt Z$ when applying the above lemma to convert (\ref{secondapproximation}) into the following form.
\begin{align}\label{boundattempt}
\sup_{t \in \mathcal{A}} \Big| \mathbb{P}\Big( \omega_n \sup_{x\in\mathcal{S}}\wh p(x) \leq t\Big) - \mathbb{P}(\wt Z \leq t) \Big| \leq \eta(\wt Z,C_1 \delta_n^+,\mathcal{A}) + (\log n)^{9/8} (\gamma_{n,h}^{(0)})^{1/4}.
\end{align}
Then the question is to show the concentration $\eta(\wt Z,C_1 \delta_n^+,\mathcal{A})$ is small. Typical anti-concentration inequalities for the supremum of a Gaussian process (Lemma A.1 in \citealt{chernozhukov2014gaussian}, for example) requires that the infimum variance of the Gaussian process over the index set is bounded away from zero, which does not hold in our problem. As shown in the following lemma, the infimum variance of the Gaussian processe $G_P$ is zero, although their supremum variances are strictly positive. Let $\bar\sigma^2=\sup_{g\in \mathcal{F}}\text{Var}(G_P(g))$ and $\ubar\sigma^2=\inf_{g\in \mathcal{F}}\text{Var}(G_P(g))$. 

\begin{lemma}\label{sigmah2bounds}
%
Assume that (F1), (F2), (K1) and (K2) hold. Then we have $\ubar\sigma^2=0$ and when $h$ is small enough there exist constants $C_1,C_2>0$ such that
\begin{align}\label{barsignabound}
C_1<\bar\sigma^2< C_2.
\end{align}
\end{lemma}
\begin{remark}
{\em The Gaussian process $G_P$ we are dealing with is degenerate with zero variance at some points, for the following two reasons. 

(i) Recall that $L(x)\nabla f(x)=0$ for a ridge point $x$, where $L(x)\nabla f(x)$ consists of $d$ functions but at most $d-r$ of them are linearly independent. The dependence of the functions in $L(x)\nabla f(x)$ brings certain redundancy to the index sets $\mathcal{F}$ and $\mathcal{S}$, which results in zero infimum variance of the Gaussian processes. We take case (a) as an example. Recall that for $g\in\mathcal{F}$, we can write $g(\cdot)=g_{x,z}(\cdot)=h^{-d/2}z^\top  M(x)^\top  d^2K(\frac{x-\cdot}{h})$ for some $x\in\mathcal{S}$ and $z\in\mathbb{S}_{d-1}$, where $M$ is given in (\ref{MhTexpress}). Let $\mathcal{V}^\perp(x)$ denote the subspace spanned by $v_{1}(x),\cdots,v_{r}(x)$, and $\mathcal{V}(x)$ denote the subspace spanned by $v_{r+1}(x),\cdots,v_{d}(x)$. Then for any fixed $x\in\mathcal{S}$, $\mathcal{V}^\perp(x)$ and $\mathcal{V}(x)$ are orthogonal subspaces. Notice that $M(x)^\top  d^2K(\frac{x-\cdot}{h})\in \mathcal{V}(x)$, and consequently $g_{x,z}\equiv 0$ for $z\in \mathbb{S}_{d-1}\cap \mathcal{V}^\perp(x)$. Hence $\text{Var}(G_P(g))$ is zero for some $g\in\mathcal{F}$. For case (b), a similar reason for $\ubar \sigma^2=0$ can be given. 

 (ii) There is another important circumstance when $\ubar\sigma^2=0$ for case (a): if the ridge $\mathcal{S}$ contains a critical point $x$, then we have $\nabla f(x)=0$ and hence $g_{x,z}\equiv 0$ for all $z\in\mathbb{S}_{d-1}$. 
}\end{remark}
 
As stated in the following lemma, we have a useful anti-concentration inequality that only requires the supremum variance of the Gaussian process to be positive, which is applicable to our problem because we show $\bar\sigma^2$ is positive in Lemma~\ref{sigmah2bounds}. We use $\Psi$ to denote the cumulative distribution function of a standard normal random variable.
 %
%

\begin{lemma}\label{newanticoncentration}
Let $T$ be a non-empty set, and let $\ell^\infty(T)$ be the set of all bounded functions on $T$ endowed with the sup-norm. Let $W(t),t\in T$ be a centered tight Gaussian random element in $\ell^\infty(T)$ such that $\bar\sigma_W^2:=\sup_{t\in T}\text{Var}(W(t))>0$. Let $\bar\sigma_{W,0} = \bar\sigma \Phi^{-1}(0.95)$ and $\mathcal{A}=[\bar\sigma_{W,0},\infty)$. Define $d(s,t):=\sqrt{\mathbb{E}[(W(t)-W(s))^2]}$, $s,t\in T$, and for $\delta>0$, define $\phi(\delta):=\mathbb{E}[\sup_{(s,t)\in T_\delta}|W(t)-W(s)|]$, where $T_\delta=\{(s,t):d(s,t)\leq\delta\}$. Then for every $0<\epsilon<\frac{1}{16}\bar\sigma_{W,0}$, 
\begin{align}\label{anticoncnew}
&\eta\Big(\sup_{t\in T} W(t), \epsilon, \mathcal{A}\Big) \nonumber\\
%
\leq & \inf_{\substack{\delta,\kappa>0\\\phi(\delta) + \kappa\delta\leq \bar\sigma_{W,0}/16 }} \Big\{ 2(1/\bar\sigma) (\epsilon + \phi(\delta) + \kappa\delta) \sqrt{\log N(T,d,\delta)} (\sqrt{2\log N(T,d,\delta)}+3) + e^{-\kappa^2/2}\Big\}.
\end{align}
\end{lemma}

In the above lemma, small values less than $\bar\sigma_{W,0}$ are not included in $\mathcal{A}$ so that we only consider the event that $\sup_{t\in T} W(t)$ is strictly above a positive level. It is unlikely for $W(t)$ where the variance is nearly zero to impact this event, because the supremum of $W(t)$ is most likely achieved at $t$ where the variance is not small. Applying the above anti-concentration inequality in (\ref{anticoncnew}) and showing that the right-hand side is small in our problem, we are able to compare the distributions of the suprema of the estimated nonridgeness and their Gaussian approximations, given in the following theorem. Let $\bar\sigma_{0} = \bar\sigma\Phi^{-1}(0.95)$. 
%

\begin{theorem}\label{gaussianappanti}
Assume that (F1), (F2), (K1), (K2) and (H) hold. Then we have
\begin{align}\label{thirdapproximation}
\sup_{t\geq \bar\sigma_{0} } \Big| \mathbb{P}\Big( \omega_n \sup_{x\in\mathcal{S}}\wh p(x) \leq t\Big) - \mathbb{P}(\wt Z \leq t) \Big| = O( (\log n)\delta_n^+ ).
\end{align}
\end{theorem}

In the above theorem the comparisons of the distributions are for $t\geq \bar\sigma_{0}$. Based on Lemma~\ref{sigmah2bounds}, it is clear that $\bar\sigma_{0}$ is a positive finite number. To utilize the distributions of the suprema of the approximating Gaussian processes to establish confidence regions for ridges, we are interested in the range of $t$ corresponding to the $(1-\alpha)$-quantiles of $\wt Z$ for all $\alpha\in (0,1)$. Then a natural question is where $\bar\sigma_{0}$ is located as quantiles of the distributions of $\wt Z$. This can be answered by the following lemma.

\begin{lemma}\label{wtzhquantile}
Assume that (F1), (F2), (F4), (K1) and (K2) hold. There exist constants $0<C_1<C_2<\infty$, and $0<C_3<\infty$ such that when $h$ is small enough,
\begin{align}
\mathbb{P}(C_1\sqrt{|\log h|} \leq \wt Z \leq  C_2\sqrt{|\log h|} ) \leq  1 - h^{C_3}.\label{gaussianexp1}
%
\end{align}
\end{lemma}

The implication of this lemma is that for any fixed $\alpha$, the $(1-\alpha)$-quantile of $\wt Z$ exceeds $\bar\sigma_{0}$ for small $h$. As a result, for large $n$ (and hence small $h$) the range of small $t$ excluded in Theorem~\ref{gaussianappanti} has little impact on comparing fixed quantiles of the distributions of the suprema of the nonridgeness and its Gaussian approximation.

The next corollary gives theoretical confidence regions for $\mathcal{S}$ based on the Gaussian approximation developed above. Let $t_{1-\alpha}$ be the $(1-\alpha)$-quantile of distribution of $\omega_n^{-1} \wt Z$. 
%

\begin{corollary}\label{theoryci}
Assume that (F1) - (F4), (K1), (K2) and (H) hold. Then 
\begin{align}\label{theorycia}
\mathbb{P}(\mathcal{S}\subset \wh{\mathcal{S}}(t_{1-\alpha})) = (1-\alpha) + O( (\log n)\delta_{n}^+ ).
\end{align}
\end{corollary}
The quantile $t_{1-\alpha}$ used in $\wh{\mathcal{S}}(t_{1-\alpha})$ is not directly computable, because not only the distribution of $\wt Z$ depends on the unknown covariance structures of $G_P$, but also the classes of functions $\mathcal{F}$ involve the unknown ridge $\mathcal{S}$. This issue can be addressed using a bootstrap method studied in the next section.

\subsection{Bootstrap confidence regions for filaments}\label{bootstrapsec}
%

Recall that $t_{1-\alpha}^e$ is the $(1-\alpha)$-quantile of $\sup_{x\in\wh{\mathcal{S}}(\rho_n)} |\wh p^e (x) - \wh p (x)|$ conditional on $\X_n$. Next we show the asymptotic validity of $\wh{\mathcal{S}}(t_{1-\alpha}^e)$ as a confidence region for $\mathcal{S}$. We suppose that $\beta^\prime$ is known to us. See the examples in Remark \ref{assumptionremark}. 

Below is a result analogous to Theorem~\ref{gaussianappanti}. For any $\gamma,\kappa,q>0$, let 
\begin{align}
\zeta_n^{\gamma,\kappa,q} = (\log n)^{3/2} \frac{\rho_n^{1/\beta^\prime}}{h}  + \frac{(\log n)^{3/2}}{\gamma^{1+q}} (n^{1/q} \gamma_{n,h}^{(0)} + (\gamma_{n,h}^{(0)})^{1/2}) + \kappa^{-1} (\gamma+n^{-1}). \label{zetan2def}
%
\end{align}
%

\begin{proposition}\label{multiplierbootapp}
Assume that (F1) - (F3), (K1), (K2) and (H) hold. For every $\alpha\in(0,1)$, $\gamma\in(0,1)$, $\kappa\in(0,1)$ and $q\geq 4$ and $n$ large enough, with probability $1-\kappa$, there exists a constant $C>0$ such that
\begin{align}\label{zetan2res}
&\sup_{t\geq \bar\sigma_{0}} \Big|\mathbb{P}\Big(\omega_n \sup_{x\in\wh{\mathcal{S}}(\rho_n)} \wh p^e(x) \leq t \;|\;\X_n\Big) - \mathbb{P}(\wt Z \leq t ) \Big| \leq C\zeta_n^{\gamma,\kappa,q},
%
 %
\end{align}
%
 %
where $\bar\sigma_{0}$ is given in Theorem~\ref{gaussianappanti}.
\end{proposition}

The Gaussian approximation using $\wt Z$ serves as a bridge that connects the distribution of $\omega_n \sup_{x\in\mathcal{S}}\wh p(x)$  and its bootstrap approximation $\omega_n \sup_{x\in\wh{\mathcal{S}}(\rho_n)} \wh p^e(x)$, as can be seen from Theorem~\ref{gaussianappanti} and Proposition~\ref{multiplierbootapp}. Combining these two results, we can obtain the asymptotic validity of $\wh{\mathcal{S}}(t_{1-\alpha}^e)$ as a confidence region of $\mathcal{S}$.

\begin{theorem}\label{multiplierquantile}
Assume that (F1) - (F4), (K1), (K2) and (H) hold. Then 
\begin{align}
\label{multiplierquantileres}
\mathbb{P}( \mathcal{S}\subset \wh{\mathcal{S}}(t_{1-\alpha}^e) )  = (1-\alpha) + O(\tau_n).
\end{align}
where $\tau_n=(\log n)\delta_{n}^++(\sqrt{\log n} (\gamma_{n,h}^{(0)})^{1/6})^{1-\eta} +  (\log n)^{3/2}  h^{-1}\rho_n^{1/\beta^\prime}=o(1)$. In addition, there exist positive constants $C_1$ and $C_2$ such that when $n$ is large enough we have with probability at least $1-n^{-C_1}$,
\begin{align}\label{confiregsize}
\sup_{x\in \wh{\mathcal{S}}(t_{1-\alpha}^e)} d(x,\mathcal{S}) \leq C_2\sqrt{\log n}/\omega_n.
%
%
\end{align}
%
%
%
%
%
\end{theorem}
\begin{remark}\label{casebfaster}
{\em 1. Note that unlike the confidence region $\wh{\mathcal{S}}(t_{1-\alpha})$ given in Corollary~\ref{theoryci} where $t_{1-\alpha}$ explicitly depends on $\omega_n$ which are different for cases (a) and (b), the form of $\wh{\mathcal{S}}(t_{1-\alpha}^e)$ does not distinguish these two cases, although we have different requirements for the bandwidth $h$ as specified in assumption (H).


2. The confidence region $\wh{\mathcal{S}}(t_{1-\alpha}^e)$ is constructed by including the area where the value of $\wh p$ is slightly above zero. The result in (\ref{confiregsize}) describes the location of $\wh{\mathcal{S}}(t_{1-\alpha}^e)$, that is, asymptotically it is unlikely that this confidence region contains points far away from the true ridge $\mathcal{S}$.
}\end{remark}

\subsection{Estimation and Inference for the Unknowns}
The framework of deriving our confidence regions involves the parameter $\beta^\prime$ and the distinction between cases (a) and (b). Sometimes such information is unknown and requires estimation and inference, which is addressed in this section.

\label{unknowninference}

\subsubsection{Estimation of $\beta^\prime$}
\label{betaestimation}
The knowledge of $\beta^\prime$ is required to ensure that the sequence $\rho_n$ satisfies assumption (H). Below we consider the estimation of $\beta^\prime$. For a positive sequence $r_n$ converging to zero, define $\wh\beta^\prime = \log_{r_n}(R_n)$, where 
\begin{align}
R_n = \inf_{x\in\mathcal{H} } \sup_{y\in\mathcal{B}(x,r_n)} \wh p(y).
\end{align}
This type of {\em vernier} also appears in \cite{singh2009adaptive} for level set estimation.
\begin{theorem}
\label{betaprimeest}
Suppose that $(\gamma_{n,h}^{(2)} + h^\beta)^{1/\beta^\prime} = o(r_n)$ where $r_n\to0$ as $n\to\infty$. Assume that (F1) - (F3), and (K1) hold. There exists a positive constant $C$ such that for any $\ell$ and $n$ large enough, we have with probability at least $1-n^{\ell}$ that $|\wh\beta^{\prime} - \beta^{\prime}| \le C/|\log r_n|.$
\end{theorem}

\begin{remark}
\label{betaprimeest}
{\em This result guarantees that if we choose $\rho_n$ satisfying a condition analogous to assumption (H), where we replace $\beta^\prime$ by $\wh\beta^{\prime} + |\epsilon_n^\prime|$ for some $\epsilon_n^\prime\to0$ such that $|\log r_n|=o(\epsilon_n^\prime)$, e.g., $\epsilon_n^\prime=|\log r_n|^\gamma$ for some $\gamma\in(0,1)$, then assumption (H) is also satisfied with probability at least $1-n^{\ell}$ for $n$ large enough, and the asymptotic validity of the confidence region in Theorem~\ref{multiplierquantile} still holds. Here $r_n$ can be chosen to satisfy the condition in the above theorem without exactly knowing $\beta^\prime$, such as $1/\log n$.}
\end{remark}

\subsubsection{Hypothesis testing between cases (a) and (b)}
As mentioned in Remark \ref{casebfaster}, although our confidence region has a uniform form for cases (a) and (b), there are different requirements for these two cases in assumption (H). Next we consider the hypothesis testing problem for these two cases, i.e.,
\begin{align*}
&H_0: \nabla f(x)=0 \text{ for all }x\in\mathcal{S},\\
\text{versus } &H_1: \text{there exists }x_0\in\mathcal{S} \text{ such that }\|\nabla f(x_0)\| >0.
\end{align*}
Define $t_n = \wh\beta^{\prime} + 1$.
Our test statistic is given by
\begin{align}
T_n = \sup_{x\in \wh{\mathcal{S}}(\rho_n)} \inf_{y\in \mathcal{B}(x,\rho_n^{1/t_n})}\|\nabla \wh f(y)\|.
\end{align}
For $\alpha\in(0,1)$, let $\tau^e_{1-\alpha}$ be the $(1-\alpha)$-quantile of $\sup_{x\in\wh{\mathcal{S}}(\rho_n)} \|\nabla \wh f^e (x) - \nabla \wh f (x)\|$ conditional on $\X_n$. We reject $H_0$ if $T_n \geq \tau^e_{1-\alpha}.$ The following result states that this hypothesis test asymptotically achieves type I error at most $\alpha$ and has power tending to 1.

\begin{theorem}
\label{testab}
Assume that (F1) - (F4), (K1), (K2) and (H) for case (a) hold, and $(\gamma_{n,h}^{(2)} + h^\beta)^{1/\beta^\prime} = o(r_n)$ where $r_n\to0$ as $n\to\infty$. Then,
\begin{align*}
\text{under } H_0: &\;\;\;\mathbb{P}(T_n < \tau^e_{1-\alpha}) \geq (1-\alpha) + o(1); \\
\text{under } H_1: &\;\;\;\mathbb{P}(T_n \geq \tau^e_{1-\alpha}) \rightarrow 1.
\end{align*}
\end{theorem}

\begin{remark}
\label{betaprimeest}
{\em Note that the requirement for $h$ and $\rho_n$ in assumption (H) is stronger for case (a) than (b). We use the stronger condition so that we can analyze both the type I error and power without having to determine the case at the beginning. Note that Remark~\ref{betaprimeest} also applies here, that is, the requirement for $\rho_n$ is still sufficient when $\beta^\prime$ is replaced by $\wh\beta^{\prime} + |\epsilon_n^\prime|$ for some $\epsilon_n^\prime\to0$ such that $|\log r_n|=o(\epsilon_n^\prime)$. }
\end{remark}

\subsection{Some Variants}
In this section we discuss three variants of the confidence region $\wh{\mathcal{S}}(t_{1-\alpha}^e)$. 
\subsubsection{Confidence regions based on empirical bootstrap}
\label{empiricalboot}

We can also use the empirical bootstrap to construct confidence regions for $\mathcal{S}$. Let $X_1^*,\cdots,X_n^*$ be an empirical bootstrap sample of $\X_n$. Define
\begin{align*}
\wh f^*(x) = \frac{1}{nh^d} \sum_{i=1}^n K\Big(\frac{x-X_i^* }{h}\Big) .
\end{align*}
Let $\wh\lambda_{1}^*(x)\geq\cdots\geq \wh\lambda_d^*(x)$ be the eigenvalues of $\nabla^2 \wh f^*(x)$ with associated orthonormal eigenvectors $\wh v_1^*(x),\cdots,\wh v_d^*(x)$. Denote $\wh V^*(x)=(\wh v_{r+1}^*(x),\cdots,\wh v_d^*(x))$ and $\wh L^* (x)=\wh V^*(x)\wh V^*(x)^\top $. Let $\wh p^* (x) = \|\wh L^* (x) \nabla \wh f^* (x)\| = \|\wh V^* (x)^\top  \nabla \wh f^* (x)\|$ and $t_{1-\alpha}^*$ be the $(1-\alpha)$-quantile of $\sup_{x\in\wh{\mathcal{S}}(\rho_n)} |\wh p^* (x) - \wh p(x)|$ conditional on $\X_n$ for $\alpha\in(0,1)$. Our $(1-\alpha)$ confidence region for $\mathcal{S}$ based on empirical bootstrap is given by $\wh{\mathcal{S}}(t_{1-\alpha}^*)$, which can be approximated based on Monte Carlo simulations.
The result in Theorem~\ref{multiplierquantile} still holds if we replace $\wh{\mathcal{S}}(t_{1-\alpha}^e)$ by $\wh{\mathcal{S}}(t_{1-\alpha}^*)$. Conditional on the sample $\X_n$, the bootstrap nonridgeness functions $\wh p^e(x)$ and $\wh p^*(x)$ on $\textsf{ridge}(\wh f\,)$ behave like a Gaussian process and an empirical process, respectively, which is attributed to the major difference between the theoretical justification for the asymptotic validity of the confidence regions based on multiplier bootstrap and empirical bootstrap, although the main steps of the proofs for these two theorems are similar. 

\subsubsection{Bootstrapping over the estimated ridge $\wh{\mathcal{S}}$}
\label{DirectBoot}

The confidence region $\wh{\mathcal{S}}(t_{1-\alpha}^e)$ relies on bootstrapping a supremum quantity over $\wh{\mathcal{S}}(\rho_n)$, a small neighborhood of $\wh{\mathcal{S}}$. The main technical reason for using this neighborhood is to guarantee $d_H(\mathcal{S}, \wh{\mathcal{S}}(\rho_n)) \le C\rho_n^{1/\beta^\prime}$ with a large probability (see Lemma \ref{ridgeinclusionlemma}), where $d_H$ is the Hausdorff distance between subsets of $\mathbb{R}^d$. In fact, under some further assumptions such that $d_H(\mathcal{S}, \wh{\mathcal{S}})\le C\rho_n^{1/\beta^\prime}$ with a large probability, it is sufficient to bootstrap the supremum over $\wh{\mathcal{S}}$ itself. For example, under the assumption $\beta=4$ and  $\nabla (L(x) \nabla f(x))$ has rank $d-r$ for all $x\in\mathcal{S}$, and using a 2-valid kernel, it is shown in Theorem 4 of \cite{qiao2021algorithms} that $d_H(\mathcal{S}, \wh{\mathcal{S}}) \le C(\gamma_{n,h}^{(2)} + h^2)$ with a large probability. This is enough to establish that $\wh{\mathcal{S}}(\wt t_{1-\alpha}^e)$ covers $\mathcal{S}$ with asymptotic probability $1-\alpha$, following the same procedure of proof as we have provided for $\wh{\mathcal{S}}(t_{1-\alpha}^e)$ in Theorem~\ref{multiplierquantile}. Here for any $\alpha\in(0,1)$, $\wt t_{1-\alpha}^e$ is the $(1-\alpha)$-quantile of the distribution of $\sup_{x\in\wh{\mathcal{S}}} |\wh p^e (x) - \wh p (x)|$ conditional on $\X_n$, which equals $t_{1-\alpha}^e$ when $\rho_n\equiv0$. 

\subsubsection{Confidence regions for ridges of log-density}

Often what is of interest is the estimation of the ridge of the log-density, that is, $\textsf{ridge}(\log( f))$. The application of the logarithm transformation has an advantage in stabilizing the ridge estimation, as argued in \cite{genovese2014nonparametric}. Also, in the special case that $f$ is Gaussian, $\textsf{ridge}(\log( f))$ is connected to PCA, which is also a reason that the a log transformation of the kernel density estimation is implicitly used in the original SCMS algorithm proposed by \cite{ozertem2011locally}. In fact, the two sets $\textsf{ridge}(f)$ and $\textsf{ridge}(\log( f))$ coincide under certain conditions. This can seen from the fact that the gradients of $f$ and $\log (f)$ have the same direction, and for $i=r+1,\cdots,d$, the eigenvector $v_i(x)$ of $\nabla^2 f(x)$ is an eigenvector of the Hessian of $\log( f (x))$ associated with eigenvalue $\lambda_i(x)/f^2(x)$, for $x\in\textsf{ridge}(f)$. 

The method that we construct confidence regions for $\textsf{ridge}( f)$ can be analogously applied to those of $\textsf{ridge}(\log( f))$: Let $p_\textsc{l}$ be the nonridgeness function of $\log(f)$, and $\lambda_{\textsc{l},r+1}$ be the $(r+1)$th largest eigenvalue of $\nabla^2 [\log(f)]$. For any $\epsilon\geq0$, define $\mathcal{S}_\textsc{l}(\epsilon)=\{x\in\mathcal{H}: p_\textsc{l}(x)\leq \epsilon, \lambda_{\textsc{l},r+1}(x)<0\}$. Let $\wh p_\textsc{l}$ and $\wh{\mathcal{S}}_\textsc{l}$ be the plug-in estimators of $p_\textsc{l}$ and $\mathcal{S}_\textsc{l}$, respectively.  The estimator $\wh p_\textsc{l}^e$ is similarly defined based on multiplier bootstrap. An asymptotic $(1-\alpha)$ confidence region for $\textsf{ridge}(\log( f))$ is given by $\wh{\mathcal{S}}_\textsc{l}(t_{\textsc{l},1-\alpha}^e)$, where $t_{\textsc{l},1-\alpha}^e$ is the $(1-\alpha)$-quantile of the distribution of $\sup_{x\in\wh{\mathcal{S}}_\textsc{l}(\rho_n)} |\wh p_\textsc{l}^e (x) - \wh p_\textsc{l} (x)|$. The asymptotic validity of the confidence region can be shown following the same proof for Theorem \ref{multiplierquantile}, while all the assumptions are with respect to $\log(f)$.  

We note that the above three variants of the confidence region $\wh{\mathcal{S}}(t_{1-\alpha}^e)$ correspond to three independent aspects and thus their combinations also exist. 
\subsection{Examples}\label{examplesec}

We only use two 2-dimensional examples below to give a proof of concept for our confidence regions. In our theoretical results we consider a region strictly contained in the support of the density function. For this reason, the confidence regions are restricted to the region where the kernel density estimates exceed the 5\% quantiles of the density estimates evaluated on the sample points. Below we only show the results using the empirical bootstrap method and the corresponding results based on the multiplier bootstrap look very similar. We have applied a logarithm transformation to the kernel density estimator and its bootstrap version in the examples. 

In the first example shown in Figure~\ref{fig: simpsimu}, we have two random samples, each of size 2000, generated from density functions with circle shape ridges and the 90\% confidence regions for the ridges are shown. These two examples correspond to the two cases (a) and (b), respectively. In this example, we see how the local geometry near the ridges impacts the shape of the confidence regions. Not surprisingly, the confidence region of the ridge corresponding to case (b) is a band of roughly constant width. Using the sample corresponding to case (a), the confidence band has varying local width, which reflects how steep the nonridgeness function increases away from the ridge. Also note that in this example the ridge for case (b) consists of all the modes of the density function. The modes are all degenerate in this example because the density function is flat on them, which means that the density function is not a Morse function. See the discussion of cases (a) and (b) in Section~\ref{Twocases}.

\begin{figure}[ht]
\begin{center}
\includegraphics[scale=0.31]{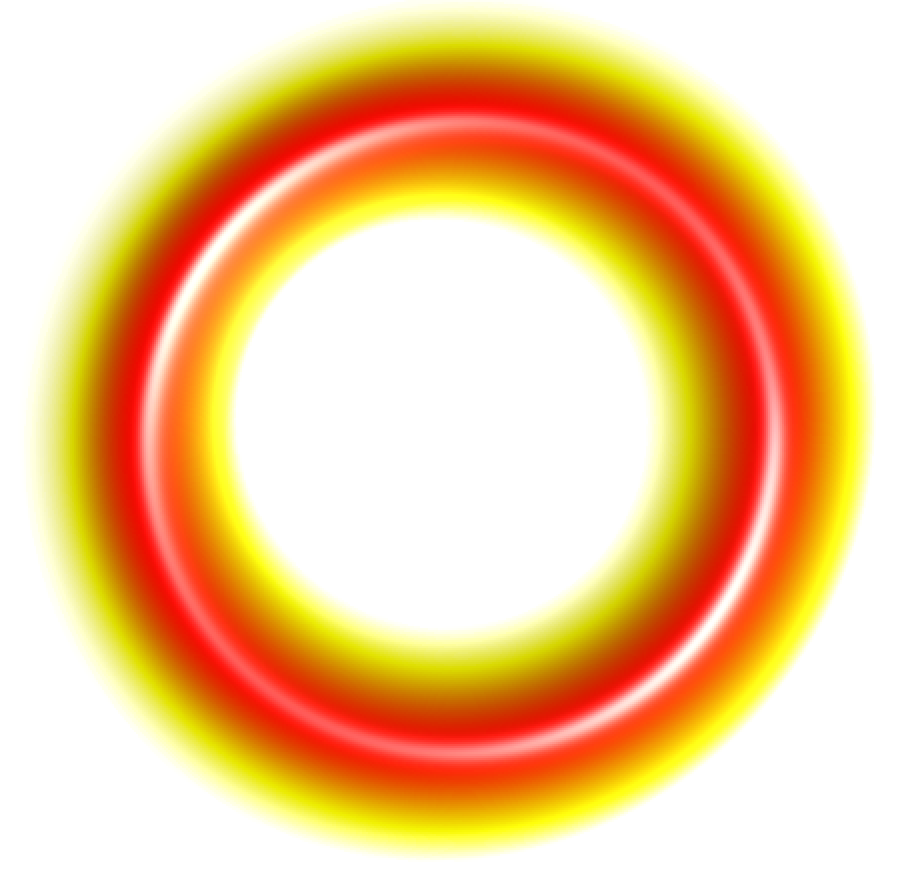}
\includegraphics[scale=0.225]{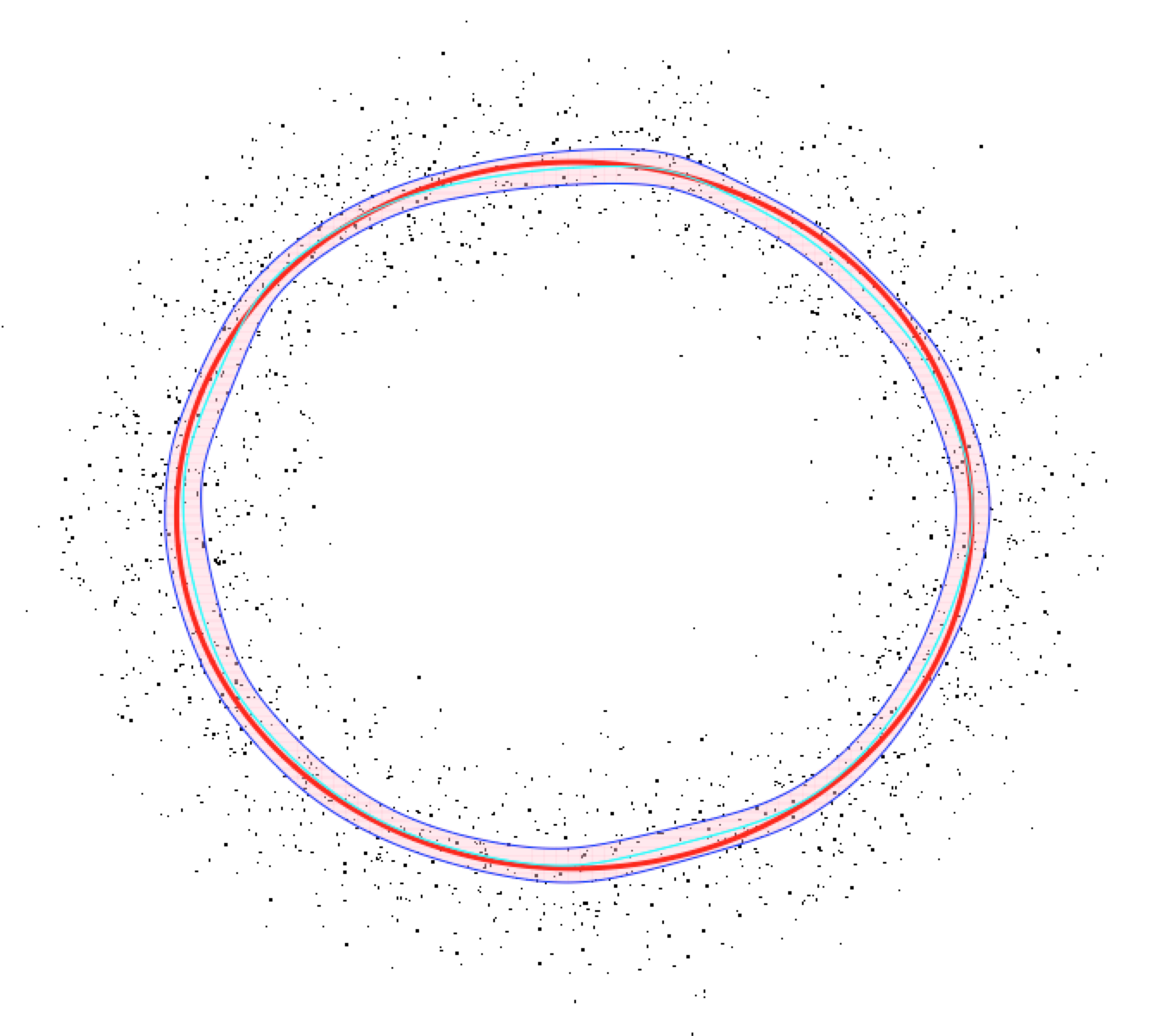}
\includegraphics[scale=0.316]{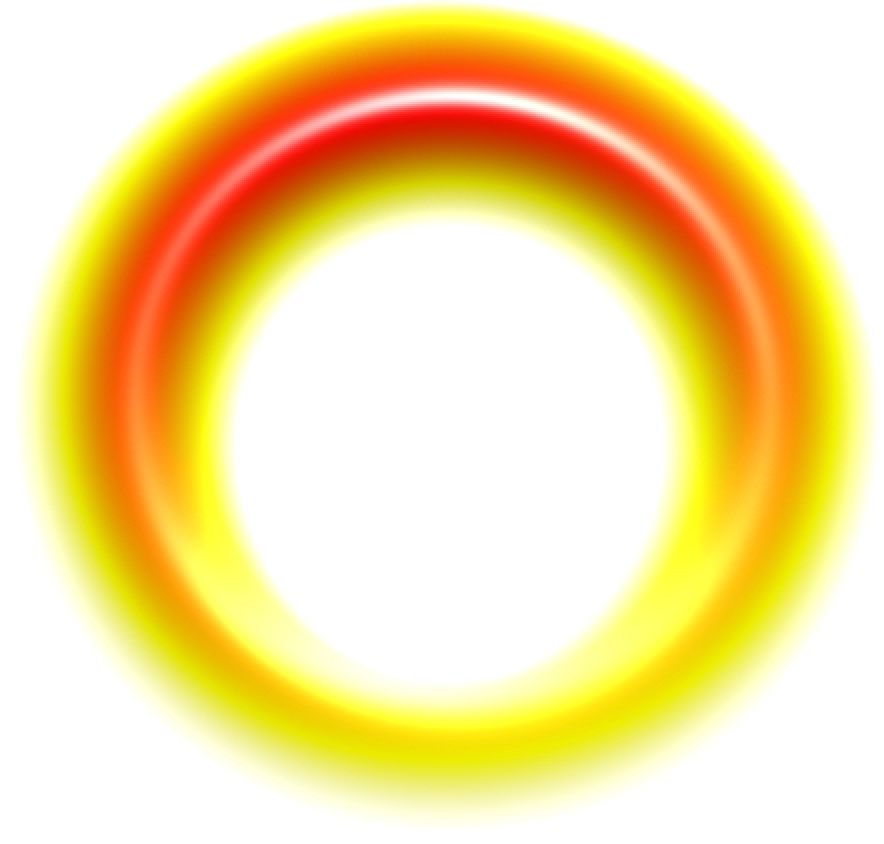}
\includegraphics[scale=0.2084]{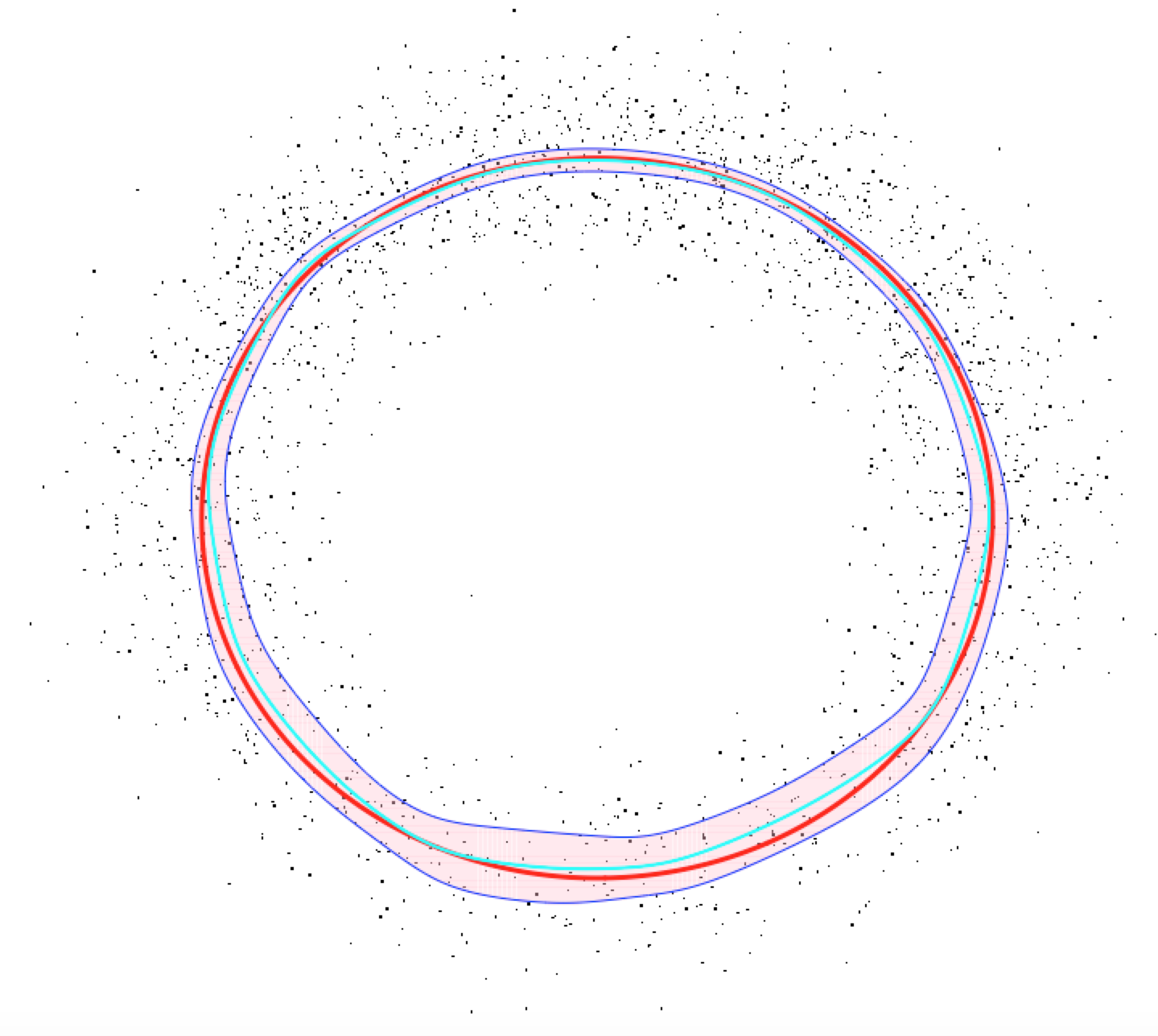}
\caption{First example. Upper panels: result for case (b); lower panels: result for case (a);  The left graphs show the 3D plots of density functions from which the sample points are generated. The right graphs show true ridge (red curve), estimated ridge (cyan curves), and 90\% confidence regions ridges using empirical bootstrap (shaded pink regions) with boundaries (blue curves). The corresponding confidence regions using multiplier bootstrap look similar.}
\label{fig: simpsimu}
\end{center}
\end{figure}

The second example belongs to case (b) but with self-intersections.  As shown in Figure~\ref{fig: simulation}, a sample of size 5000 is generated from a density whose ridge is a sun cross. The plug-in estimate of the ridge itself is not connected at the intersections as expected --- see the discussions in Section~\ref{introsection}. But we are able to use the 90\% bootstrap confidence region to capture the structure of the true ridge, especially the self-intersections.

\begin{figure}[ht]
\begin{center}
\includegraphics[scale=0.31]{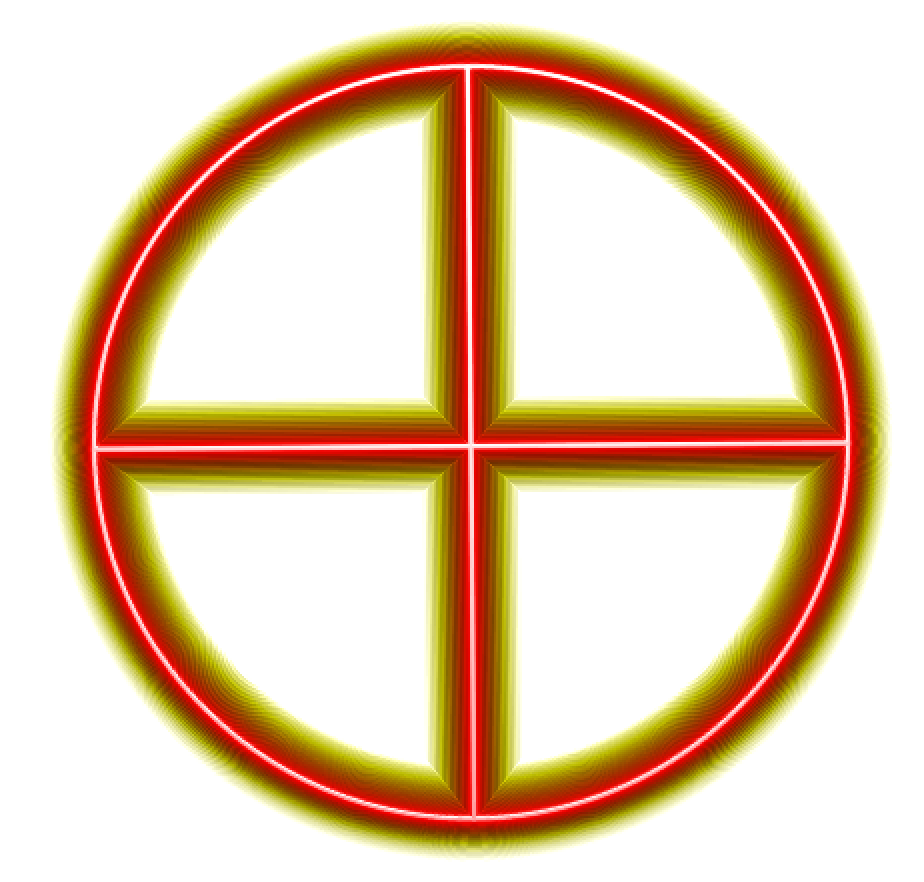}
\includegraphics[scale=0.245]{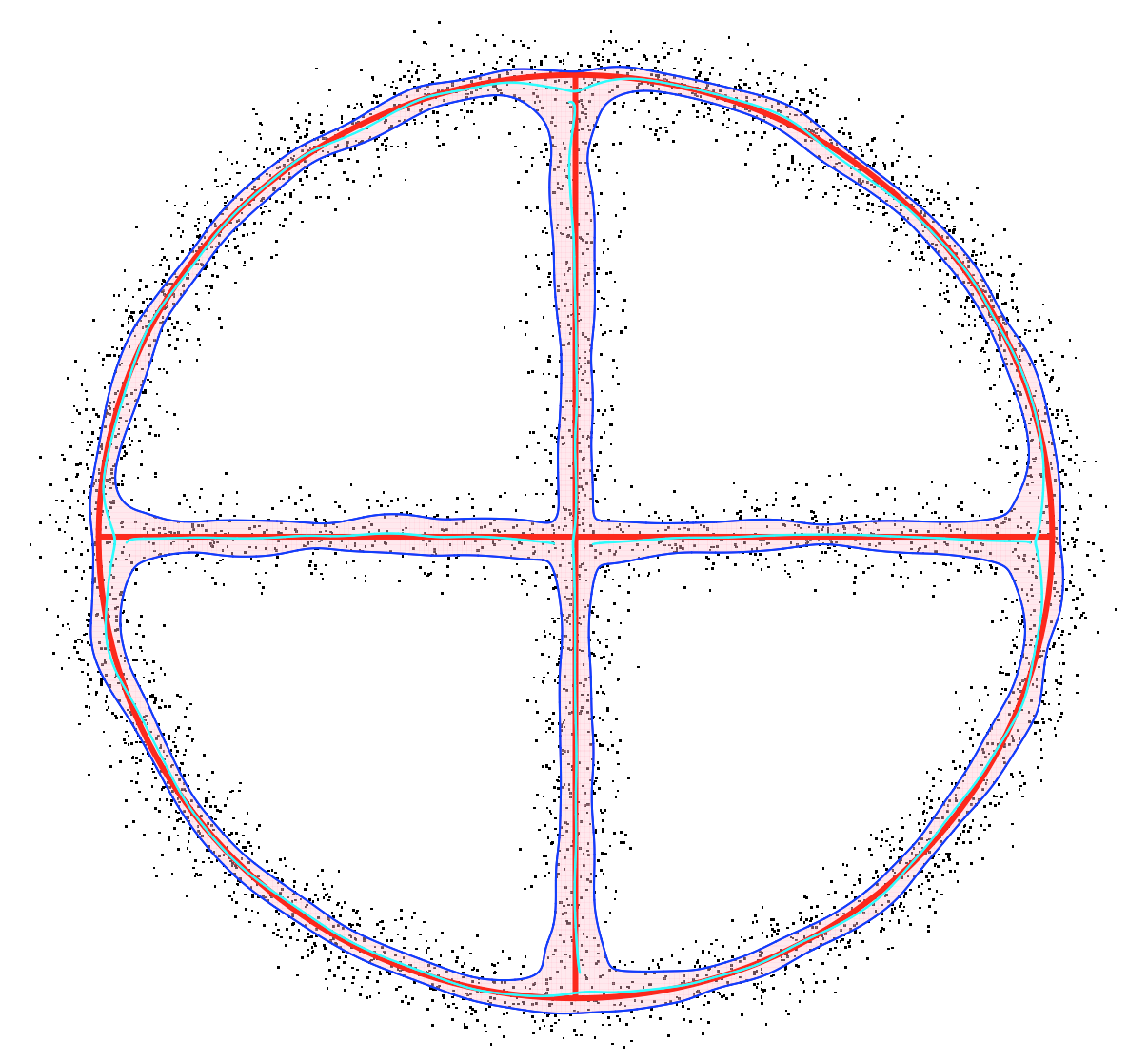}
\caption{Second example with a ridge in the shape of a sun cross. Left panel:  the 3D plot of density function; right panels: true ridge (red curve), estimated ridge (cyan curves), and 90\% confidence regions for ridges using empirical bootstrap (shaded pink regions) with boundaries (blue curves). The corresponding confidence region using multiplier bootstrap looks similar.}
\label{fig: simulation}
\end{center}
\end{figure}


\section{Discussion}

In this paper we develop bootstrap confidence regions for density ridges. Instead of restricting the study only to ridges as manifolds with positive reach, we consider ridges as filamentary structures by allowing the existence of self-intersections. Along the theoretical analysis of these confidence regions, as a critical step we adopt an anti-concentration inequality for the suprema of Gaussian processes without requiring the infimum variance to be positive. We cover two different cases depending on whether the ridge is completely flat, which gives a complete picture of the behaviors of the ridge estimators. 

An open and interesting question is the topological relation between our confidence regions and the filamentary structure of the true underlying model. We have shown in Theorem~\ref{multiplierquantile} that with a high probability the confidence regions are close to the true filament, so that the confidence regions can be viewed as tubes or bands around it. We expect that the confidence regions and the true filament have some topological similarity, and if so, the confidence regions can be used to make statistical inference for the topology of the filamentary structure. Here some ideas from \cite{genovese2014nonparametric} can be useful, where they show that under some regularity conditions $\textsf{ridge}(\wh f\,)\oplus\delta_n$ are $\mathcal{S}$ are {\em nearly homotopic}, for a $\delta_n>0$ depending on some unknown quantities of the model. 

Further topological information can be extracted from our confidence regions. For example, one can use the Reeb graphs (\citealt{ge2011data}) or the medial axes~\citep{genovese2012geometry} of our confidence regions as estimates of the filaments to represent the structure of the data. Again it is an open question to find out the topological relation between these estimates and the filaments of the true underlying model. 

\section{Proofs}\label{proofsec}

In the proofs we use $c,c_1,c_2,\cdots,C,C_1,C_2,\cdots$ to denote constants that may change from place to place.

Let $(T,\dbar)$ be an arbitrary (semi)metric space. For any $\epsilon$, a subset $T_\epsilon$ of $T$ is called an $\epsilon$-covering of $T$ if for any $t$, there exists a point $t_\epsilon\in T_\epsilon$ such that $\dbar(t,t_\epsilon)\leq \epsilon$. $T_\epsilon$ is called an $\epsilon$-packing of $T$ if $d(t_1,t_2)>\epsilon$ for any two different $t_1,t_2\in T_\epsilon$. $T_\epsilon$ is called an $\epsilon$-net of $T$ if $T_\epsilon$ is both an $\epsilon$-covering and $\epsilon$-packing of $T$. The $\epsilon$-covering number $N(T,\dbar,\epsilon)$ of $T$ is the minimum cardinality of $\epsilon$-covering of $T$. The $\epsilon$-packing number $N_p(T,\dbar,\epsilon)$ of $T$ is the maximum cardinality of $\epsilon$-packing of $T$. 

We use the following definition of VC type classes. Let $(\mathcal{X},\mathcal{A})$ be a measurable space. For all $m\geq 1$, any probability measure $Q$ on $(\mathcal{X},\mathcal{A})$ and all measurable functions $g:\mathcal{X}\mapsto \mathbb{R}$, let $\|g\|_{Q,m}=(\int_\mathcal{X} |g|^m d Q)^{1/m}$. For two measurable functions $g_1,g_2$, define $e_Q(g_1,g_2)=\|g_1-g_2\|_{Q,2}$. Let $\mathcal{L}$ be a class of measurable functions from $\mathcal{X}$ to $\mathbb{R}$, with an envelope $F:\mathcal{X}\mapsto \mathbb{R}$ meaning $\sup_{f\in \mathcal{L}} |f(x)|\leq F(x)$ for all $x\in\mathcal{X}$. $\mathcal{L}$ is called VC type if there exist constants $A,v>0$ such that $\sup_{Q} N(\mathcal{L},e_Q,\varepsilon \|F\|_{Q,2})\leq (A/\varepsilon)^v$ for all $\varepsilon\in(0,1]$, where the supremum is taken over all finitely discrete probability measures on $(\mathcal{X},\mathcal{A})$. We call $(A,v)$ the VC constants of $\mathcal{L}$. 

\subsection{Auxiliary Results}
We first give a few auxiliary results that will be used in the subsequent proofs. 

Our proofs use the Talagrand's inequality (see Corollary 2.2 in \citealt{gine2002rates}). The following version of this inequality is taken from Theorem B.1 in \cite{chernozhukov2014anti}.
\begin{lemma}[Talagrand's inequality]\label{talagrandineq}
Let $X_1,\cdots,X_n$ be i.i.d. random variables taking values in a measurable space $(S,\mathcal{S})$. Suppose that $\mathcal{G}$ is a nonempty, pointwise measurable class of functions on $S$ uniformly bounded by a constant $b$ such that there exist constants $a\geq e$ and $v>1$ with $\sup_Q N(\mathcal{G},e_Q,b\varepsilon)\leq (a/\varepsilon)^v$ for all $0<\varepsilon\leq 1$. Let $\sigma^2$ be a constant such that $\sup_{g\in\mathcal{G}}\text{Var}(g)\leq\sigma^2\leq b^2$. If $b^2v\log(ab/\sigma)\leq n\sigma^2$, then for all $t\leq n\sigma^2/b^2$,
\begin{align*}
\mathbb{P}\Big[ \sup_{g\in\mathcal{G}} |\mathbb{G}_n(g)| > A\sigma \sqrt{t\wedge (v\log (ab/\sigma))}\Big]\leq e^{-t},
\end{align*}
where $A>0$ is an absolute constant. 
\end{lemma}

We will use the following two lemmas.
\begin{lemma}\label{biasbound}
Assume that (F1) and (K1) hold. For any $\alpha\in\mathbb{N}^d$ with $|\alpha|=0,1,2$, there exists a constant $C_\alpha>0$ such that 
\begin{align}\label{biasboundres}
\sup_{x\in \mathcal{H} } |\mathbb{E} \partial^{\alpha} \wh f(x) - \partial^{\alpha} f(x) | \leq C_\alpha h^\beta.
\end{align}
\end{lemma}
The proof of Lemma~\ref{biasbound} uses the definition of the $\beta$-valid kernel, and follows from standard arguments for the bias of kernel density estimate. It is similar to the proof of Lemma 2 in \cite{arias2016estimation} and the proof of Lemma 4.1 of \cite{rigollet2009optimal}, and hence is omitted. 

\begin{lemma}\label{talabound}
Assume that (F1), (K1), and (K2) hold and $\gamma_{n,h}^{(2)}\rightarrow0$. For any $\ell>0$ and $\alpha\in\mathbb{N}^d$ with $|\alpha|=0,1,2$, there exists a constant $C_\alpha>0$ such that for $n$ large enough we have
\begin{align}\label{talaboundres}
\mathbb{P} \Big(\sup_{x\in \mathcal{H} } |\partial^{\alpha} \wh f(x) -\mathbb{E} \partial^{\alpha} \wh f(x)| \geq C_\alpha \gamma_{n,h}^{(|\alpha|)}\Big) \leq n^{-\ell}.
\end{align}
\end{lemma}
%
%
\begin{proof}
We will apply the Talagrand's inequality (Lemma~\ref{talagrandineq}) to show this result. We fix $\alpha\in\mathbb{N}^d$ with $|\alpha|\in\{0,1,2\}$. Let
\begin{align}\label{kalpha}
\mathcal{K}_\alpha = \Big\{ y\mapsto \partial^\alpha K \Big(\frac{x-y}{h}\Big): x\in {\mathcal{H}} \Big\}.
\end{align}
Then we can write
\begin{align}\label{alphakrela}
\sup_{x\in \mathcal{H} } |\partial^{\alpha} \wh f(x) -\mathbb{E} \partial^{\alpha} \wh f(x)| = \frac{1}{\sqrt{n}h^{d+|\alpha|}}\sup_{g\in\mathcal{K}_\alpha} |\mathbb{G}_n(g)|.
\end{align}
We have
\begin{align*}
\sup_{g\in \mathcal{K}_\alpha} \text{Var}(g) \leq \sup_{x\in\mathcal{H}} \int_{\mathbb{R}^d} \Big[\partial^\alpha K \Big(\frac{x-y}{h}\Big)\Big]^2 f(y)dy \leq h^d \sup_{y\in\mathbb{R}^d}|f(y)| \int_{\mathbb{R}^d} [\partial^\alpha K(x)]^2 dx =O(h^d).
\end{align*}
It follows from assumption (K2) that $\sup_{y\in\mathcal{H}}\sup_{g\in\mathcal{K}_\alpha}|g(y)|\leq b$ and $\sup_Q N(\mathcal{K}_\alpha,e_Q,b\varepsilon)\leq (A/\varepsilon)^v$, where the supremum is taken over all finitely discrete probability measures. For any $\ell>0$, we apply Lemma~\ref{talagrandineq} with $\sigma^2=O(h^d)$, and $t=\ell\log{n}$ and then get 
\begin{align}\label{talaapplic}
\mathbb{P} \Big(\sup_{g\in\mathcal{K}_\alpha} |\mathbb{G}_n(g)| \geq C_\alpha \sqrt{h^d\log n} \Big) \leq n^{-\ell}.
\end{align}
for some constant $C_\alpha>0$. Using (\ref{alphakrela}), we obtain (\ref{talaboundres}). 
\end{proof}

The following lemma shows that $\textsf{ridge}(\wh f\,)$ is close to $\mathcal{S}$ with a high probability under our assumptions. 
\begin{lemma}
\label{ridgeinclusionlemma}
Assume that (F1), (F2), (F3), (K1), (K2) and (H) hold. For any $\ell>0$ and $n$ large enough we have with probability at least $1-n^{-\ell}$ that 
\begin{align}
\label{ridgeinclusion}
\mathcal{S} \subset \wh{\mathcal{S}}(\rho_n) \subset \mathcal{S}(2\rho_n) \subset\mathcal{S} \oplus (2\rho_n/C_L)^{1/\beta^\prime}.
\end{align}
\end{lemma}

\begin{proof}
We first consider case (a). Recall the definition of $\varepsilon_n^{(i)}, i=1,2,$ as given in (\ref{varepsilon1}) and (\ref{varepsilon2}). Let $\kappa_1=\max_{x\in\mathcal{H}}\|\nabla f(x)\|$ and $\kappa_2=\max_{x\in\mathcal{H}}\|\nabla^2 f(x)\|_F$.  Let $\psi_n=2\sqrt{2} e_0^{-1}\kappa_1\varepsilon_n^{(2)} + \varepsilon_n^{(1)}$. Using assumption (F2), we have
\begin{align}\label{whpxphdiffbound}
|\wh p(x) - p(x)|
 \leq &\|\wh L(x)\nabla \wh f(x) - L(x)\nabla f(x)\| \nonumber\\
\leq & \|\wh L(x)-L(x)\|_F\|\nabla f(x)\|  + \|\wh L(x)\|_F\|\nabla \wh f(x)-\nabla f(x)\| \nonumber\\
\leq & 2\sqrt{2}e_0^{-1}\|\nabla^2 \wh f(x)-\nabla^2 f(x)\|_F\|\nabla f(x)\| + \|\nabla \wh f(x)-\nabla f(x)\|,
\end{align}
where in the last step we use the Davis-Kahan theorem (see Theorem 7 in \citealt{von2007tutorial}). It then follows that
\begin{align}
\label{supphatbound}
\sup_{x\in\mathcal{S}} \wh p(x) \le \sup_{x\in \mathcal{H}} |p(x) - \wh p(x)| \le \psi_n.
\end{align}
By Weyl inequality (Page 57, \citealt{serre2002matrices}), we have 
\begin{align}\label{lambdar1diff}
|\wh\lambda_{r+1}(x)-\lambda_{r+1}(x)|\leq \|\nabla^2\wh f(x) - \nabla^2 f(x)\|_F.
\end{align}
Using assumption (F2), we get
\begin{align}
\label{lambdar1}
\sup_{x\in\mathcal{S}} \wh\lambda_{r+1}(x) \leq -e_0 + \sup_{x\in\mathcal{S}} |\wh \lambda_{r+1}(x) - \lambda_{r+1}(x)| \leq -e_0 + \varepsilon_n^{(2)}.
\end{align}
By Lemmas~\ref{biasbound} and \ref{talabound}, for any $\ell>0$ and $n$ large enough we have with probability at least $1-n^{-\ell}$ that $\mathcal{S} \subset \wh{\mathcal{S}}(\rho_n)$, which is the first inclusion relation in (\ref{ridgeinclusion}).

For any $\gamma\in\mathbb{R}$, let $\mathcal{J}_\gamma^-=\{x\in\mathcal{H}: \lambda_{r+1}(x)\leq -\gamma\}$ and $\mathcal{J}_\gamma^+=\{x\in\mathcal{H}: \lambda_{r+1}(x)\geq \gamma\}$. Denote $\mathcal{J}_0^o=\{x\in\mathcal{H}: \lambda_{r+1}(x)<0\}$. Let $\delta_1=\delta_0/2$, where $\delta_0$ is given in assumption (F2). Let $\mathcal{L}_0= \{x\in\mathcal{H}:d(x,\mathcal{S})\geq\delta_1\}$ and $\mathcal{Q}_0= \{x\in\mathcal{H}:d(x,\mathcal{P}_0)\geq\delta_1\}$. Since $\mathcal{L}_0\cap \mathcal{J}_0^-$ is a compact set and $p$ is a continuous function on $\mathcal{H}$, assumption (F2) implies that there exists $a_1>0$ such that
\begin{align}\label{barlsubset}
\inf_{x\in \mathcal{L}_0\cap \mathcal{J}_0^-} \; p(x) \geq a_1.
\end{align}
For any $a>0$, let $\mathcal{P}_a=\{x\in\mathcal{H}: p(x)\leq a\}$. Let $a_2=\frac{1}{2}\inf_{x\in\widebar{\mathcal{Q}}_0}p(x) >0 $. Then $\mathcal{P}_{a_2} \subset \mathcal{P}_0\oplus \delta_1$. Assumption (F2) says that $\inf_{x\in\mathcal{P}_0 \oplus \delta_0}|\lambda_{r+1}(x)|\geq e_0$, which implies that for all $a\in[0,a_2]$, $\mathcal{P}_a=[\mathcal{P}_a\cap \mathcal{J}_{e_0}^-]\cup[\mathcal{P}_\rho\cap \mathcal{J}_{e_0}^+]$. Take $a_0=\min(\frac{1}{2}a_1,a_2)$. Then (\ref{barlsubset}) implies that $\mathcal{P}_{a_0}\subset (\mathcal{L}_0\cap \mathcal{J}_0^-)^\complement$ and hence for any $|\gamma|<e_0$,
\begin{align}
\label{Pa0}
(\mathcal{P}_{a_0}\cap \mathcal{J}_\gamma^-)=(\mathcal{P}_{a_0}\cap \mathcal{J}_{e_0}^-)\subset (\mathcal{S}\oplus\delta_0 ).
\end{align}

Let $E$ be the event that $\psi_n + \rho_n  \le a_0$ and $\varepsilon_n^{(2)} < e_0/2$. By Lemmas~\ref{biasbound} and \ref{talabound}, under the assumption (H), for any $\ell>0$ and $n$ large enough we have that $E$ occurs with probability at least $1-n^{-\ell}$. The rest of the proof is conditional on $E$. By (\ref{supphatbound}), we have that
\begin{align}
\label{pmax}
\sup_{x\in\wh{\mathcal{S}}(\rho_n)} p(x) \leq \sup_{x\in \mathcal{H}} |p(x) - \wh p(x)| + \rho_n \leq \psi_n + \rho_n.
\end{align}
Hence $\wh{\mathcal{S}}(\rho_n) \subset \mathcal{P}_{a_0}$ on event $E$. Note that by (\ref{lambdar1diff})
\begin{align}
\label{lambdahatbound}
\sup_{x\in\wh{\mathcal{S}}(\rho_n)} \lambda_{r+1}(x) \leq \sup_{x\in\wh{\mathcal{S}}(\rho_n)} \wh \lambda_{r+1}(x) + \sup_{x\in\wh{\mathcal{S}}(\rho_n)} |\wh \lambda_{r+1}(x) - \lambda_{r+1}(x)| \leq \varepsilon_n^{(2)}.
\end{align}
Then by \eqref{Pa0},
\begin{align*}
\wh{\mathcal{S}}(\rho_n) \subset (\mathcal{P}_{\rho_0}\cap \mathcal{J}_{-e_0/2}^-) = (\mathcal{P}_{\rho_0}\cap \mathcal{J}_{e_0}^-) \subset (\mathcal{S}\oplus\delta_0 ).
\end{align*}
Hence by \eqref{pmax}, on event $E$,
\begin{align}
\label{Shatinclusion}
\wh{\mathcal{S}}(\rho_n) \subset \mathcal{S}(\psi_n+\rho_n).
\end{align}
%
%
Note that assumption (H), for any $\ell>0$ and $n$ large enough we have with probability at least $1-n^{-\ell}$ that $\psi_n \le \rho_n$ and hence $\mathcal{S}(\psi_n+\rho_n) \subset \mathcal{S}(2\rho_n)$. Using assumption (F3), we immediately get $\mathcal{S}(2\rho_n) \subset \mathcal{S} \oplus (2\rho_n/C_L)^{1/\beta^\prime}$ and consequently (\ref{ridgeinclusion}) for case (a).

Next we consider case (b), for which we use similar arguments as for case (a). Using (\ref{whpxphdiffbound}) again we have that 
\begin{align}
\sup_{x\in\mathcal{S}} \wh p(x) \le \varepsilon_n^{(1)}.
\end{align}
Combining this with \eqref{lambdar1} we get $\mathcal{S} \subset \wh{\mathcal{S}}(\rho_n)$ with probability at least $1-n^{-\ell}$ for $n$ large enough. 
Let $\phi_{n,0}=\psi_n + \rho_n$, $c_0=(2/C_L)^{1/\beta^\prime}$ and sequentially define $\phi_{n,j}=\varepsilon_n^{(1)} + 2\sqrt{2}e_0^{-1}c_0\kappa_2\varepsilon_n^{(2)}\phi_{n,j-1}^{1/\beta^\prime} + \rho_n$, for $j=1,2,\cdots.$ Suppose that $n$ is large enough that $(2\rho_n)^{1/\beta^\prime} \ge \kappa_1/\kappa_2$, which makes $\phi_{n,j}$ a decreasing sequence. As shown in \eqref{Shatinclusion}, on event $E$ we still have $\sup_{x\in\wh{\mathcal{S}}(\rho_n)}d(x,\mathcal{S})\leq c_0\phi_{n,0}^{1/\beta^\prime}$ using assumption (F3). By using Taylor expansion we get that
\[\sup_{x\in\mathcal{S} \oplus (c_0\phi_{n,0}^{1/\beta^\prime})} \|\nabla f(x)\| \leq \kappa_2 (c_0\phi_{n,0}^{1/\beta^\prime}).\]
Then by (\ref{whpxphdiffbound}), for each $x\in\mathcal{S} \oplus (c_0\phi_{n,0}^{1/\beta^\prime})$, we have
\begin{align*}
|\wh p(x) - p(x)|\le \varepsilon_n^{(1)} + 2\sqrt{2}e_0^{-1}\varepsilon_n^{(2)}\|\nabla f(x)\| \le \varepsilon_n^{(1)} +2\sqrt{2}e_0^{-1}c_0\kappa_2\varepsilon_n^{(2)}\phi_{n,0}^{1/\beta^\prime}.
%
%
%
\end{align*}
Therefore 
\[\sup_{x\in\wh{\mathcal{S}}(\rho_n)} p(x) \le \sup_{x\in\mathcal{S} \oplus (c_0\phi_{n,0}^{1/\beta^\prime})} |\wh p(x) - p(x)| + \rho_n \le \phi_{n,1}.\]
By \eqref{lambdahatbound}, we have that on event $E$, $\wh{\mathcal{S}}(\rho_n) \subset \mathcal{S}(\phi_{n,1}) \subset \mathcal{S} \oplus (c_0\phi_{n,1}^{1/\beta^\prime})$. Inductively, we get that for $j=2,3,\cdots,$
\begin{align}
\wh{\mathcal{S}}(\rho_n) \subset \mathcal{S}(\phi_{n,j}) \subset \mathcal{S} \oplus (c_0\phi_{n,j}^{1/\beta^\prime}).
\end{align}
Under assumption (H), for any $\ell>0$ and $n$ and $j$ large enough we have with probability at least $1-n^{-\ell}$ that $\phi_{n,j} \le 2\rho_n$, by applying Lemmas~\ref{biasbound} and \ref{talabound}. Immediately, we obtain (\ref{ridgeinclusion}) for case (b). 
\end{proof}

\subsection{Proofs for Section \ref{gaussapprox}}
{\bf Proof of Proposition~\ref{ridgeapptheorem}}

\begin{proof}
We first consider case (a). Note that $\sup_{x\in\mathcal{S}}\wh p(x) = \sup_{x\in\mathcal{S}} \|\wh L(x) \nabla \wh f(x) - L(x) \nabla f(x)\|$, because $L(x) \nabla f(x)=0$ for $x\in\mathcal{S}$. We have the following telescoping. 
\begin{align}\label{overalldecomp}
& \wh L(x) \nabla \wh f(x) - L(x) \nabla f(x)  \nonumber\\
= & [ \wh L(x) - L(x)] \nabla f(x) + L(x)[\nabla \wh f(x)- \nabla f(x) ] + [ \wh L(x) - L(x)] [\nabla \wh f(x)- \nabla f(x) ].
\end{align}
We first focus on the first term on the right-hand side, in particular, the difference $\wh L(x) - L(x)$. 

We define the following matrix-valued eigenprojection mapping $Q:S\mathbb{R}^{d\times d}\mapsto S\mathbb{R}^{d\times d}$. For any $\Sigma\in S\mathbb{R}^{d\times d}$, supposing it has eigenvalues $\gamma_1\geq\cdots\geq\gamma_d$ with associated eigenvectors $u_1,\cdots,u_d$, we define $Q(\Sigma) =\sum_{k=r+1}^{d}u_ku_k^\top $. Note that with this function we can write $L(x)=Q(\nabla^2 f(x))$ and $\wh L(x)=Q(\nabla^2 \wh f(x))$. It is known that $Q$ is an analytic matrix-valued function at $\Sigma$ if $\gamma_{r}>\gamma_{r+1}$. For such $\Sigma$, we define the following G\^{a}teaux derivatives: for any $D\in S\mathbb{R}^{d\times d}$, and for $k=1,2,\cdots,$ let
\begin{align*}
Q^{(k)}(\Sigma,D) = \frac{d^k}{d \tau^k} Q(\Sigma + \tau D) \Big |_{\tau=0} .
\end{align*}

Using Weyl's inequality (Page 57, Serre, 2002) and assumption (F2) we have
\begin{align}\label{weylgap}
& \inf_{x\in\mathcal{S}\oplus\delta_0}  [\wh\lambda_{r}(x) - \wh\lambda_{r+1}(x)] \nonumber\\
\geq & \inf_{x\in\mathcal{S}\oplus\delta_0}  [\lambda_{r}(x) - \lambda_{r+1}(x)] - \sum_{k=r}^{r+1}\sup_{x\in\mathcal{S}\oplus\delta_0} |\wh \lambda_{k}(x) - \lambda_{k}(x)| \nonumber\\
\geq & \inf_{x\in\mathcal{S}\oplus\delta_0}  [\lambda_{r}(x) - \lambda_{r+1}(x)] - 2\sup_{x\in\mathcal{S}\oplus\delta_0}  \|\nabla^2 \wh f(x) - \nabla^2 f(x)\|_{\text{op}} \nonumber\\
\geq &e_0 - 2\varepsilon_n^{(2)}.
\end{align}
In what follows we suppose that $\varepsilon_n^{(2)}$ is small enough such that $e_0 - 2\varepsilon_n^{(2)}>0$. Denote $\widebar D_{n,2}(x) = \nabla^2 \wh f(x)-\nabla^2 f(x)$ and $\wt D_{n,2}(x)=d^2\wh f(x)-d^2 f(x)$, which are related through $\mathbf{D}\textsf{vec}[\widebar D_{n,2}(x)]=\wt D_{n,2}(x)$. Let $H_n(x,t)=\nabla^2 f(x)+t  \widebar D_{n,2}(x)$ for $t\in[0,1]$. Then we have the following Taylor expansion using G\^{a}teaux derivatives. 
\begin{align}\label{LLhx}
 \wh L(x) - L(x) = Q^{(1)}(\nabla^2 f(x),\widebar D_{n,2}(x)) + R_n(x),
\end{align}
where 
\begin{align}\label{rnxexpress}
R_n(x) = \int_0^1 (1-t) Q^{(2)}(H_n(x,t),\widebar D_{n,2}(x)) dt.
\end{align}
We will find out the explicit expressions for the derivatives on the right-hand side of (\ref{LLhx}) and a uniform bound for $R_n$.

With slight abuse of notation, for $x\in\mathcal{S}$ and $t \in[0,1]$, suppose that $H_n(x,t)$ has $q=q(x,t)$ distinct eigenvalues: $\mu_1(x,t)>\cdots>\mu_q(x,t)$, each with multiplicities $\ell_1,\cdots,\ell_q$, which all depend on $x$ and $t$. Let $s_0=0$, $s_1=\ell_1,\cdots,s_q=\ell_1+\cdots+\ell_q$. Then $\mu_j(x,t)=\lambda_{s_{j-1}+1}(x,t)=\cdots=\lambda_{s_j}(x,t)$. Let $E_j(x,t)$ be a $d\times \ell_j$ matrix with orthonormal eigenvectors $v_{s_{j-1}+1}(x,t),\cdots,v_{s_j}(x,t)$ as its columns, and $P_j(x,t) = E_j(x,t)E_j(x,t)^\top ,$ $j=1,\cdots,q$. Let $\nu_{jk}(x,t)=[\mu_j(x,t) - \mu_k(x,t)]^{-1} $ and
\begin{align}\label{sjxtexpress}
S_j(x,t) = \sum_{\substack{k=1 \\ k\neq j}}^q \nu_{jk}(x,t)P_k(x,t).
\end{align}
When $t=0$, it is easy to see that $\mu_j(x,t)=\mu(x)$, $P_j(x,t)=P_j(x)$, $\nu_{jk}(x,t)=\nu_{jk}(x)$ and $S_j(x,t)=S_j(x)$. For any $\Sigma\in S\mathbb{R}^{d\times d}$, supposing it has eigenvalues $\gamma_1\geq\cdots\geq\gamma_d$ with associated eigenvectors $u_1,\cdots,u_d$, we define 
\begin{align*}
Q_{j,x,t}(\Sigma) = \sum_{k=s_{j-1}+1}^{s_j}u_ku_k^\top ,
\end{align*}
which is an analytic matrix-valued function at $\Sigma$ if $\gamma_{s_{j-1}}>\gamma_{s_{j-1}+1}$ and $\gamma_{s_j}>\gamma_{s_j+1}$, where we take $\gamma_0=+\infty$ and $\gamma_{d+1}=-\infty$. Here the dependence of $Q_{j,x,t}$ on $x$ and $t$ is through $s_{j-1}$ and $s_j$. Then we can write $P_j(x,t)=Q_{j,x,t}(H_n(x,t))$. For any $D\in S\mathbb{R}^{d\times d}$, and for $k=1,2,\cdots,$ let 
\begin{align*}
Q_{j,x,t}^{(k)}(\Sigma,D) = \frac{d^k}{d \tau^k} Q_{j,x,t}(\Sigma + \tau D) \Big |_{\tau=0},\; j=1,\cdots,q .
\end{align*}
The following expressions can be obtained using classical matrix perturbation theory (see page 77, Kato, 1976). For any $x\in\mathcal{S}$, $t\in[0,1]$, $D\in S\mathbb{R}^{d\times d}$, and $j=1,\cdots,q$, let
\begin{align}
&Q_{j,x,t}^{(1)}(H_n(x,t),D) = P_j(x,t)DS_j(x,t) + S_j(x,t)DP_j(x,t), \label{qj1}\\
&Q_{j,x,t}^{(2)}(H_n(x,t),D) \nonumber\\
= &2\{P_j(x,t)DS_j(x,t)DS_j(x,t) + S_j(x,t)DP_j(x,t)DS_j(x,t) + S_j(x,t)DS_j(x,t)DP_j(x,t) \nonumber\\
&  \hspace{1cm}- P_j(x,t)DP_j(x,t)D[S_j(x,t)]^2 - P_j(x,t)D[S_j(x,t)]^2DP_j(x,t) \nonumber\\
& \hspace{1cm} - [S_j(x,t)]^2DP_j(x,t)DP_j(x,t)\}. \label{qj2}
\end{align}
Denote $\mathcal{I}=\{1,\cdots,q\}$, $\mathcal{I}_\leq=\{1,\cdots,q_r\}$ and $\mathcal{I}_>=\{q_r+1,\cdots,q\}$ so that $\mathcal{I}=\mathcal{I}_\leq\cup \mathcal{I}_>$. Using (\ref{qj1}) and (\ref{sjxexpress}) we can write
\begin{align}\label{q1sum}
Q^{(1)}(\nabla^2 f(x),D) = & \sum_{j\in\mathcal{I}_>} Q_{j,x,0}^{(1)}(\nabla^2 f(x),D) \nonumber\\
= & \sum_{j\in\mathcal{I}_>} \; \sum_{k\in\mathcal{I}\backslash\{j\}}  \Big\{\nu_{jk}(x) [P_j(x)DP_k(x) + P_k(x)DP_j(x) ] \Big\} \nonumber\\
= & \sum_{j\in\mathcal{I}_>} \; \sum_{k\in\mathcal{I}_\leq} \Big\{\nu_{jk}(x)  [P_j(x)DP_k(x) + P_k(x)DP_j(x) ] \Big\}, 
\end{align}
where the last step holds because (noticing that $\nu_{jk}(x)=-\nu_{kj}(x)$) $$ \sum_{j\in\mathcal{I}_>} \; \sum_{k\in\mathcal{I}_>\backslash\{j\} } \Big\{\nu_{jk}(x)  [P_j(x)DP_k(x) + P_k(x)DP_j(x) ] \Big\}=0.$$
Next we consider 
\begin{align}
Q^{(2)}(H_n(x,t),D) = \sum_{j\in\mathcal{I}_>} Q_{j,x,t}^{(2)}(H_n(x,t),D).
\end{align}
Notice that $P_i(x,t)$, $i=1,\cdots,q$ have the following properties: (i) $P_i(x,t)P_j(x,t)=0$ if $i\neq j$; (ii) $[P_i(x,t)]^2=P_i(x,t)$, $i=1,\cdots,q$; and (iii) $\sum_{i=1}^qP_i(x,t)=\mathbf{I}_d$. Hence by (\ref{sjxtexpress}) we get
\begin{align}\label{sjx2}
[S_j(x,t)]^2 = \sum_{k\in \mathcal{I}\backslash\{ j\} } [\nu_{jk}(x,t)]^2 P_k(x,t).
\end{align}
For $j,k,\ell\in \mathcal{I}$, we denote $\Pi_{jk\ell}=\Pi_{jk\ell}(x,t;D)=P_j(x,t)DP_k(x,t)DP_\ell(x,t)$. Plugging (\ref{sjx2}) and (\ref{sjxtexpress}) into (\ref{qj2}) we can write for $j\in\mathcal{I}_>$,
\begin{align*}
&\frac{1}{2}Q_{j,x,t}^{(2)}(H_n(x,t),D) \\
= &\mathop{\sum\sum}\limits_{k,\ell\in\mathcal{I}\backslash\{ j\}} \nu_{jk}(x,t)\nu_{j\ell}(x,t)(\Pi_{jk\ell} + \Pi_{kj\ell} + \Pi_{k\ell j})   - \sum_{k\in\mathcal{I}\backslash\{ j\}} [\nu_{jk}(x,t)]^2(\Pi_{jjk} + \Pi_{jkj} + \Pi_{kj j}) ,
\end{align*}
and hence
\begin{align}\label{q2express}
& \frac{1}{2}Q^{(2)}(H_n(x,t),D) \nonumber\\
= &\sum_{j\in\mathcal{I}_>} \mathop{\sum\sum}\limits_{k,\ell\in\mathcal{I}\backslash\{ j\}} \nu_{jk}(x,t)\nu_{j\ell}(x,t)(\Pi_{jk\ell} + \Pi_{kj\ell} + \Pi_{k\ell j}) \nonumber\\
&\hspace{1cm} - \sum_{j\in\mathcal{I}_>} \sum_{k\in\mathcal{I}\backslash\{ j\}} [\nu_{jk}(x,t)]^2(\Pi_{jjk} + \Pi_{jkj} + \Pi_{kj j}) \nonumber\\
=& \sum_{j\in\mathcal{I}_>} \mathop{\sum\sum}\limits_{k,\ell\in\mathcal{I}_\leq} \nu_{jk}(x,t)\nu_{j\ell}(x,t)(\Pi_{jk\ell} + \Pi_{kj\ell} + \Pi_{k\ell j}) \nonumber\\
&\hspace{1cm} - \sum_{j\in\mathcal{I}_>} \sum_{k\in\mathcal{I}_\leq}[\nu_{jk}(x,t)]^2(\Pi_{jjk} + \Pi_{jkj} + \Pi_{kj j}),
%
\end{align}
where the last equality holds because $\widebar\nu_{jk\ell}(x,t):=\nu_{jk}(x,t)\nu_{j\ell}(x,t)+ \nu_{kj}(x,t)\nu_{k\ell}(x,t) + \nu_{\ell j}(x,t)\nu_{\ell k}(x,t)\equiv0$ for all distinct $j,k,\ell$ so that
\begin{align*}
& \sum_{j\in\mathcal{I}_>} \; \mathop{\sum\sum}\limits_{k,\ell\in\mathcal{I}_>\backslash\{ j\},k\neq\ell}\nu_{jk}(x,t)\nu_{j\ell}(x,t)(\Pi_{jk\ell} + \Pi_{kj\ell} + \Pi_{k\ell j}) \\
= & \frac{1}{3} \mathop{\sum\sum\sum}\limits_{j,k,\ell\in\mathcal{I}_>,j\neq k,k\neq\ell,j\neq \ell} \widebar\nu_{jk\ell}(x,t) (\Pi_{jk\ell} + \Pi_{kj\ell} + \Pi_{k\ell j}) \\
%
%
=&0.
\end{align*}
%
For $k\in\mathcal{I}$, let $\mathcal{U}_k(x,t)$ be the subspace spanned by the column vectors of $E_k(x,t)$. Since $P_k(x,t)a=a$ if $a\in \mathcal{U}_k(x,t)$ and $P_k(x,t)a=0$ if $a\in \mathcal{U}_\ell(x,t)$ for $\ell\neq k$, it is clear that $\|P_k(x,t)\|_{\text{op}}=1$. Hence $\|P_k(x,t)\|_F\leq \sqrt{d}\|P_k(x,t)\|_{\text{op}}=\sqrt{d}$. For $j,k,\ell\in \mathcal{I}$, notice that 
\begin{align}\label{pijkellf}
\|\Pi_{jk\ell}(x,t)\|_F\leq \|D\|_F^2\|P_j(x,t)\|_F\|P_k(x,t)\|_F\|P_\ell(x,t)\|_F\leq d^{3/2} \|D\|_F^2 .
\end{align}
Hence by (\ref{q2express}),
\begin{align}\label{q2bound}
\frac{1}{2}\|Q^{(2)}(H_n(x,t),D)\|_F \leq 3 d^{3/2} [q_r^2 (q-q_r) + q_r(q-q_r)]\frac{1}{[e_0-2\varepsilon_n^{(2)}]^2}\|D\|_F^2.
\end{align}
%
%
When $\max\{\varepsilon_n^{(0)},\varepsilon_n^{(1)},\varepsilon_n^{(2)}\}$ is small enough, using (\ref{rnxexpress}) and (\ref{q2bound}), we get
\begin{align}\label{rnxbound}
\sup_{x\in\mathcal{S}} \|R_n(x) \|_F \lesssim \sup_{x\in\mathcal{S}} \|\nabla^2 \wh f(x)-\nabla^2 f(x)\|_F^2 \lesssim (\varepsilon_n^{(2)})^2.
\end{align}
%
%
%
%
%
%
%
%
Similarly using (\ref{q1sum}), we can show that $\sup_{x\in\mathcal{S}}\|Q^{(1)}(\nabla^2 f(x),\widebar D_{n,2}(x))\|_F\lesssim \varepsilon_n^{(2)}$. It then follows from (\ref{LLhx}) and (\ref{rnxbound}) that when $\varepsilon_n^{(2)}$ is small enough we have
\begin{align}\label{whlxdiffbound}
\|\wh L(x) - L(x)\|_F \lesssim \varepsilon_n^{(2)}.
\end{align}
With the expansion in (\ref{LLhx}), next we study $[\wh L(x) - L(x)]\nabla f(x)$, the first term on the right-hand side of (\ref{overalldecomp}). Notice that
\begin{align}\label{qj1nablafx}
& \Big[ \sum_{j=q_r+1}^q Q_{j}^{(1)}(\nabla^2 f(x),\widebar D_{n,2}(x)) \Big] \nabla f(x) \nonumber\\
= & \sum_{j=q_r+1}^q \Big[  P_j(x)\widebar D_{n,2}(x) S_j(x) + S_j(x)\widebar D_{n,2}(x) P_j(x)\Big] \nabla f(x) \nonumber\\
%
%
\stackrel{(*)}{=}& \sum_{j=q_r+1}^q \{  P_j(x) \otimes [S_j(x) \nabla f(x)]^\top  + S_j(x) \otimes [P_j(x) \nabla f(x)]^\top  \} \textsf{vec} [\widebar D_{n,2}(x)] \nonumber\\
= & M(x)^\top  \wt D_{n,2}(x).
\end{align}
Here step (*) holds because of the following property of Kronecker products: for matrices $A,B$ and a column vector $v$ which are compatible, we have $ABv=(A\otimes v^\top )\textsf{vec} B^\top $ (see page 30 of \citealt{magnus2007matrix}).

%
Note that for $x\in\mathcal{S}$, $M(x)$ has simpler expression given in (\ref{MhTexpress}). Since $M(x)$ and $L(x)$ are uniformly bounded on $\mathcal{S}$ under assumptions (F1) and (F2), from (\ref{overalldecomp}), (\ref{whlxdiffbound}) and (\ref{qj1nablafx}) we obtain that when $\max\{\varepsilon_n^{(0)},\varepsilon_n^{(1)},\varepsilon_n^{(2)}\}$ is small enough,
\begin{align*}
\sup_{x\in\mathcal{S}} \| [\wh L(x) \nabla \wh f(x) - L(x) \nabla f(x) ] -  M(x)^\top \wt D_{n,2}(x) \| \lesssim \varepsilon_n^{(1)} + (\varepsilon_n^{(2)} )^2.
\end{align*}
Also note that $\sup_{x\in\mathcal{S}}\|\wt D_{n,2}(x) - D_{n,2}(x)\| = O(h^\beta)$ by Lemma~\ref{biasbound}. 
%
%
%
We then obtain 
\begin{align*}
|\sup_{x\in\mathcal{S}}\wh p(x) - \sup_{x\in\mathcal{S}} \|M(x)^\top D_{n,2}(x)\|| 
\leq &\sup_{x\in\mathcal{S}} |\wh p(x) - \|M(x)^\top D_{n,2}(x)\| | \\
\leq & \sup_{x\in\mathcal{S}} \| [\wh L(x) \nabla \wh f(x) - L(x) \nabla f(x) ] -  M(x)^\top D_{n,2}(x) \| \\
\lesssim & h^\beta +\varepsilon_n^{(1)} + (\varepsilon_n^{(2)} )^2.
\end{align*}
This is case (a). 
%
%
%
%
%
Next we consider the case (b). When $\nabla f(x)=0$, for all $x\in\mathcal{S}$, 
%
%
from (\ref{overalldecomp}) and (\ref{whlxdiffbound}) we have
\begin{align*}
\sup_{x\in\mathcal{S}} \| [\wh L(x) \nabla \wh f(x) - L(x) \nabla f(x) ] -  L(x)\wt D_{n,1}(x) \| \lesssim \varepsilon_n^{(1)} \varepsilon_n^{(2)} , 
%
\end{align*}
where $\wt D_{n,1}(x)=\nabla \wh f(x)- \nabla f(x)$. Since $\sup_{x\in\mathcal{S}}\|\wt D_{n,1}(x) - D_{n,1}(x)\| = O(h^\beta)$ by Lemma~\ref{biasbound}, we then have
\begin{align*}
|\sup_{x\in\mathcal{S}}\wh p(x) - \sup_{x\in\mathcal{S}} \|L(x)^\top D_{n,1}(x)\|| 
\leq &\sup_{x\in\mathcal{S}} |\wh p(x) - \|L(x)^\top D_{n,1}(x)\| | \\
\leq & \sup_{x\in\mathcal{S}} \| [\wh L(x) \nabla \wh f(x) - L(x) \nabla f(x) ] -  L(x)^\top D_{n,1}(x) \| \\
\lesssim & h^\beta + \varepsilon_n^{(1)} \varepsilon_n^{(2)}.
\end{align*}
This is case (b) and the proof is completed.
%
%
%
%
\end{proof}

{\bf Proof of Corollary~\ref{gaussianappcorollary}}
\begin{proof}

Lemmas~\ref{biasbound} and \ref{talabound} imply that for all $\ell>0$, there exist constants $C_k$ for $k=0,1,2$ such that
\begin{align}\label{originalepb}
\mathbb{P} ( \varepsilon_n^{(k)} \geq C_k (\gamma_{n,h}^{(k)}+h^\beta)) \leq n^{-\ell}.
\end{align}
We then obtain (\ref{firstapproximation}) by using (\ref{ridgenessapprox}). 
\end{proof}

{\bf Proof of Proposition~\ref{2ndapproxtheorem}}
\begin{proof}

Define the class of functions
\begin{align}\label{FGrelation}
\mathcal{G} = \{h^{d/2}g: g\in\mathcal{F} \}.
\end{align}
%
It is easy to verify that $\mathcal{G}$ is VC type under assumption (F1) and (K1). We will apply Theorem 2.1 in \cite{chernozhukov2016empirical} to the empirical process $\mathbb{G}_n$ indexed by $\mathcal{G}$. To this end, we need to verify that there exists a constant $\sigma>0$ such that $\sup_{g\in \mathcal{G}} \mathbb{E}|g(X)|^k \leq \sigma^2 b^{k-2}$ for $k=2,3,4$. For any $x\in\mathcal{S}$ and $z\in\mathbb{S}_{d-1}$ we have 
\begin{align*}
\mathbb{E}|h^{d/2}g_{x,z}(X)|^k = 
\begin{cases}
\int_{\mathbb{R}^d} |z^\top  M(x)^\top  d^2K(\frac{x-y}{h})|^k f(y)dy & \text{ for case (a)}\\
\int_{\mathbb{R}^d} |z^\top  L(x)^\top  \nabla K(\frac{x-y}{h})|^k f(y)dy & \text{ for case (b)}
\end{cases}.
\end{align*}
Using the change of variable $u=(x-y)/h$ to evaluate the above integrals, we get that $\sup_{g\in \mathcal{G}} \mathbb{E}|g(X)|^k =O(h^d )$ for $k=2,3,4$. Note that $h^{d/2}Z=\sup_{g\in \mathcal{G}} \mathbb{G}_n(g)$. Using Theorem 2.1 in \cite{chernozhukov2016empirical}, where we take $b=O(1)$, $\sigma=O(h^{d/2})$, $K_n=O(\log n)$, and have that for every $\gamma\in(0,1)$ and $q$ sufficiently large, there exists a random variable $\wt Z \stackrel{d}{=} \sup_{g\in\mathcal{F}} G_P(g)$ such that
\begin{align*}
\mathbb{P} (|h^{d/2}Z- h^{d/2} \wt Z| >C_1 \kappa_{n} ) \leq C_2 (\gamma+n^{-1}),
\end{align*}
for some constants $C_1,C_2>0$, where $\kappa_{n} =\gamma^{-1/3}(\log n)^{1/2}(\gamma_{n,h}^{(0)})^{1/3} + (\gamma/n)^{-1/q}(\log n)^{1/2}\gamma_{n,h}^{(0)}$. Combining this result with (\ref{firstapproximation}) we then have
\begin{align*}
\mathbb{P} \Big(  \Big| \omega_n \sup_{x\in\mathcal{S}}\wh p(x) - \wt Z\Big| > \wt C_1 \wt\delta_{n} \Big)  \leq \wt C_2 (\gamma+n^{-1}),
\end{align*}
for some constants $\wt C_1, \wt C_2>0$, where $\wt\delta_{n} = \delta_n + h^{-d/2}\kappa_{n}$. Choosing $\gamma=(\log n)^{9/8} (\gamma_{n,h}^{(0)})^{1/4}$ and $q$ sufficiently large, we then get (\ref{secondapproximation}).

%
%
\end{proof}

{\bf Proof of Lemma~\ref{newanticoncentration}}
\begin{proof}
We use the ideas in the proofs of Lemma 2.2 in \cite{chernozhukov2016empirical} and Theorem 3.2 in \cite{belloni2018high}. For any $\delta>0$, let $\{t_1,\cdots,t_N\}$ be a $\delta$-net of $(T,d)$ with $N=N(T,d,\delta)$. Let $\zeta=\sup_{(s,t)\in T_\delta}|W(t)-W(s)|$. Since $|\sup_{t\in T} W(t) - \max_{1\leq j\leq N} W(t_j)|\leq \zeta$, we have that for every $x\in\mathbb{R}$ and $\epsilon,\kappa^\prime>0$,
\begin{align}\label{concentrationbound}
\mathbb{P} \Big(\Big|\sup_{t\in T} W(t) - x\Big|\leq\epsilon\Big) \leq \mathbb{P} \Big(\Big|\max_{1\leq j\leq N} W(t_j) -x\Big|\leq\epsilon+\kappa^\prime\Big) + \mathbb{P}(\zeta>\kappa^\prime).
\end{align}
We take $\epsilon<\bar\sigma_0/16$ and $\kappa^\prime=\mathbb{E}(\zeta)+\kappa\delta\leq \bar\sigma_0/16$, so that $\epsilon+\kappa^\prime<\bar\sigma_0/8$. By Theorem 3.2 in \cite{belloni2018high}, we have
\begin{align}\label{concentrationbound1}
\mathbb{P} \Big(\Big|\max_{1\leq j\leq N} W(t_j) -x\Big|\leq\epsilon+\kappa^\prime\Big) \leq 2(1/\bar\sigma) (\epsilon + \phi(\delta) + \kappa\delta) \sqrt{\log N} (\sqrt{2\log N}+3).
\end{align}
Also by Borell's inequality (Proposition A.2.1 in \citealt{vaart1996weak}), for $\kappa>0$, we get
\begin{align}\label{concentrationbound2}
\mathbb{P}(\zeta>\kappa^\prime) \leq e^{-\kappa^2/2}.
\end{align}
Plugging (\ref{concentrationbound1}) and (\ref{concentrationbound2}) into (\ref{concentrationbound}), we then obtain the desired result in this lemma. 
\end{proof}

 {\bf Proof of Lemma~\ref{sigmah2bounds}}
 \begin{proof}
 
 We first consider case (a). For $g\in \mathcal{F}$, we denote $g(y)=h^{-d/2}z^\top  M(x)^\top  d^2 K(\frac{x-y}{h})$ for some $x\in\mathcal{S}$ and $z\in\mathbb{S}_{d-1}$. By the definition of $G_P(g)$, we have 
\begin{align*}
\text{Var}(G_P(g)) = \text{Var}\Big(h^{-d/2}z^\top  M(x)^\top  d^2K\Big(\frac{x-X}{h}\Big)\Big).
\end{align*}
Notice that using the expression of $M$ in (\ref{MhTexpress}) we have
\begin{align}\label{ztmxt}
z^\top M(x)^\top  =  \mathop{\sum\sum}\limits_{q\geq j\geq q_r+1>k\geq 1} \nu_{jk}(x) \{z^\top P_j(x)  \otimes [P_k(x) \nabla f(x)]^\top \} \mathbf{D} ,
\end{align}
where we have used the following property of Kronecker products of matrices: $(A\otimes B)(C\otimes D)=(A\otimes C)(B\otimes D)$ for any compatible matrices $A,B,C$ and $D$, and we have treated $z^\top $ as $z^\top \otimes 1.$ Recall that $\mathcal{V}^\perp(x)$ is the subspace spanned by $v_{1}(x),\cdots,v_{r}(x)$. When $z\in \mathbb{S}_{d-1}\cap \mathcal{V}^\perp(x)$, we have $z^\top M(x)^\top  =0$ and hence $\text{Var}(G_P(g))=0$, which implies $\ubar\sigma^2=0$.  

Next we show (\ref{barsignabound}). 
%
%
%
We compute $\text{Var}(G_P(g))$ below.
\begin{align}\label{variancecasea}
& \text{Var}(G_P(g)) \nonumber\\
= & \text{Var}\Big(h^{-d/2}z^\top  M(x)^\top  d^2K\Big(\frac{x-X}{h}\Big)\Big) \nonumber\\
= & h^{-d} \int_{\mathbb{R}^d} \Big[z^\top M(x)^\top  d^2K\Big(\frac{x-y}{h}\Big)\Big]^2  f(y) dy - h^{-d} \Big[\int_{\mathbb{R}^d} z^\top M(x)^\top  d^2K\Big(\frac{x-y}{h}\Big) f(y) dy \Big]^2 \nonumber\\
=& z^\top M(x)^\top  \int_{\mathbb{R}^d} d^2K(u) d^2K(u)^\top f(x-hu) du M(x) z \nonumber\\
& \hspace{2cm} - h^d \Big[ z^\top M(x)^\top  \int_{\mathbb{R}^d}d^2K(u) f(x-hu) du \Big]^2 \nonumber\\
=& f(x) z^\top M(x)^\top  \mathbf{R} M(x) z + O(h^2) ,
\end{align}
where $\mathbf{R} = \int_{\mathbb{R}^d} d^2K(u) [d^2K(u)]^\top  du$, and the $O$-term is uniform in $g\in\mathcal{S}$. In the above calculation we have used the change of variable $u=(x-y)/h$ and a Taylor expansion for $f(x-hu)$.
%
Note that $\mathbf{R}$ is a positive definite matrix under assumptions (K1) and (K2) because for any $a\in\mathbb{R}^{d(d+1)/2}$, $a^\top \mathbf{R}a\geq \int_{\mathcal{B}}(a^\top d^2 K(x))dx>0$, where $\mathcal{B}$ is given in assumption (K2). Recall that $\mathcal{V}(x)$ is the subspace spanned by $v_{r+1}(x),\cdots,v_{d}(x)$. 
%
Note that $\mathbb{S}_{d-1}\cap \mathcal{V}(x) = \{V(x)a:\; a\in \mathbb{S}_{d-r-1}\}$. In other words, for any $z\in\mathbb{S}_{d-1}\cap \mathcal{V}(x)$, there always exists $a\in \mathbb{S}_{d-r-1}$ such that $z=V(x)a$. For $x\in\mathcal{S}$, denote $B(x)=\mathop{\sum\sum}\limits_{q\geq j\geq q_r+1>k\geq 1}  \nu_{jk}(x)\{P_j(x)^\top V(x)  \otimes [P_k(x) \nabla f(x)]\}$. Using (\ref{ztmxt}), we have 
\begin{align*}
z^\top  M(x)^\top  =  \mathop{\sum\sum}\limits_{q\geq j\geq q_r+1>k\geq 1} \nu_{jk}(x)\{a^\top V(x)^\top P_j(x)  \otimes [P_k(x) \nabla f(x)]^\top \} \mathbf{D} 
=  a^\top  B(x)^\top  \mathbf{D}.
\end{align*}
Note that we can write $V(x)=(E_1(x),\cdots,E_q(x))$ and $$P_j(x)^\top V(x)=E_j(x)E_j(x)^\top V(x)=(0,\cdots,0,E_j(x),0,\cdots,0),$$
where $E_j(x)$ occupies the same position as in $V(x)$. For $i=1,\cdots,d$, let $\theta_i$ be such that $\mu_{\theta_i}(x)=\lambda_i(x)$. Therefore we can write $B(x)=(b_{r+1}(x),\cdots,b_{d}(x))$ with $b_{i}(x)=v_{i}(x)\otimes \sum_{j=1}^r [\nu_{\theta_i\theta_j}(x)v_{j}(x)^\top  \nabla f(x) v_{j}(x)]$, $i=r+1,\cdots,d.$ 
Recall that $x_0\in\mathcal{S}$ and $\|\nabla f(x_0)\|>0$ as given in assumption (F4). Then we obtain
%
\begin{align}\label{quadraticbound}
\sup_{z\in\mathbb{S}_{d-1}, x\in\mathcal{S}} z^\top M(x)^\top  \mathbf{R} M(x) z 
\geq & \sup_{z\in\mathbb{S}_{d-1}\cap \mathcal{V}(x_0)} z^\top  M(x_0)^\top  \mathbf{R}M(x_0) z \nonumber\\
= &\sup_{a\in \mathbb{S}_{d-r-1}} a^\top  B(x_0)^\top  \mathbf{D} \mathbf{R} \mathbf{D}^\top  B(x_0) a \nonumber\\
\geq &\lambda_{\min}(\mathbf{R})\|\mathbf{D}^\top  B(x_0)\|_{\text{op}}^2.
\end{align}
For $x\in\mathcal{S}$, let $\wt b_{i}(x)= \sum_{j=1}^r [\nu_{ij}(x)v_{j}(x)^\top  \nabla f(x) v_{j}(x)] \otimes v_{i}(x)$, $i=r+1,\cdots,d,$ and $\wt B(x)=(\wt b_{r+1}(x),\cdots,\wt b_{d}(x))$. Then $\wt b_i(x)=K_{d^2} b_i(x)$, where $K_{d^2}$ is the $d^2\times d^2$ commutation matrix such that $K_{d^2}(s\otimes t)=t\otimes s$ for any $s,t\in\mathbb{R}^d$. It is known by \cite{magnus2007matrix} (Theorem 12, page 57) that $(\mathbf{D}^+)^\top \mathbf{D}^\top =\frac{1}{2}(\mathbf{I}_{d^2} + K_{d^2})$. Hence
\begin{align}\label{l2normbound}
\|B(x_0) + \wt B(x_0)\|_{\text{op}}^2 = \|2(\mathbf{D}^+)^\top \mathbf{D}^\top B(x_0)\|_{\text{op}}^2 \leq 4 \|(\mathbf{D}^+)^\top \|_{\text{op}}^2 \|\mathbf{D}^\top B(x_0)\|_{\text{op}}^2.
\end{align}
Notice that $\|(\mathbf{D}^+)^\top \|_{\text{op}}>0$ because $(\mathbf{D}^+)^\top $ has full rank. Also notice that $$\frac{1}{4}\|B(x_0) + \wt B(x_0)\|_{\text{op}}^2=\max_{i\in\{r+1,\cdots,d\}}\sum_{j=1}^r [\nu_{\theta_i\theta_j}(x_0)v_{j}(x_0)^\top  \nabla f(x_0)]^2\geq [\nu_{1q}(x_0)]^2\| \nabla f(x_{0})\|^2.$$  
Combining (\ref{quadraticbound}) and (\ref{l2normbound}) we get
\begin{align}
\sup_{z\in\mathbb{S}_{d-1}, x\in\mathcal{S}} z^\top M(x)^\top  \mathbf{R} M(x) z 
\geq & \frac{\lambda_{\min}(\mathbf{R})}{4 \|(\mathbf{D}^+)^\top \|_{\text{op}}^2} \|B(x_0) + \wt B(x_0)\|_{\text{op}}^2\nonumber\\
%
\geq & \frac{\lambda_{\min}(\mathbf{R})}{\|(\mathbf{D}^+)^\top \|_{\text{op}}^2}[\nu_{1q}(x_0)]^2\| \nabla f(x_{0})\|^2 
=:b_1>0. \label{varupperb}
\end{align}
Hence by (\ref{variancecasea}), $\bar\sigma^2 \geq \frac{1}{2}b_1$ when $h$ is small enough. Also notice that using (\ref{MhTexpress})
\begin{align*}
&\sup_{z\in\mathbb{S}_{d-1}, x\in\mathcal{S}} z^\top M(x)^\top  \mathbf{R} M(x) z \\
\leq &  \lambda_{\max}(\mathbf{R}) \sup_{ x\in\mathcal{S}} \|M(x)\|_{\text{op}}^2 \\
\leq & \lambda_{\max}(\mathbf{R}) \|\mathbf{D}\|_{\text{op}}^2\sup_{ x\in\mathcal{S}} \Big\|\mathop{\sum\sum}\limits_{q\geq j\geq q_r+1>k\geq 1} \nu_{jk}(x) \{  P_j(x) \otimes [P_k(x) \nabla f(x)] \}\Big\|_{\text{op}}^2\\
\leq & \frac{1}{e_0^2}  \lambda_{\max}(\mathbf{R}) \|\mathbf{D}\|_{\text{op}}^2\sup_{ x\in\mathcal{S}} \Big\|\mathop{\sum\sum}\limits_{q\geq j\geq q_r+1>k\geq 1} \{  P_j(x) \otimes [P_k(x) \nabla f(x)] \}\Big\|_{\text{op}}^2\\
= &  \frac{\sup_{ x\in\mathcal{S}}\|\nabla f(x)\|}{e_0^2}  \lambda_{\max}(\mathbf{R}) \|\mathbf{D}\|_{\text{op}}^2 =:b_2<\infty.
\end{align*}
Then $\bar\sigma^2 <\frac{3}{2}b_2<\infty$ when $h$ is small enough.

Next we consider case (b). For $g\in \mathcal{F}$, we denote $g(y)=h^{-d/2}z^\top  L(x)^\top  d^2 K(\frac{x-y}{h})$ for some $x\in\mathcal{S}$ and $z\in\mathbb{S}_{d-1}$. By the definition of $G_P(g)$, we have 
\begin{align*}
\text{Var}(G_P(g)) = \text{Var}\Big(h^{-d/2}z^\top  L(x)^\top  d^2K\Big(\frac{x-X}{h}\Big)\Big).
\end{align*}
When $z\in \mathbb{S}_{d-1}\cap \mathcal{V}^\perp(x)$, we have $z^\top L(x)^\top  =0$ and hence $\text{Var}(G_P(g))=0$, which implies $\ubar \sigma^2=0$. 
Next we show (\ref{barsignabound}). Similar to (\ref{variancecasea}) we have 
\begin{align}\label{wpgapprox}
 \text{Var}(G_P(g)) = z^\top L(x)^\top  \mathbf{S} L(x) z + O(h^2),
\end{align}
where $\mathbf{S} = \int_{\mathbb{R}^d} \nabla K(x) [\nabla K(x)]^\top dx$, and the $O$-term is uniform in $g\in\mathcal{S}$. Note that $\mathbf{S}$ is a positive definite matrix under assumptions (K1) and (K2), for a similar reason given above for $\mathbf{R}$. Then 
\begin{align}\label{case2lowerbound}
 \sup_{z\in\mathbb{S}_{d-1}\cap \mathcal{V}(x), x\in\mathcal{S}} z^\top  L(x)^\top  \mathbf{S}L(x) z 
= \sup_{a\in \mathbb{S}_{d-r-1}, x\in\mathcal{S}} a^\top  V(x)^\top  \mathbf{S} V(x) a 
\geq  \lambda_{\min}(\mathbf{S})>0.
\end{align}
Then it can be seen that $\bar \sigma^2 \geq \frac{1}{2}\lambda_{\min}(\mathbf{S})$ 
when $h$ is small enough, using (\ref{wpgapprox}). Also notice that 
\begin{align*}
\sup_{z\in\mathbb{S}_{d-1}, x\in\mathcal{S}} z^\top  L(x)^\top  \mathbf{S}L(x) z\leq \lambda_{\max}(\mathbf{S})\|L(x)\|_{\text{op}}^2 = \lambda_{\max}(\mathbf{S}).
\end{align*}
Hence $\bar \sigma^2 \leq \frac{3}{2}\lambda_{\max}(\mathbf{S})$ when $h$ is small enough. 
\end{proof}

{\bf Proof of Theorem~\ref{gaussianappanti}}
\begin{proof}
We first show (\ref{thirdapproximation}) for case (a). By applying Proposition~\ref{2ndapproxtheorem} and Lemma~\ref{Kolmogrovgeneral}, we have that for some constant $C_1>0$,
\begin{align}\label{firstconclusion}
 \sup_{t\geq \bar\sigma_{0}} \Big| \mathbb{P}( \sqrt{nh^{d+4}} \sup_{x\in\mathcal{S}}\wh p(x) \leq t) - \mathbb{P}(\wt Z \leq t) \Big| 
\lesssim \sup_{t\geq \bar\sigma_{0}} \mathbb{P} (|\wt Z-t|\leq C_1\delta_{n,a} ) + (\log n)^{9/8} (\gamma_{n,h}^{(0)})^{1/4}.
\end{align}
We need to show that the first term on the right-hand side is a small quantity. Let $N(\delta) \equiv  N(\mathcal{F},e_P,\delta)$, where $e_P$ is a semi-metric such that for $g_1,g_2\in \mathcal{F}$, 
\begin{align}\label{canonicalmetric}
e_P(g_1,g_2)=\sqrt{\mathbb{E}[(G_P(g_1)-G_P(g_2))^2]}.
\end{align} 
Also let $\phi(\delta):=\mathbb{E}[\sup_{e_P(g_1,g_2) \leq \delta}|G_P(g_1)-G_P(g_2)|]$. From Lemma~\ref{sigmah2bounds} we know that $\bar\sigma_{0}>C_2$ for some constant $C_2>0$. Using Lemma~\ref{newanticoncentration}, when $n$ is large enough, we have $C_1\delta_{n,a} < \frac{1}{16}\bar\sigma_{0}$, and 
\begin{align}\label{lemma3.7result}
 \sup_{t\geq \bar\sigma_{0}} \mathbb{P}(|\wt Z - t|\leq C_1\delta_{n,a}) \leq \inf_{\substack{\delta,\eta>0\\\phi(\delta) + \eta\delta\leq \bar\sigma_0/16 }} \hbar (\delta,\eta) ,
%
\end{align}
where 
\begin{align*}
\hbar (\delta,\eta) = 2(1/\bar\sigma) (C_1\delta_{n,a}  + \phi(\delta) + \eta\delta) \sqrt{\log N(\delta)} (\sqrt{2\log N(\delta)}+3) + e^{-\eta^2/2}.
\end{align*}
We will bound the right-hand side of (\ref{lemma3.7result}). Using Dudley's maximal inequality (Corollary 2.2.8, \citealt{vaart1996weak}) we have
\begin{align}\label{phideltabound}
\phi(\delta) & \lesssim \int_0^{2\delta} \sqrt{\log N(\varepsilon)}d\varepsilon.
\end{align}
Using the relation between $\mathcal{F}$ and $\mathcal{G}$ in (\ref{FGrelation}), we have that there exist $A_0,v_0>0$ such that
\begin{align}\label{cvbound}
N(\varepsilon) \leq \sup_Q N(\mathcal{F},e_Q,\varepsilon) \leq  \sup_Q N(\mathcal{G},e_Q,h^{d/2}\varepsilon) \leq \Big( \frac{A_0}{h^{d/2} \varepsilon}\Big)^{v_0}.
\end{align}
Also we take $A_0/(2h^{d/2}\delta)\geq e$. Hence using integration by parts
\begin{align}\label{entropyintegral}
\int_0^{2\delta} \sqrt{\log N(\varepsilon)}d\varepsilon 
%
%
\leq & \int_0^{2\delta} \sqrt{ v_0\log (A_0/(h^{d/2}\varepsilon))}d\varepsilon \nonumber\\
=& \frac{\sqrt{v_0}A_0}{h^{d/2}} \Big[-(e^{-y^2}y)\Big|_{\sqrt{\log(A_0/(2h^{d/2}\delta))}}^{\infty}  + \int_{\sqrt{\log(A_0/(2h^{d/2}\delta))}}^{\infty} e^{-y^2}dy\Big] \nonumber\\
%
\lesssim & \; \delta \sqrt{\log (A_0/(2h^{d/2}\delta))}.
\end{align}
Using (\ref{phideltabound}) and (\ref{entropyintegral}) we then get
\begin{align}\label{hbarbound}
\hbar (\delta,\eta) \lesssim \log (1/(h^{d/2}\delta)) (\delta_{n,a}  + \delta \sqrt{\log (1/(h^{d/2}\delta))} +\eta\delta)  + e^{-\eta^2/2}.
\end{align}
We choose  $\delta=1/n$ and $\eta=\sqrt{\log n}$, which satisfies $\phi(\delta) + \eta\delta\leq \bar\sigma_0/16$ when $n$ is large enough. Note that by assumption (H), we have $\log n\simeq \log(n/h^{d/2})$. Using (\ref{lemma3.7result}) and (\ref{hbarbound}), we have
\begin{align}\label{anticoncenbound}
 \sup_{t\geq \bar\sigma_{0}} \mathbb{P}(|\wt Z - t|\leq C_1\delta_{n,a} ) \lesssim (\log n) (\delta_{n,a} ).
\end{align}
%
%
Plugging this result into (\ref{firstconclusion}), we then get (\ref{thirdapproximation}) for case (a). The proof of (\ref{thirdapproximation}) for case (b) is similar and hence omitted. 
\end{proof}

{\bf Proof of Lemma~\ref{wtzhquantile}}
\begin{proof}
Only the proof for case (a) is given because for case (b) the proof is similar. 
To show (\ref{gaussianexp1}), we need to prove that for some positive constants $C_1,C_2,$ and $C_3$,
\begin{align}
&\mathbb{P}(\wt Z \leq C_1\sqrt{|\log h|} ) \lesssim  h^{C_3}, \label{ztildelower}\\
&\mathbb{P}(\wt Z \geq C_2\sqrt{|\log h|} ) \lesssim  h^{C_3}.\label{ztildeupper}
\end{align}
We first show (\ref{ztildelower}). Recall that $\bar\sigma=\sup_{g\in\mathcal{F}}\text{Var}(G_p(g))$. We have shown that there exist constants $C_u,C_\ell>0$ such that $0<C_u<\bar\sigma<C_\ell<\infty$ in Lemma~\ref{sigmah2bounds}. Using the Borell's inequality (Proposition A.2.1 in \citealt{vaart1996weak}), we have
\begin{align}\label{Borell}
\max\Big\{\mathbb{P}\Big( \wt Z \leq  \frac{1}{2}\mathbb{E} \wt Z\Big), \mathbb{P}\Big( \wt Z \geq  \frac{3}{2}\mathbb{E} \wt Z\Big) \Big\}\leq \mathbb{P}\Big( |\wt Z -  \mathbb{E} \wt Z| \geq   \frac{1}{2}\mathbb{E} \wt Z \Big) \leq 2 \exp\Big(-\frac{[\mathbb{E} \wt Z]^2}{8\bar \sigma^2}\Big).
\end{align}
Using the Sudakov's minoration inequality (Theorem 7.4.1, \citealt{vershynin2018high}), we get
\begin{align}\label{minoration}
\mathbb{E} \wt Z \gtrsim \sup_{\varepsilon\geq 0} \varepsilon \sqrt{\log N(\mathcal{F},e_P,\varepsilon)},
\end{align} 
where $e_P$ is defined in (\ref{canonicalmetric}). We will find an upper bound for $N(\mathcal{F},e_P,\varepsilon)$. For $g_1,g_2\in\mathcal{F}$, we need to compute $e_P(g_1,g_2)$. Under assumption (F4), there exist $x_0\in\mathcal{S}$, $c_0>0$, and $r_0>0$ such that $\|\nabla f(x)\|>0$ and $f(x)\ge c_0$ for all $x\in \mathcal{S}_{x_0,r_0}.$ 

Under assumption (K2), there exists $b_0>0$ such that $\mathbf{R}_0 : = \int_{\mathcal{B}(\mathbf{0},b_0)} d^2K(u) [d^2K(u)]^\top du$ is a positive definite matrix, where $\mathbf{0}$ is the origin of $\mathbb{R}^d$.  In what follows we consider $h$ small enough such that $b_0h\leq r_0$. 
Recall that $\mathcal{V}(x)$ is the subspace spanned by $v_{r+1}(x),\cdots,v_{d}(x)$, for $x\in\mathcal{S}$. Let
\begin{align*}
\mathcal{F}_{x_0,r_0} = \{g_{x,z}(\cdot): x\in\mathcal{S}_{x_0,r_0},z\in\mathbb{S}_{d-1}\cap \mathcal{V}(x)\} \subset \mathcal{F}. 
\end{align*}
Next we consider any $g_{x_1,z_1}, g_{x_2,z_2}\in\mathcal{F}_{x_0,r_0}$ and will compute $e_P(g_{x_1,z_1}, g_{x_2,z_2})$. We consider $x_1$ and $x_2$ to be separated such that $\|x_1-x_2\|\geq h^\kappa$ for some $0<\kappa<1$. Then using change of variable $u=(x_1-y)/h$ below we have
%
\begin{align}\label{canonicmetricdiff}
& e_P(g_{x_1,z_1}, g_{x_2,z_2}) \nonumber\\
%
= &\frac{1}{h^d} \int_{\mathbb{R}^d} \Big( z_1^\top  M(x_1)^\top  d^2K\Big(\frac{x_1-y}{h}\Big) - z_2^\top  M(x_2)^\top  d^2K\Big(\frac{x_2-y}{h}\Big) \Big)^2 f(y)dy \nonumber\\
=& \int_{\mathbb{R}^d} \Big( z_1^\top  M(x_1)^\top  d^2K(u) - z_2^\top  M(x_2)^\top  d^2K\Big(\frac{x_2-x_1}{h} + u\Big) \Big)^2 f(x_1+hu)du \nonumber\\
\geq & \int_{\mathcal{B}(\mathbf{0},b_0)} \Big( z_1^\top  M(x_1)^\top  d^2K(u) - z_2^\top  M(x_2)^\top  d^2K\Big(\frac{x_2-x_1}{h} + u\Big) \Big)^2 f(x_1+hu)du \nonumber\\
\geq & \text{I} - \text{II},
\end{align}
where 
\begin{align*}
&\text{I} = \frac{1}{2} \int_{\mathcal{B}(\mathbf{0},b_0)}  ( z_1^\top  M(x_1)^\top  d^2K(u))^2 f(x_1+hu)du,\\
&\text{II} = \int_{\mathcal{B}(\mathbf{0},b_0)} \Big( z_2^\top  M(x_2)^\top  d^2K\Big(\frac{x_2-x_1}{h} + u\Big) \Big)^2 f(x_1+hu)du.
\end{align*}
%
%
%
%
For $\text{I}$, we have
\begin{align*}
\text{I} \geq \frac{1}{2} c_0 \int_{\mathcal{B}(\mathbf{0},b_0)}  ( z_1^\top  M(x_1)^\top  d^2K(u))^2 du = \frac{1}{2} c_0 z_1^\top  M(x_1)^\top  \mathbf{R}_0 M(x_1) z_1.  
\end{align*}
Since $x_1\in\mathcal{S}_{x_0,r_0}$ and $z_1\in \mathbb{S}_{d-1}\cap \mathcal{V}(x_1)$, similar to (\ref{varupperb}) we have
\begin{align}\label{Ibound}
\text{I} & \geq \frac{1}{2}c_0\frac{\lambda_{\min}(\mathbf{R}_0)}{\|(\mathbf{D}^+)^\top \|_{\text{op}}^2}[\nu_{1q}(x_1)]^2\| \nabla f(x_{1})\|^2>0.
%
%
\end{align}
%
For $\text{II}$, notice that for all $u\in\mathcal{B}(\mathbf{0},b_0)$, $\|(x_2-x_1)/h+u\| \geq h^{\kappa-1}-b_0\rightarrow\infty$, as $h\rightarrow0$. Hence under the assumption (K1), $d^2K((x_2-x_1)/h+u)\rightarrow 0$, and $\text{II}\rightarrow 0$ as $h\rightarrow0$. Then as a result of (\ref{canonicmetricdiff}) and (\ref{Ibound}), there exists a constant $C_0>0$ such that $e_P(g_{x_1,z_1}, g_{x_2,z_2}) \geq C_0$ when $h$ is small enough, for all $g_{x_1,z_1}, g_{x_2,z_2}\in\mathcal{F}_{x_0,r_0}$ such that $\|x_1-x_2\|\geq h^\kappa$ with $0<\kappa<1$. 

Let $\mathcal{J}=\{x_1,\cdots,x_{N_0}\}$ be an $h^\kappa$-packing of $\mathcal{S}_{x_0,r_0}$ based on the Euclidean distance, where $N_0$ is the cardinality, i.e., $\|x_i-x_j\|>h^\kappa$ for all $1\leq i\neq j\leq N_0$. We associate each $x_i$ with an arbitrarily chosen $z_i\in \mathbb{S}_{d-1}\cap \mathcal{V}(x_i)$. Then the above calculation implies that $\{g_{x_i,z_i}: i=1,\cdots,N_0\}$ forms a $C_0$-packing of $\mathcal{F}_{x_0,r_0}$ using $e_P$. 
%
Assumption (F4) implies that we can choose $\mathcal{J}$ such that $N_0 \gtrsim h^{-\kappa \phi_0}$. Using the relation between packing and covering numbers (see Lemma 5.2 in \citealt{niyogi2008finding}), we have
\begin{align*}
N(\mathcal{F},e_P,C_0/2) \geq N(\mathcal{F}^\star,e_P,C_0/2) \geq N_0 \gtrsim h^{-\kappa \phi_0}.
\end{align*}
Hence it follows from (\ref{minoration}) that
\begin{align}\label{eztileupper}
\mathbb{E} \wt Z \gtrsim \sqrt{|\log h|}.
\end{align}
Using (\ref{Borell}) we then get (\ref{ztildelower}). Next we show (\ref{ztildeupper}). By (\ref{Borell}) and (\ref{eztileupper}), it suffices to prove that $\mathbb{E} \wt Z \lesssim \sqrt{\log n}$. Since $\bar\sigma$ is an upper bound of the radius of $\mathcal{F}$ using $e_P$, using (\ref{cvbound}) and Dudley's maximal inequality (Corollary 2.2.8, \citealt{vaart1996weak}) we have for any $g_0\in \mathcal{F}$ and some positive constant $C$,
\begin{align*}
\mathbb{E} \wt Z \leq \mathbb{E} |G_P(g_0)| + C \int_0^{\bar\sigma} \sqrt{\log N(\mathcal{F},e_P,\varepsilon)} d\varepsilon \lesssim \sqrt{|\log h|},
\end{align*}
where the calculation is similar to (\ref{entropyintegral}) and we have used $\mathbb{E} |G_P(g_0)|=O(1)$. The proof is then completed. 
%
\end{proof}

{\bf Proof of Corollary~\ref{theoryci}}
\begin{proof}
We only give the proof for case (a), because it is similar for case (b). Recall that $t_{1-\alpha}$ is the $(1-\alpha)$ quantile of the distribution of $\omega_n^{-1} \wt Z$. Lemma~\ref{wtzhquantile} implies that for any fixed $\alpha\in(0,1)$, $t_{1-\alpha}\geq \omega_n^{-1} C_1\sqrt{\log n} \geq \omega_n^{-1} \bar\sigma_0$ for $n$ large enough, where $C_1>0$ is given in Lemma~\ref{wtzhquantile}. Hence by Theorem~\ref{gaussianappanti}, we get 
\begin{align}\label{ridgenessci}
 \Big| \mathbb{P}\Big( \omega_n \sup_{x\in\mathcal{S}}\wh p(x) \leq t_{1-\alpha}\Big) - (1-\alpha) \Big| = O( (\log n)\delta_{n}^+ ).
\end{align}
By Lemmas~\ref{biasbound} and \ref{talabound}, the result in \eqref{lambdar1diff} implies that for some constant $C>0$
\[\mathbb{P}\Big(\sup_{x\in\mathcal{S}} |\wh\lambda_{r+1}(x)-\lambda_{r+1}(x)| \leq C(\gamma_{n,h}^{(2)} +h^\beta)\Big)\geq 1- n^{-1}.\]
By assumption (F2), we get 
\begin{align}\label{eigenvaluesneg}
\mathbb{P}\Big(\sup_{x\in\mathcal{S}} \wh\lambda_{r+1}(x) < 0\Big)\geq 1-n^{-1},
\end{align}
when $n$ is large enough. Hence by using (\ref{ridgenessci}) and (\ref{eigenvaluesneg}) we get
\begin{align*}
\mathbb{P}(\mathcal{S}\subset \wh{\mathcal{S}}(t_{1-\alpha}) ) = & \mathbb{P}\Big( \omega_n \sup_{x\in\mathcal{S}} \wh p(x) \leq  t_{1-\alpha} \;\;\text{and}\;\; \sup_{x\in\mathcal{S}} \wh\lambda_{r+1}(x) < 0\Big) \\
= & 1-\alpha + O( (\log n)\delta_{n}^+ ).
\end{align*}
\end{proof}

\subsection{Proofs for Section \ref{bootstrapsec}}

Notice that 
\begin{align*}
\partial^{\alpha}\wh f^e (x) = \partial^{\alpha}\wh f(x) + \frac{1}{n} \sum_{i=1}^n e_i \Big( \frac{1}{h^{d+|\alpha|}} \partial^{\alpha} K\Big(\frac{x-X_i}{h}\Big) - \partial^{\alpha} \wh f(x)\Big),
\end{align*}
and $\mathbb{E} [\partial^{\alpha}\wh f^e (x) \;|\;\X_n] = \partial^{\alpha} \wh f(x).$ Let
\begin{align*}
&\varepsilon_{n,e}^{(0)} = \sup_{x\in\mathcal{H}} |\wh f^e(x) - \wh f(x)|,\\
&\varepsilon_{n,e}^{(1)} = \sup_{x\in\mathcal{H}}\|\nabla \wh f^e(x)- \nabla \wh f(x)\|, \\
&\varepsilon_{n,e}^{(2)} = \sup_{x\in\mathcal{H}} \|\nabla^2 \wh f^e(x)-\nabla^2 \wh f(x)\|_F^2.
\end{align*}
%
%
%
%
For any measurable function $g:\mathbb{R}^d\mapsto \mathbb{R}$, let 
\begin{align}\label{gne}
\mathbb{G}_n^e(g) = \frac{1}{\sqrt{n}} \sum_{i=1}^n e_i [g(X_i) - \mathbb{P}_n (g)].
\end{align}
 %
%
We will use the following two lemmas in the proof. Recall that $\sigma(\X_n)$ is the sigma-algebra generated by $\X_n$.

\begin{lemma}\label{multiplierboottalaglemma}
Assume that (F1), (K1), and (K2) hold and $\gamma_{n,h}^{(2)}\rightarrow0$. For any $\ell_1,\ell_2>0$ and $\alpha\in\mathbb{N}^d$ with $|\alpha|=0,1,2$, there exist a constant $C_\alpha>0$ and an event $E_1\in \sigma(\X_n)$ with probability $1-n^{-\ell_1}$ such that on the event $E_1$ for $n$ large enough we have
%
\begin{align}\label{multiplierboottalag}
\mathbb{P} \Big(\sup_{x\in\mathcal{H}} |\partial^{\alpha} \wh f^e (x) -\partial^{\alpha} \wh f(x)| > C_\alpha \gamma_{n,h}^{(|\alpha|)} \;|\;\X_n \Big) \leq n^{-\ell_2}.
\end{align}

\end{lemma}
\begin{proof}
Recall $\mathcal{K}_\alpha$ defined in (\ref{kalpha}) and $\mathbb{G}_n^e$ given in (\ref{gne}). Note that conditional on $\X_n$, $\{\mathbb{G}_n^e(g):g\in\mathcal{K}_{\alpha}\}$ is a centered Gaussian process indexed by $\mathcal{K}_{\alpha}$.
%
We can write
\begin{align}\label{multiplierbootscaled}
\sup_{x\in\mathcal{H}} |\partial^{\alpha}\wh f^e (x) - \partial^{\alpha}\wh f(x)| = \frac{1}{\sqrt{n}h^{d+|\alpha|}} \sup_{g\in \mathcal{K}_{\alpha}} |\mathbb{G}_n^e (g)|.
\end{align}
%
%
Consider the classes of functions $\mathcal{K}_{\alpha}^2 = \{g^2:g\in \mathcal{K}_{\alpha}\}.$ We have
%
\begin{align}
&\sup_{g\in \mathcal{K}_\alpha^2} \mathbb{E}[g(X)] = \sup_{x\in\mathcal{H}} \int_{\mathbb{R}^d} \Big[\partial^\alpha K \Big(\frac{x-y}{h}\Big)\Big]^2 f(y)dy \leq h^d \|f\|_\infty \int_{\mathbb{R}^d} [\partial^\alpha K(u)]^2 du =O(h^d),\label{expboundg}\\
&\sup_{g\in \mathcal{K}_\alpha^2} \text{Var}[g(X)] \leq \sup_{x\in\mathcal{H}} \int_{\mathbb{R}^d} \Big[\partial^\alpha K \Big(\frac{x-y}{h}\Big)\Big]^4 f(y)dy \leq h^d \|f\|_\infty \int_{\mathbb{R}^d} [\partial^\alpha K(u)]^4 du =O(h^d),\label{varboundg}
\end{align}
where we have used change of variable $u=(x-y)/h$ and assumption (K1).
%
%
Also assumption (K2) implies that $\sup_{y\in\mathbb{R}^d} \sup_{g\in\mathcal{K}_\alpha^2}|g(y)|\leq b^2$ and it follows from Corollary A.1 in \cite{chernozhukov2014gaussian} that $\sup_Q N(\mathcal{K}_\alpha^2,e_Q,2b^2\varepsilon)\leq (\frac{A}{\varepsilon})^v$, for $\varepsilon\in(0,1]$. Using Talagrand's inequality given in Lemma~\ref{talagrandineq}, for any $\ell_1>0$, we have with probability at least $1 - n^{-\ell_1}$,
\begin{align}\label{condevent1}
\sup_{|\alpha|\leq 3}\sup_{g\in \mathcal{K}_{\alpha}^2} |\mathbb{G}_n(g)| \lesssim h^{d/2} \sqrt{\log n}.
\end{align}
%
Denote the above event by $E_1$. The rest of the proof is conditional on the event $E_1$. Using (\ref{expboundg}) and (\ref{condevent1}), we have
\begin{align}\label{signan2cond}
\sigma_n^2  := \sup_{g\in \mathcal{K}_{\alpha} } \mathbb{P}_n (g^2) & \leq  \sup_{g\in \mathcal{K}_{\alpha}^2 } \mathbb{E} [g(X)] + n^{-1/2} \sup_{g\in \mathcal{K}_{\alpha}^2} |\mathbb{G}_n(g)| 
 \lesssim h^d + n^{-1/2} h^{d/2} \sqrt{\log n} 
 \lesssim h^d.
\end{align}
Note that $\sup_{g\in \mathcal{K}_{\alpha}} \text{Var}(\mathbb{G}_n^e (g) \;|\;\X_n) \leq \sigma_n^2$. Then applying the Borell's inequality (Proposition A.2.1 in \citealt{vaart1996weak}) we get that for any $\ell_2>0$ there exists a constant $C_1>0$ such that
\begin{align}\label{gegborell}
\mathbb{P} \Big(  \sup_{g\in \mathcal{K}_{\alpha}} |\mathbb{G}_n^e (g)| > \mathbb{E} [\sup_{g\in \mathcal{K}_{\alpha}} |\mathbb{G}_n^e (g)| \;|\;\X_n] + C_1h^{d/2} \sqrt{\log n} \;|\;\X_n\Big) \leq n^{-\ell_2}.
\end{align}
Next we will find an upper bound for $\mathbb{E} [\sup_{g\in \mathcal{K}_{\alpha}} |\mathbb{G}_n^e (g)| \;|\;\X_n]$. Denote
\begin{align}\label{gne12}
\mathbb{G}_n^{e,(1)} (g) = \frac{1}{\sqrt{n}} \sum_{i=1}^n e_i g(X_i),\; \text{ and } \mathbb{G}_n^{e,(2)} (g) =  \mathbb{P}_n(g)\Big[\frac{1}{\sqrt{n}} \sum_{i=1}^n e_i \Big].
\end{align}
Then obviously $\mathbb{G}_n^{e} (g) = \mathbb{G}_n^{e,(1)} (g) - \mathbb{G}_n^{e,(2)} (g).$ Hence
\begin{align}\label{gnerela}
\mathbb{E} \Big[\sup_{g\in \mathcal{K}_{\alpha}} |\mathbb{G}_n^e (g)| \;|\;\X_n\Big] \leq  \mathbb{E} \Big[\sup_{g\in \mathcal{K}_{\alpha}} |\mathbb{G}_n^{e,(1)} (g)| \;|\;\X_n\Big] + \mathbb{E} \Big[\sup_{g\in \mathcal{K}_{\alpha}} |\mathbb{G}_n^{e,(2)} (g)| \;|\;\X_n\Big].
\end{align}
Given $\X_n$, define the semimetric $e_{P_n}$ on $\mathcal{K}_{\alpha}$ as follows: for $g_1,g_2\in\mathcal{K}_{\alpha}$, let
\begin{align}\label{epndef}
e_{P_n}(g_1,g_2) = & \sqrt{\mathbb{E}[ (\mathbb{G}_n^{e,(1)}(g_1) - \mathbb{G}_n^{e,(1)}(g_2))^2 \;|\;\X_n]} 
= \sqrt{\frac{1}{n} \sum_{i=1}^n [g_1(X_i) - g_2(X_i)]^2} .
\end{align}
Take $\sigma_n=\sqrt{\sigma_n^2}$. Note that $\sigma_n$ is an upper bound of the radius of $\mathcal{K}_{\alpha}$ using the semimetric $e_{P_n}$. Also note that by using Jensen's inequality, $\sup_{g\in\mathcal{K}_\alpha}\mathbb{E}[|\mathbb{G}_n^{e,(1)}(g)|\, \;|\;\X_n]\leq\sigma_n.$ Since $N(\mathcal{K}_{\alpha}, e_{P_n},\varepsilon b) \leq \sup_Q N(\mathcal{K}_{\alpha},e_Q,\varepsilon b) \leq (A/\varepsilon)^v$, for $\varepsilon\in(0,1]$, using Dudley's maximal inequality (Corollary 2.2.8, \citealt{vaart1996weak}) we have
\begin{align}\label{ge1bound}
\mathbb{E} \Big[\sup_{g\in \mathcal{K}_{\alpha}} |\mathbb{G}_n^{e,(1)} (g)| \;|\;\X_n\Big] & \lesssim \int_0^{\sigma_n} \sqrt{\log N(\mathcal{K}_{\alpha},e_{P_n},\varepsilon)} d \varepsilon + \sup_{g\in\mathcal{K}_\alpha}\mathbb{E}[|\mathbb{G}_n^{e,(1)}(g)|\, \;|\;\X_n] \nonumber\\
%
%
& \lesssim \int_0^{\sigma_n/b} \sqrt{v \log (A/\varepsilon)} d \varepsilon + \sigma_n \nonumber\\
& \lesssim \sigma_n \sqrt{v \log (Ab/\sigma_n)}  \nonumber \\
& \lesssim h^{d/2} \sqrt{\log n},
\end{align}
where the computation of the integral is similar to (\ref{entropyintegral}). Furthermore, using Jensen's inequality, we have $\sup_{g\in \mathcal{K}_{\alpha}}|\mathbb{P}_n(g)| \leq \sigma_n$. Hence
\begin{align}\label{ge2bound}
\mathbb{E} \Big[\sup_{g\in \mathcal{K}_{\alpha}} |\mathbb{G}_n^{e,(2)} (g)| \;|\;\X_n\Big] \leq \sigma_n \mathbb{E} \Big|\frac{1}{\sqrt{n}} \sum_{i=1}^n e_i \Big| \lesssim h^{d/2}.
\end{align}
Plugging (\ref{gnerela}), (\ref{ge1bound}), and (\ref{ge2bound}) into (\ref{gegborell}) we get for some constant $C_2>0$,
\begin{align*}
\mathbb{P} \Big(  \sup_{g\in \mathcal{K}_{\alpha}} |\mathbb{G}_n^e (g)| >  C_2h^{d/2} \sqrt{\log n} \;|\;\X_n\Big) \leq n^{-\ell_2}.
\end{align*}
Using the relation in (\ref{multiplierbootscaled}) we then obtain (\ref{multiplierboottalag}). 
\end{proof}

For $x\in\mathcal{H}$, and $z\in\mathbb{S}_{d-1}$, denote 
\begin{align}\label{qjxzdef}
q(\cdot;x,z)=
\begin{cases}
z^\top  M(x)^\top  d^2 K(\frac{x - \cdot}{h})  & \text{ for case (a)},\\
z^\top  L(x)^\top  \nabla K(\frac{x - \cdot}{h})  & \text{ for case (b)}.
\end{cases}
\end{align}
Given $\X_n$ and any $\epsilon>0$, consider the classes of functions:
\begin{align}\label{djepsilon}
\mathcal{D}(\epsilon) = \{ q(\cdot;x,z) - q(\cdot;y,z): \; x\in \mathcal{S}, \|x-y\|\leq \epsilon, z\in\mathbb{S}_{d-1} \}.
\end{align}
%
Let 
\begin{align}\label{xinje}
\Xi_{n}^{e}(\epsilon) = \frac{1}{\sqrt{h^d}}\sup_{g\in \mathcal{D}(\epsilon) } \mathbb{G}_n^e(g).
\end{align} 

\begin{lemma}\label{differenceprocess}
Assume that (F1), (F2), (K1), and (K2) hold. Suppose $h\rightarrow0$ and $nh^q\rightarrow\infty$ for some $q>0$ as $n\rightarrow\infty$. Let $\varepsilon_n$ be a sequence such that $\varepsilon_n/h\rightarrow0$. For any $\ell_1,\ell_2>0$, there exit a constant $C>0$ and an event $E_2\in\sigma(\X_n)$ with $\mathbb{P}(E_2)\geq 1-n^{-\ell_1}$ such that on $E_2$, for $n$ large enough, we have 
\begin{align}\label{differenceprocessextreme}
\mathbb{P} \Big(  \Xi_{n}^{e}(\varepsilon_n)  >  C (\frac{\varepsilon_n}{h}) \sqrt{\log n} \;|\;\X_n\Big) \leq n^{-\ell_2}.
\end{align}
\end{lemma}
%
\begin{proof}
Given $\X_n$, notice that $\{\mathbb{G}_n^e(g):g\in \mathcal{D}(\varepsilon_n)\}$ is a centered Gaussian process indexed by $\mathcal{D}(\varepsilon_n)$. Let
\begin{align}\label{khalpha}
\mathcal{D}^2(\varepsilon_n) = \{g^2:g\in \mathcal{D}(\varepsilon_n)\}.
\end{align}
For $k=2,4$, let $\eta_k(x,y,z):= \mathbb{E}\{ [ q(X;x,z) - q(X;y,z) ]^k\}$, where $q(\cdot;x,z)$ is defined in (\ref{qjxzdef}). Next we show that for $k=2,4$,
\begin{align}\label{etakjxyz}
\eta_{k,\sup} := \sup_{x\in\mathcal{S}, \|x-y\|\leq\varepsilon_n,z\in\mathbb{S}_{d-1}} \eta_k(x,y,z) \lesssim  h^d (\frac{\varepsilon_n}{h})^k.
\end{align}
First we consider case (a), for which we have by using change of variables $u=(x-w)/h$,
\begin{align*}
\eta_k(x,y,z) = & \int_{\mathbb{R}^d} \Big( z^\top  M(x)^\top  d^2 K\Big(\frac{x-w}{h}\Big) - z^\top  M(y)^\top  d^2 K\Big(\frac{y - w}{h}\Big)\Big)^k f(w)dw \\
= & h^d \int_{\mathbb{R}^d} \Big( z^\top  M(x)^\top  d^2 K(u) - z^\top  M(y)^\top  d^2 K\Big(\frac{y - x}{h}+u\Big)\Big)^k f(x-hu)du \\
\leq & 2^{k-1}h^d [\kappa_{k,\text{I}}(x,y,z)  + \kappa_{k,\text{II}}(x,y,z) ],
\end{align*}
where 
\begin{align*}
&\kappa_{k,\text{I}}(x,y,z) = \int_{\mathbb{R}^d} \Big\{ z^\top  [M(x) - M(y)]^\top  d^2 K(u) \Big\}^k f(x-hu)du,\\
&\kappa_{k,\text{II}}(x,y,z) = \int_{\mathbb{R}^d} \Big\{ z^\top  M(y)^\top  \Big[d^2 K(u) - d^2 K\Big(\frac{y - x}{h}+u\Big)\Big]\Big\}^k f(x-hu)du.
\end{align*}
Let $\widebar D_1(x,y) = \nabla f(x) - \nabla f(y)$, $\widebar D_2(x,y) = \nabla^2 f(x) - \nabla^2 f(y)$ and $H(x,y,t)=\nabla^2 f(y) + tD_2(x,y)$ for $t\in[0,1]$. Using the definition of $T$ in (\ref{tsigmad0}) we can write
\begin{align*}
[M(x) - M(y)]^\top  d^2 K(u) = T(\nabla^2 f(x);\nabla^2 K(u))\nabla f(x) - T(\nabla^2 f(y);\nabla^2 K(u))\nabla f(y).
\end{align*}
Similar to (\ref{ttaylorexp}) and (\ref{tdiffbound}), when $\varepsilon_n$ is small enough we get for all $x\in\mathcal{S}$, $\|x-y\|\leq\varepsilon_n$, and $u\in\mathbb{R}^d$,
\begin{align*}
 \|T(\nabla^2 f(x);\nabla^2 K(u)) - T(\nabla^2 f(y);\nabla^2 K(u))\|_F
\leq &\sup_{t\in[0,1]}\| T^{(1)}(H(x,y,t),\widebar D_2(x,y); \nabla^2 K(u)) \|_F \\
\lesssim & \|\widebar D_2(x,y)\|_F \|\nabla^2 K(u)\|_F.
\end{align*}
Hence using the assumptions (F1) and (K1), we have
\begin{align*}
\sup_{x\in\mathcal{S}, \|x-y\|\leq\varepsilon_n,z\in\mathbb{S}_{d-1}} \kappa_{k,\text{I}}(x,y,z) \lesssim \sup_{x\in\mathcal{S}, \|x-y\|\leq\varepsilon_n} (\|\widebar D_1(x,y)\|_F + \|\widebar D_2(x,y)\|_F)^k \lesssim (\varepsilon_n)^k.
\end{align*}
Also the assumption (K1) leads to
\begin{align*}
\sup_{x\in\mathcal{S}, \|x-y\|\leq\varepsilon_n,z\in\mathbb{S}_{d-1}} \kappa_{k,\text{II}}(x,y,z) \lesssim (\frac{\varepsilon_n}{h})^k.
\end{align*}
Hence the result in (\ref{etakjxyz}) follows for case (a). For case (b), the role of $M$ is replaced by $L$, and the calculation is similar and omitted. 
Hence for both cases (a) and (b),
\begin{align}
&\sup_{g\in \mathcal{D}^2(\varepsilon_n)} \mathbb{E}(g(X)) = \eta_{2,\sup} \lesssim h^d (\frac{\varepsilon_n}{h})^2, \label{supgdj2e}\\
&\sup_{g\in \mathcal{D}^2(\varepsilon_n)} \text{Var}(g(X)) \leq \eta_{4,\sup} \lesssim h^d (\frac{\varepsilon_n}{h})^4.\label{supgdj2var}
\end{align}
Next we show that both $\mathcal{D}(\varepsilon_n)$ and $\mathcal{D}^2(\varepsilon_n)$ are VC type classes with constant envelopes. For $\delta_0$ given in assumption (F2), let
\begin{align*}
\mathcal{G}^{\delta_0} = \{q(\cdot;x,z): x\in\mathcal{S}\oplus\delta_0, z\in\mathbb{S}_{d-1}\}.
\end{align*}
Then following similar arguments in the proof of Proposition~\ref{2ndapproxtheorem}, we can show that $\mathcal{G}^{\delta_0}$ has a constant envelop $b_0$ and there exist $A_0>0$ and $v_0>0$ such that for every $0<\varepsilon\leq 1$, $\sup_Q N(\mathcal{G}^{\delta_0},e_Q,\varepsilon b_0) \leq (\frac{A_0}{\varepsilon})^{v_0},$ 
where the supremum is taken over all finitely discrete probability measures.
Given $\X_n$, let $\mathcal{W} = \{ q(\cdot;x,z) - q(\cdot;y,z): \; x,y\in\mathcal{S}\oplus\delta_0, z\in\mathbb{S}_{d-1} \}\subset\{g_1(\cdot)-g_2(\cdot): \; g_1,g_2\in\mathcal{G}^{\delta_0}\}.$ When $\varepsilon_n$ is small enough, we have $\mathcal{D}(\varepsilon_n)\subset \mathcal{W}$. Hence
\begin{align}\label{d1mhenv}
\sup_{g\in \mathcal{D}(\varepsilon_n)} |g(y)| \leq \sup_{g\in\mathcal{W}}|g(y)| \leq 2b_0.
\end{align}
Also using Lemma A.6 in \cite{chernozhukov2014gaussian} we have for $0<\varepsilon\leq 1$,
\begin{align}\label{d1mhcn}
\sup_Q N( \mathcal{D}(\varepsilon_n),e_Q,2b_0\varepsilon) \leq \sup_Q N(\mathcal{W},e_Q,2b_0\varepsilon) \leq (\frac{A_0}{\varepsilon})^{2v_0}.
\end{align}
Let $\mathcal{W}^2 =\{g^2: g\in \mathcal{W}\}$. When $\varepsilon_n$ is small enough, we have $\mathcal{D}^2(\varepsilon_n)\subset \mathcal{W}^2$. Then 
\begin{align*}
\sup_{g\in \mathcal{D}^2(\varepsilon_n)} |g(y)| \leq \sup_{g\in\mathcal{W}^2}|g(y)| \leq 4b_0^2,
\end{align*}
and using Corollary A.1 in \cite{chernozhukov2014gaussian} we have for $0<\varepsilon\leq 1$,
\begin{align}\label{d2mhcn}
\sup_Q N( \mathcal{D}^2(\varepsilon_n),e_Q,4b_0^2\varepsilon) \leq \sup_Q N(\mathcal{W}^2,e_Q,4b_0^2\varepsilon) \leq (\frac{A_0}{\varepsilon})^{2v_0}.
\end{align}
%
Using (\ref{d1mhcn}), (\ref{d2mhcn}), (\ref{supgdj2var}), and the Talagrand's inequality (Lemma~\ref{talagrandineq}), for any $\ell_1>0$, when $n$ is large enough, we have with probability at least $1-n^{\ell_1}$
\begin{align}\label{condevent2}
\sup_{g\in \mathcal{D}^2(\varepsilon_n) } |\mathbb{G}_n(g)| \lesssim h^{d/2} (\frac{\varepsilon_n}{h})^2 \sqrt{\log n} .
\end{align}
Denote the above event by $E_2$. The rest of the proof is conditional on the event $E_2$. Then using (\ref{supgdj2e}) it follows that
\begin{align}\label{wtsigma2}
\wt \sigma_n^2  := \sup_{g\in \mathcal{D}(\varepsilon_n) } \mathbb{P}_n (g^2) & \leq  \sup_{g\in \mathcal{D}^2(\varepsilon_n) } \mathbb{E} [g(X)] + n^{-1/2} \sup_{g\in \mathcal{D}^2(\varepsilon_n) } |\mathbb{G}_n(g)| \nonumber\\
& \lesssim h^d (\frac{\varepsilon_n}{h})^2 + n^{-1/2} h^{d/2} (\frac{\varepsilon_n}{h})^2 \sqrt{\log n} 
 \lesssim h^d (\frac{\varepsilon_n}{h})^2.
\end{align}
Note that $\sup_{g\in \mathcal{D}(\varepsilon_n) } \text{Var}(\mathbb{G}_n^e (g) \;|\;\X_n) \leq \wt\sigma_n^2$. Then applying the Borell's inequality (Proposition A.2.1, \citealt{vaart1996weak}) we have that for any $\ell_2>0$ there exists a constant $C_1>0$ such that
\begin{align}\label{gnegprob}
\mathbb{P} \Big(  \sup_{g\in \mathcal{D}(\varepsilon_n) } |\mathbb{G}_n^e (g)| > \mathbb{E} \Big[\sup_{g\in \mathcal{D}(\varepsilon_n) } |\mathbb{G}_n^e (g)| \;|\;\X_n\Big] + C_1h^{d/2}(\frac{\varepsilon_n}{h}) \sqrt{\log n} \;|\;\X_n\Big) \leq n^{-\ell_2}.
\end{align}
Next we compute $\mathbb{E} [\sup_{g\in \mathcal{D}(\varepsilon_n) } |\mathbb{G}_n^e (g)| \;|\;\X_n]$. Recall $\mathbb{G}_n^{e,(1)}$ and $\mathbb{G}_n^{e,(2)}$ defined in (\ref{gne12}). We have
\begin{align}\label{gegineq}
\mathbb{E} \Big[\sup_{g\in \mathcal{D}(\varepsilon_n) } |\mathbb{G}_n^e (g)| \;|\;\X_n\Big] \leq  \mathbb{E} \Big[\sup_{g\in \mathcal{D}(\varepsilon_n) } |\mathbb{G}_n^{e,(1)} (g)| \;|\;\X_n\Big] + \mathbb{E} \Big[\sup_{g\in \mathcal{D}(\varepsilon_n) } |\mathbb{G}_n^{e,(2)} (g)| \;|\;\X_n\Big].
\end{align}
Note that $\wt\sigma_n=\sqrt{\wt\sigma_n^2}$ is an upper bound of the radius of $\mathcal{D}(\varepsilon_n)$ using the semimetric $e_{P_n}$ defined in (\ref{epndef}). Also note that by Jensen's inequality, $\sup_{g\in\mathcal{D}(\varepsilon_n)}\mathbb{E}[|\mathbb{G}_n^{e,(1)}(g)|\, \;|\;\X_n]\leq\wt\sigma_n.$ Using (\ref{d1mhcn}) and Dudley's maximal inequality (Corollary 2.2.8, \citealt{vaart1996weak}) we have
\begin{align}\label{ge1gbound}
\mathbb{E} [\sup_{g\in \mathcal{D}(\varepsilon_n) } |\mathbb{G}_n^{e,(1)} (g)| \;|\;\X_n] & \lesssim \int_0^{\tilde{\sigma}_n} \sqrt{\log N(\mathcal{D}(\varepsilon_n) ,e_{P_n},\varepsilon)} d \varepsilon + \sup_{g\in \mathcal{D}(\varepsilon_n) } \mathbb{E}[|\mathbb{G}_n^{e,(1)} (g)| \;|\;\X_n] \nonumber\\
%
%
& \lesssim \int_0^{\tilde\sigma_n/(2b_0)} \sqrt{2v_0 \log (A_0/\varepsilon)} d \varepsilon + \tilde{\sigma}_n \nonumber\\
& \lesssim \wt\sigma_n \sqrt{v_0 \log (A_0b_0/\wt\sigma_n)} \nonumber  \\
& \lesssim h^{d/2} (\frac{\varepsilon_n}{h}) \sqrt{\log n},
\end{align}
where we have used (\ref{wtsigma2}). Furthermore, we have that $\sup_{g\in \mathcal{D}(\varepsilon_n) }|\mathbb{P}_n(g)| \leq \wt\sigma_n$ by using Jensen's inequality. Hence
\begin{align}\label{ge2gbound}
\mathbb{E} \Big[\sup_{g\in \mathcal{D}(\varepsilon_n) } |\mathbb{G}_n^{e,(2)} (g)| \;|\;\X_n\Big] \leq \wt\sigma_n \mathbb{E} \Big|\frac{1}{\sqrt{n}} \sum_{i=1}^n e_i \Big| \lesssim h^{d/2} (\frac{\varepsilon_n}{h}).
\end{align}
Plugging (\ref{gegineq}), (\ref{ge1gbound}), and (\ref{ge2gbound}) into (\ref{gnegprob}) we get for some constant $C_2>0$,
\begin{align*}
\mathbb{P} \Big(  \sup_{g\in \mathcal{D}(\varepsilon_n) } |\mathbb{G}_n^e (g)| >  C_2h^{d/2} (\frac{\varepsilon_n}{h}) \sqrt{\log n} \;|\;\X_n\Big) \leq n^{-\ell_2}.
\end{align*}
and hence (\ref{differenceprocessextreme}) follows and the proof is completed. 
\end{proof}

{\bf Proof of Proposition~\ref{multiplierbootapp}}
\begin{proof}
%
%
Let $E_\dagger$ be the intersection of the events $E_1$ and $E_2$ and the event specified in (\ref{ridgeinclusion}). Then $E_\dagger\in\sigma(\X_n)$ and $\mathbb{P}(E_\dagger)\geq 1- n^{-1}$ by choosing $\ell,\ell_1,\ell_2$ large enough in lemmas~\ref{ridgeinclusionlemma}, \ref{multiplierboottalaglemma} and \ref{differenceprocess}. We only focus on case (a); the proof for case (b) is similar and hence omitted. For clarity the rest of the proof is divided into the following five steps.

{\bf Step 1.} Lemma~\ref{multiplierboottalaglemma} implies that there exist constants $C_{k,e}>0$ for $k=0,1,2$ such that on the event $E_\dagger$ when $n$ is large enough 
\begin{align}\label{multiplierbootecondpb}
\mathbb{P} ( \varepsilon_{n,e}^{(k)} > C_{k,e} \gamma_{n,h}^{(k)} \;|\;\X_n) \leq n^{-1}.
\end{align}
Consequently,
\begin{align}\label{multiplierbootepb}
 \mathbb{P} ( \varepsilon_{n,e}^{(k)} > C_{k,e} \gamma_{n,h}^{(k)} ) 
= & \mathbb{E} [\mathbb{P} ( \varepsilon_{n,e}^{(k,e)} > C_{k,e} \gamma_{n,h}^{(k)} \;|\;\X_n )] \nonumber\\
\leq & \mathbb{E} [\mathbb{P} ( \varepsilon_{n,e}^{(k,e)} > C_{k,e} \gamma_{n,h}^{(k)} \;|\;\X_n )\mathbf{1}(\X_n\in E_\dagger)] + \mathbb{P}(E_\dagger^\complement) 
%
\leq  2n^{-1}.
\end{align}

{\bf Step 2.} We denote
\begin{align*}
%
Z^e = & \sup_{g\in \mathcal{F}} \mathbb{G}_n^e(g) = \omega_n \sup_{x\in\mathcal{S}} \|M(x)^\top D_{n,2}^e(x) \|,
\end{align*}
where $D_{n,2}^e(x)=d^2 \wh f^e(x)- d^2 \wh f(x)$. We want to show that there exists a constant $C_1>0$ such that
\begin{align}\label{multibootstep2}
\mathbb{P} \Big( \Big|\omega_n \sup_{x\in\wh{\mathcal{S}}(\rho_n)} |\wh p^e(x) - \wh p(x)| -  Z^e\Big| > C_1\sqrt{\log n} \frac{\rho_n^{1/\beta^\prime}}{h} \Big) \lesssim n^{-1},
\end{align}
We divide the proof in this step into the following three substeps.

{\bf Step 2a.} Let $\wh M$ the direct plug-in estimator of $M$ as defined below. Suppose that $\nabla^2 \wh f(x)$ has $\wh q=\wh q(x)$ distinct eigenvalues: $\wh \mu_1(x)>\cdots>\wh \mu_{\wh q}(x)$, each with multiplicities $\wh \ell_1,\cdots,\wh \ell_{\wh q}$. Let $\wh s_0=0$, $\wh s_1=\ell_1,\cdots,\wh s_{\wh q}=\wh \ell_1+\cdots+\wh \ell_{\wh q}$. Then for $j=1,\cdots,\wh q$, $\wh \mu_j(x)=\wh \lambda_{\wh s_{j-1}+1}(x)=\cdots=\wh \lambda_{\wh s_j}(x)$. We write $\wh P_j(x) = \wh E_j(x)\wh E_j(x)^\top ,$ where $\wh E_j(x)$ is a $d\times \wh\ell_j$ matrix with orthonormal eigenvectors $\wh v_{\wh s_{j-1}+1}(x),\cdots,\wh v_{\wh s_j}(x)$ as its columns. Suppose that there are $\wh q_r=\wh q_r(x)$ distinct eigenvalues large than $\wh{\lambda}_{r+1}(x)$. For $j=1,\cdots,\wh q$, let 
\begin{align}\label{sjxhatexpress}
\wh S_j(x) =  \sum_{\substack{k=1 \\ k\neq j}}^{\wh q}\frac{1}{\wh \mu_j(x) - \wh \mu_k(x)} \wh P_k(x), 
\end{align}
and
\begin{align}\label{MhatTexpress}
\wh M(x)^\top  =  \sum_{j=\wh q_r+1}^{\wh q} \{  \wh P_j(x) \otimes [\wh S_k(x) \nabla \wh f(x)]^\top  + \wh S_j(x) \otimes [\wh P_k(x) \nabla \wh f(x)]^\top  \} \mathbf{D}.
\end{align}
Define $\wh g_{x,z}(y) = h^{d/2}z^\top  \wh M(x)^\top  d^2K(\frac{x-y}{h})$ and the class of functions
\begin{align*}
\wh{\mathcal{F}} = \{\wh g_{x,z}(\cdot): x\in \wh{\mathcal{S}}(\rho_n),\; z\in\mathbb{S}_{d-1}  \}.
\end{align*}
Then similar to (\ref{ridgenessapprox}), conditional on $\X_n$ we have that when $\max\{\varepsilon_{n,e}^{(0)},\varepsilon_{n,e}^{(1)},\varepsilon_{n,e}^{(2)}\}\rightarrow 0$, 
\begin{align}\label{multipbootappbound}
\Big|\sup_{x\in\wh{\mathcal{S}}(\rho_n)} |\wh p^e(x) -\wh p(x) | - \sup_{x\in\wh{\mathcal{S}}(\rho_n)} \|\wh M(x)^\top D_{n,2}^e(x) \| \Big| = O(\varepsilon_{n,e}^{(1)} + (\varepsilon_{n,e}^{(2)})^2).
\end{align}
Note that compared with (\ref{ridgenessapprox}), here the $O(h^2)$-term does not appear on the right-hand side of (\ref{multipbootappbound}) because $\mathbb{E}[\partial^\alpha\wh f^e(x) \;|\;\X_n]=\partial^\alpha\wh f(x)$. Denote 
\[\wh Z^e = \sup_{g\in \wh{\mathcal{F}}} \mathbb{G}_n^e(g) = \omega_n \sup_{x\in\wh{\mathcal{S}}(\rho_n)} \|\wh M(x)^\top D_{n,2}^e(x) \|.\] 
By (\ref{multiplierbootecondpb}) and (\ref{multipbootappbound}), and using the assumption (H), on the event $E_\dagger$ we get that for some constant $C_2>0$,
\begin{align}\label{ridgenessapproxmulti}
& \mathbb{P} \Big( \Big| \omega_n \sup_{x\in\wh{\mathcal{S}}(\rho_n)} |\wh p^e(x) - \wh p(x)| - \wh Z^e \Big| > C_2 h\sqrt{\log n}  \;|\;\X_n\Big) \leq n^{-1}.
\end{align}
%
%
%
%
Next we will compare $\wh Z^e$ and $Z^e$. Note that $|\wh Z^e - Z^e| \leq \wh Z^{e,(1)} + \wh Z^{e,(2)},$ where
\begin{align*}
&\wh Z^{e,(1)} = \omega_n \Big |\sup_{x\in\wh{\mathcal{S}}(\rho_n)} \|\wh M(x)^\top D_{n,2}^e(x) \| - \sup_{x\in\wh{\mathcal{S}}(\rho_n)} \| M(x)^\top D_{n,2}^e(x) \| \Big|, \\
&\wh Z^{e,(2)} =  \omega_n \Big | \sup_{x\in\wh{\mathcal{S}}(\rho_n)} \| M(x)^\top D_{n,2}^e(x) \|  - \sup_{x\in\mathcal{S}} \| M(x)^\top D_{n,2}^e(x) \|  \Big| .
\end{align*}
We will study $\wh Z^{e,(1)}$ and $\wh Z^{e,(2)}$, respectively. 

{\bf Step 2b.} First we consider $\wh Z^{e,(1)}$. Notice that on the event $E_\dagger$, when $n$ is large enough,
\begin{align*}
\wh Z^{e,(1)} 
\leq &\omega_n\sup_{x\in\mathcal{S}\oplus\delta_0} \Big| \|\wh M(x)^\top D_{n,2}^e(x) \| - \| M(x)^\top D_{n,2}^e(x) \| \Big| \\
\leq &\omega_n\sup_{x\in\mathcal{S}\oplus\delta_0}  \|[\wh M(x)^\top   - M(x)^\top ]D_{n,2}^e(x) \| .
\end{align*}
When $\varepsilon_n^{(2)}$ is small enough we can show that 
\begin{align}\label{whmxoverbound}
\sup_{x\in\mathcal{S}\oplus\delta_0 } \|[\wh M(x)^\top   - M(x)^\top ]D_{n,2}^e(x) \| \lesssim (\varepsilon_n^{(1)} + \varepsilon_n^{(2)})\varepsilon_{n,e}^{(2)}.
\end{align}
The proof of (\ref{whmxoverbound}) is long and given in Step 5.
%
%
%
%
Using (\ref{originalepb}) and (\ref{multiplierbootepb}) and noticing that $\omega_n \gamma_{n,h}^{(2)}=\sqrt{\log n}$ and $\gamma_{n,h}^{(2)}/h^\beta\rightarrow\infty$ under the assumption (H), we have for some constant $C_3>0$,
\begin{align}\label{zhe1bound}
\mathbb{P} \Big( \wh Z^{e,(1)}  > C_3^2 \sqrt{\log n}\gamma_{n,h}^{(2)} \Big) 
%
%
%
\leq  \mathbb{P} (\varepsilon_{n}^{(1)} + \varepsilon_{n}^{(2)} > C_3\gamma_{n,h}^{(2)}) + \mathbb{P} ( \varepsilon_{n,e}^{(2)} > C_3\gamma_{n,h}^{(2)} )
\leq & 3n^{-1}.
\end{align}

{\bf Step 2c.} 
%
%
%
Define $\wh Z^{e,(3)} = \Xi_{n,1}^{e}(C_0\rho_n^{1/\beta^\prime})$, where $\Xi_{n,1}^{e}$ is given in (\ref{xinje}). We have that on the event $E_\dagger$,
\begin{align*}
%
\wh Z^{e,(2)} \leq & \omega_n \sup_{x\in\mathcal{S},y\in \mathcal{S} \oplus (C_0\rho_n^{1/\beta^\prime}) } \Big |  \| M(x)^\top D_{n,2}^e(x)\| - \| M(y)^\top D_{n,2}^e(y) \| \Big| \leq  \wh Z^{e,(3)}.
%
%
%
\end{align*}
Using Lemma~\ref{differenceprocess} we have that on the event $E_\dagger$, 
\begin{align}\label{zhe2bound}
\mathbb{P} \Big(  \wh Z^{e,(2)} >  CC_0\frac{\rho_n^{1/\beta^\prime}}{h}\sqrt{\log n}  \;|\;\X_n\Big) \leq \mathbb{P} \Big(  \wh Z^{e,(3)} >  CC_0\frac{\rho_n^{1/\beta^\prime}}{h}\sqrt{\log n}  \;|\;\X_n\Big) \leq 2n^{-1}.
\end{align}
Therefore using (\ref{zhe1bound}) and (\ref{zhe2bound}) we have
\begin{align*}
\mathbb{P} \Big( |\wh Z^e - Z^e|  >  CC_0 \frac{\rho_n^{1/\beta^\prime}}{h} \sqrt{\log n} \;|\;\X_n\Big) \leq 5n^{-1}.
\end{align*}
Combining this with (\ref{ridgenessapproxmulti}) we get on the event $E_\dagger$ for some constant $C_4>0$,
\begin{align*}
\mathbb{P} \Big( \Big|\omega_n \sup_{x\in\wh{\mathcal{S}}(\rho_n)} |\wh p^e(x) - \wh p(x)| -  Z^e\Big| > C_4 \frac{\rho_n^{1/\beta^\prime}}{h} \sqrt{\log n} \;|\;\X_n \Big) \leq 6n^{-1}.
\end{align*}
Hence (\ref{multibootstep2}) follows from a similar calculation for (\ref{multiplierbootepb}).
%
%
%

{\bf Step 3.} Recall $\wt Z = \sup_{g\in \mathcal{F}} G_P(g)$. We will show that for every $\gamma\in(0,1)$, there exists a random variable $\wt Z^e \stackrel{d\;|\;\X_n}{=} \wt Z $ such that for some constants $C_5,C_6>0,$
\begin{align}\label{unconditionres}
\mathbb{P} \Big(\Big|\omega_n\sup_{x\in\wh{\mathcal{S}}(\rho_n)} |\wh p^e(x) - \wh p(x)|-\wt Z^e\Big| >C_5 \alpha_n^e \Big) \leq C_6 (\gamma+n^{-1}),
\end{align}
where $\alpha_n^e = \sqrt{\log n} \frac{\rho_n^{1/\beta^\prime}}{h} + \frac{\sqrt{\log n}}{\gamma^{1+q}} (n^{1/q} \gamma_{n,h}^{(0)} + (\gamma_{n,h}^{(0)})^{1/2})$ for any $q\in[4,\infty)$. Here $\stackrel{d\;|\;\X_n}{=}$ means equality in conditional distribution given $\X_n.$ 
Using Theorem 2.2 in \cite{chernozhukov2016empirical}, where we take $b=O(1)$, $\sigma=O(h^{d/2})$, $K_n=O(\log n)$, we have that for every $\gamma\in(0,1)$ and $q$ sufficiently large, there exists a random variable $\wt Z^e \stackrel{d\;|\;\X_n}{=} \wt Z $ such that
\begin{align}\label{chern2.2}
\mathbb{P}(h^{d/2}|Z^e - \wt Z^e | > C_7b_n^e ) \leq C_8 (\gamma+n^{-1}),
\end{align}
for some constants $C_7,C_8>0,$ where $b_n^e = \frac{\log n}{\gamma^{1+1/q} n^{1/2-1/q}}  + \frac{h^{d/4} (\log n)^{3/4}}{\gamma^{1+1/q} n^{1/4}}$. 
%
%
Combining (\ref{multibootstep2}) with (\ref{chern2.2}), we then get (\ref{unconditionres}).

{\bf Step 4.} Applying Markov's inequality to (\ref{unconditionres}), for every $\kappa\in (0,1)$ we have that with probability at least $1-\kappa$,
\begin{align*}
\mathbb{P} \Big(\Big|\omega_n  \sup_{x\in\wh{\mathcal{S}}(\rho_n)} |\wh p^e(x) - \wh p(x)| -\wt Z^e\Big| >  C_5 \alpha_n^e \;|\;\X_n\Big) \leq \kappa^{-1} C_6 (\gamma+n^{-1}). 
\end{align*}
Using Lemma~\ref{Kolmogrovgeneral} we have that with probability at least $1-\kappa$
\begin{align*}
& \sup_{t\geq \bar\sigma_{0}} \Big|\mathbb{P}\Big(\omega_n  \sup_{x\in\wh{\mathcal{S}}(\rho_n)} |\wh p^e(x) - \wh p(x)| \leq t \;|\;\X_n\Big) - \mathbb{P}(\wt Z^e \leq t\;|\;\X_n) \Big| \\
\leq & \kappa^{-1} C_6 (\gamma+n^{-1}) + \sup_{t\geq \bar\sigma_{0}}\mathbb{P} (|\wt Z^e-t| \leq C_5 \alpha_n^e  \;|\;\X_n) ,
\end{align*}
which, due to the relation $\wt Z^e \stackrel{d\;|\;\X_n}{=} \wt Z $, is equivalent to 
\begin{align*}
&\sup_{t\geq \bar\sigma_{0}} \Big|\mathbb{P}\Big(\omega_n \sup_{x\in\wh{\mathcal{S}}(\rho_n)} |\wh p^e(x) - \wh p(x)| \leq t \;|\;\X_n\Big) - \mathbb{P}(\wt Z \leq t ) \Big| \\
\leq &\kappa^{-1} C_6 (\gamma+n^{-1}) + \sup_{t\geq \bar\sigma_{0}}\mathbb{P} (|\wt Z-t| \leq C_5 \alpha_n^e ).
\end{align*}
Similar to (\ref{anticoncenbound}) we have
\begin{align*}
\sup_{t\geq \bar\sigma_{0}}\mathbb{P} (|\wt Z-t| \leq C_5 \alpha_n^e ) \lesssim (\log n) \alpha_n^e,
\end{align*}
where we may increase the constant $C_5$ if necessary. Then (\ref{zetan2res}) follows immediately. 

{\bf Step 5.} We give the remaining proof to show that on the event $E_\dagger$, (\ref{whmxoverbound}) holds when $\varepsilon_n^{(2)}$ is small enough. Denote $\widebar D_{n,2}^e(x) = \nabla^2 \wh f^e(x)-\nabla^2 \wh f(x)$ so that we can write $\mathbf{D}\textsf{vec}[\widebar D_{n,2}^e(x)]= D_{n,2}^e(x)$. Then notice that by (\ref{MhTfullexpress}) and (\ref{MhatTexpress})
\begin{align*}
&M(x)^\top D_{n,2}^e(x) = \sum_{j= q_r+1}^q \Big[P_j(x) \widebar D_{n,2}^e(x) S_j(x) + S_j(x) \widebar D_{n,2}^e(x) P_j(x)\Big] \nabla f(x) , \\
%
&\wh M(x)^\top D_{n,2}^e(x) = \sum_{j= \wh q_r+1}^{\wh q}  \Big[\wh P_j(x) \widebar D_{n,2}^e(x) \wh S_j(x) + \wh S_j(x) \widebar D_{n,2}^e(x) \wh P_j(x) \Big] \nabla \wh f(x) .
\end{align*}
For any $\Sigma\in S\mathbb{R}^{d\times d}$, supposing it has eigenvalues $\gamma_1\geq\cdots\geq\gamma_r>\gamma_{r+1}\geq\cdots\geq\gamma_d$ with associated eigenvectors $u_1,\cdots,u_d$, for any given $D_0\in S\mathbb{R}^{d\times d}$, we define a matrix-valued mapping $T:S\mathbb{R}^{d\times d}\mapsto S\mathbb{R}^{d\times d}$ such that
\begin{align}\label{tsigmad0}
T(\Sigma;D_0) =\sum_{k=1}^r\sum_{j=r+1}^{d}\frac{1}{\gamma_j-\gamma_k}[u_ju_j^\top D_0u_ku_k^\top  + u_ku_k^\top D_0u_ju_j^\top ].
\end{align}
Then $T$ is an analytic matrix-valued function at $\Sigma$ and using (\ref{sjxexpress}) and (\ref{sjxhatexpress}) we can write 
\begin{align}
& M(x)^\top D_{n,2}^e(x) = T(\nabla^2 f(x);\widebar D_{n,2}^e(x)) \nabla f(x), \label{mxtdn2}\\
&\wh M(x)^\top D_{n,2}^e(x) =T(\nabla^2 \wh f(x);\widebar D_{n,2}^e(x)) \nabla \wh f(x).\label{hmxtdn2}
\end{align}
For any $D\in S\mathbb{R}^{d\times d}$, and for $\ell=1,2,\cdots,$ let
\begin{align}\label{ellderivative}
T^{(\ell)}(\Sigma,D;D_0) = \frac{d^\ell}{d \tau^\ell} T(\Sigma + \tau D;D_0) \Big |_{\tau=0} .
\end{align}
Recall that $\widebar D_{n,2}(x) = \nabla^2 \wh f(x)-\nabla^2 f(x)$, and $H_n(x,t)=\nabla^2 f(x)+t  \widebar D_{n,2}(x)$ for $t\in[0,1]$. By Taylor expansion, we get
\begin{align} \label{ttaylorexp}
&T(\nabla^2 \wh f(x);\widebar D_{n,2}^e(x)) - T(\nabla^2 f(x);\widebar D_{n,2}^e(x)) \nonumber\\
= &\int_0^1 (1-t) T^{(1)}(H_n(x,t),\widebar D_{n,2}(x); \widebar D_{n,2}^e(x)) dt. 
\end{align}
We use the notation defined between (\ref{rnxexpress}) and (\ref{sjxtexpress}).
For $1\leq j\neq k\leq q$, and any $D_0\in S\mathbb{R}^{d\times d}$, let 
\begin{align*}
T_{jk,x,t}(H_n(x,t);D_0) = \nu_{jk}(x,t) [  P_j(x,t) D_0 P_k(x,t) ].
\end{align*}
Define $T_{j,x,t}(H_n(x,t);D_0)=T_{j,x,t,\dagger}(H_n(x,t);D_0) + T_{j,x,t,\ddagger}(H_n(x,t);D_0)$, where
\begin{align*}
&T_{j,x,t,\dagger}(H_n(x,t);D_0)=P_j(x,t)D_0S_j(x,t)=\sum_{k\neq j}T_{jk,x,t}(H_n(x,t);D_0),\\
&T_{j,x,t,\ddagger}(H_n(x,t);D_0)=S_j(x,t)D_0P_j(x,t)=\sum_{k\neq j}T_{kj,x,t}(H_n(x,t);D_0),
\end{align*}
with $S_j(x,t)$ given in (\ref{sjxtexpress}). The $\ell$th G\^{a}teaux derivatives of $T_{kj,x,t}$, $T_{j,x,t,\dagger}$, $T_{j,x,t,\ddagger}$ and $T_{j,x,t}$, denoted by $T_{kj,x,t}^{(\ell)}$, $T_{j,x,t,\dagger}^{(\ell)}$, $T_{j,x,t,\ddagger}^{(\ell)}$ and $T_{j,x,t}^{(\ell)}$, respectively, can be defined in a way similar to (\ref{ellderivative}). Then notice that 
\begin{align}\label{thnsum}
T(H_n(x,t);D_0) = \sum_{j= q_r+1}^q T_{j,x,t}(H_n(x,t);D_0).
\end{align}
Let $\mathbb{C}$ be the set of all complex numbers. For any $z\in\mathbb{C}$ and $z\neq \mu_k(x,t)$, $k=1,\cdots,q$, the matrix $z\mathbf{I}_d - H_n(x,t)$ is nonsingular and 
\begin{align}\label{inversecomplex}
[z\mathbf{I}_d - H_n(x,t)]^{-1}=\sum_{\ell=1}^q (z-\mu_\ell(x,t))^{-1} P_\ell(x,t).
\end{align}
Let $\mathcal{C}_k=\mathcal{C}_k(x,t)$ be a sufficiently small circle in the complex plane centered at $\mu_k(x,t)$ and oriented in the counterclock direction. Let $\mathbf{i}=\sqrt{-1}$ be the imaginary unit. By Cauchy's theorem we can write
\begin{align}\label{tjkxthnxt}
T_{jk,x,t}(H_n(x,t);D_0) = \frac{1}{(2\pi\mathbf{i})^2} \oint_{\mathcal{C}_k} \oint_{\mathcal{C}_j} \frac{1}{z-\wt z} [z\mathbf{I}_d - H_n(x,t)]^{-1} D_0 [\wt z\,\mathbf{I}_d - H_n(x,t)]^{-1} dzd\wt z.
\end{align}
The idea of using contour integrals to represent matrix-valued functions has been used in classical matrix perturbation theory (\citealt{kato2013perturbation}). Also see \cite{shapiro2002differentiability}. Notice that for $z\in \cup_{k=1}^q\mathcal{C}_k$, when $\tau\rightarrow 0$,
\begin{align}\label{aidhnxt}
&[z\mathbf{I}_d - H_n(x,t) -\tau D]^{-1} \nonumber\\
= &[z\mathbf{I}_d - H_n(x,t)]^{-1} + \tau [z\mathbf{I}_d - H_n(x,t)]^{-1}D[z\mathbf{I}_d - H_n(x,t)]^{-1} + o(\tau\|D\|_F).
\end{align}
Using (\ref{tjkxthnxt}), (\ref{aidhnxt}) and the definition of G\^{a}teaux derivatives, we get
\begin{align*}
T_{jk,x,t}^{(1)}(H_n(x,t),D;D_0) = T_{jk,x,t,\text{I}}^{(1)}(H_n(x,t),D;D_0) + T_{jk,x,t,\text{II}}^{(1)}(H_n(x,t),D;D_0),
\end{align*}
where
\begin{align*}
& T_{jk,x,t,\text{I}}^{(1)}(H_n(x,t),D;D_0) \\
=& \frac{1}{(2\pi\mathbf{i})^2} \oint_{\mathcal{C}_k} \oint_{\mathcal{C}_j} \frac{1}{z-\wt z} [z\mathbf{I}_d - H_n(x,t)]^{-1} D [z\mathbf{I}_d - H_n(x,t)]^{-1} D_0 [\wt z\,\mathbf{I}_d - H_n(x,t)]^{-1} dzd\wt z,
\end{align*}
and
\begin{align*}
& T_{jk,x,t,\text{II}}^{(1)}(H_n(x,t),D;D_0) \\
=& \frac{1}{(2\pi\mathbf{i})^2} \oint_{\mathcal{C}_k} \oint_{\mathcal{C}_j} \frac{1}{z-\wt z} [z\mathbf{I}_d - H_n(x,t)]^{-1} D_0 [\wt z\,\mathbf{I}_d - H_n(x,t)]^{-1} D [\wt z\, \mathbf{I}_d - H_n(x,t)]^{-1} dzd\wt z.
\end{align*}
We will evaluate $T_{jk,x,t,\text{I}}^{(1)}(H_n(x,t),D;D_0)$ first. Using (\ref{inversecomplex}) we can write
\begin{align*}
& T_{jk,x,t,\text{I}}^{(1)}(H_n(x,t),D;D_0) \\
= & \sum_{\ell =1}^q \sum_{s=1}^q \sum_{w =1}^q \{\frac{1}{(2\pi\mathbf{i})^2} \oint_{\mathcal{C}_k} \oint_{\mathcal{C}_j} \frac{1}{(z-\wt z)(z-\mu_\ell(x,t))(z-\mu_s(x,t))(\wt z-\mu_w(x,t))}  dzd\wt z \;\\
&\hspace{2cm} \times [P_\ell(x,t) D P_s(x,t) D_0 P_w(x,t)]\}.
\end{align*}
Using Cauchy's theorem, we have for $k\neq j$,
\begin{align*}
&\frac{1}{(2\pi\mathbf{i})^2} \oint_{\mathcal{C}_k} \oint_{\mathcal{C}_j} \frac{1}{(z-\wt z)(z-\mu_\ell(x,t))(z-\mu_s(x,t))(\wt z-\mu_w(x,t))}  dzd\wt z \\
= &
\begin{cases}
0 & \text{if } (\ell\neq j \text{ and } s\neq j) \text{ or } w\neq k, \\
\nu_{jk}(x,t)\nu_{js}(x,t) &  \text{if } \ell = j \text{ and } s\neq j \text{ and } w=k,\\
\nu_{jk}(x,t)\nu_{j\ell}(x,t) &  \text{if } \ell \neq j \text{ and } s= j \text{ and } w=k,\\
-[\nu_{jk}(x,t)]^2& \text{if } \ell=s=j  \text{ and } w=k.
\end{cases}
\end{align*}
Similarly,
\begin{align*}
& T_{jk,x,t,\text{II}}^{(1)}(H_n(x,t),D;D_0) \\
= & \sum_{\ell =1}^q \sum_{s=1}^q \sum_{w =1}^q \{\frac{1}{(2\pi\mathbf{i})^2} \oint_{\mathcal{C}_k} \oint_{\mathcal{C}_j} \frac{1}{(z-\wt z)(z-\mu_\ell(x,t))(\wt z-\mu_s(x,t))(\wt z-\mu_w(x,t))}  dzd\wt z \;\\
&\hspace{2cm} \times  [P_\ell(x,t) D_0 P_s(x,t) D P_w(x,t)]\}.
\end{align*}
%
%
%
%
%
Hence overall
\begin{align*}
T_{j,x,t,\dagger}^{(1)}(H_n(x,t),D;D_0) = \sum_{m=1}^6 U_{jm,x,t,}(H_n(x,t),D;D_0),
\end{align*}
where
\begin{align*}
U_{j1,x,t}(H_n(x,t),D;D_0) = & \sum_{k\neq j} \sum_{s\neq j} \nu_{jk}(x,t)\nu_{js}(x,t) [P_j(x,t) D P_s(x,t) D_0 P_k(x,t)], \\
U_{j2,x,t}(H_n(x,t),D;D_0) =  & \sum_{k\neq j} \sum_{\ell\neq j} \nu_{jk}(x,t)\nu_{j\ell}(x,t) [P_\ell(x,t) D P_j(x,t) D_0 P_k(x,t)],\\
U_{j3,x,t}(H_n(x,t),D;D_0) =  & - \sum_{k\neq j} [\nu_{jk}(x,t)]^2 [P_j(x,t) D P_j(x,t) D_0 P_k(x,t)],\\
U_{j4,x,t}(H_n(x,t),D;D_0) =  & \sum_{k\neq j} \sum_{w\neq k} \nu_{kw}(x,t)\nu_{jk}(x,t) [P_j(x,t) D_0 P_k(x,t) D P_w(x,t)],\\
U_{j5,x,t}(H_n(x,t),D;D_0) =   & \sum_{k\neq j}\sum_{s\neq k} \nu_{jk}(x,t)\nu_{ks}(x,t)  [P_j(x,t) D_0 P_s(x,t) D P_k(x,t)],\\
U_{j6,x,t}(H_n(x,t),D;D_0) =  & \sum_{k\neq j} [\nu_{jk}(x,t)]^2 [P_j(x,t) D_0 P_k(x,t) D P_k(x,t)].
\end{align*}
Note that
\begin{align*}
\sum_{m=4}^6 U_{jm,x,t}(H_n(x,t),D;D_0) = & \sum_{k\neq j} \sum_{w\neq j} \nu_{jk}(x,t)\nu_{jw}(x,t)[P_j(x,t) D_0 P_k(x,t) D P_w(x,t)] \\
- & \sum_{k\neq j} [\nu_{jk}(x,t)]^2 [P_j(x,t) D_0 P_k(x,t) D P_j(x,t)]\\
- & \sum_{k\neq j} [\nu_{jk}(x,t)]^2 [P_j(x,t) D_0 P_j(x,t) D P_k(x,t)].
\end{align*}
For $j,k,\ell\in \mathcal{I}$, we denote $\Pi_{jk\ell}^{(\text{I})}=\Pi_{jk\ell}^{(\text{I})}(x,t;D_0,D)=P_j(x,t)D_0P_k(x,t)DP_\ell(x,t)$ and $\Pi_{jk\ell}^{(\text{II})}=\Pi_{jk\ell}^{(\text{II})}(x,t;D_0,D)=P_j(x,t)DP_k(x,t)D_0P_\ell(x,t)$. For $j\in\mathcal{I}_>$, we can write
\begin{align*}
&T_{j,x,t,\dagger}^{(1)}(H_n(x,t),D;D_0) \\
= &\mathop{\sum\sum}\limits_{k,\ell\in\mathcal{I}\backslash\{ j\}} \nu_{jk}(x,t)\nu_{j\ell}(x,t)[\Pi_{jk\ell}^{(\text{I})} + \Pi_{jk\ell}^{(\text{II})} + \Pi_{kj\ell}^{(\text{II})}] - \sum_{k\in\mathcal{I}\backslash\{ j\}} [\nu_{jk}(x,t)]^2[\Pi_{jjk}^{(\text{I})} + \Pi_{jjk}^{(\text{II})} + \Pi_{jkj}^{(\text{I})}]  .
\end{align*}
Similarly we get
\begin{align*}
&T_{j,\ddagger}^{(1)}(H_n(x,t),D;D_0) \\
= &\mathop{\sum\sum}\limits_{k,\ell\in\mathcal{I}\backslash\{ j\}} \nu_{jk}(x,t)\nu_{j\ell}(x,t)[\Pi_{jk\ell}^{(\text{II})} + \Pi_{jk\ell}^{(\text{I})} + \Pi_{kj\ell}^{(\text{I})}] - \sum_{k\in\mathcal{I}\backslash\{ j\}} [\nu_{jk}(x,t)]^2[\Pi_{jjk}^{(\text{II})} + \Pi_{jjk}^{(\text{I})} + \Pi_{jkj}^{(\text{II})}] .
\end{align*}
Hence
\begin{align*}
T_{j}^{(1)}(H_n(x,t),D;D_0) 
= &\sum_{i\in\{\text{I},\text{II}\}} \mathop{\sum\sum}\limits_{k,\ell\in\mathcal{I}\backslash\{ j\}} \nu_{jk}(x,t)\nu_{j\ell}(x,t)[\Pi_{jk\ell}^{(i)} + \Pi_{jk\ell}^{(i)} + \Pi_{kj\ell}^{(i)}] \\
&- \sum_{i\in\{\text{I},\text{II}\}} \sum_{k\in\mathcal{I}\backslash\{ j\}} [\nu_{jk}(x,t)]^2[\Pi_{jjk}^{(i)} + \Pi_{jjk}^{(i)} + \Pi_{jkj}^{(i)} ].
\end{align*}
By (\ref{thnsum}), we obtain
\begin{align}\label{t1finalexp}
T^{(1)}(H_n(x,t),D;D_0) 
= &\sum_{i\in\{\text{I},\text{II}\}} \sum_{j\in\mathcal{I}_>} \mathop{\sum\sum}\limits_{k,\ell\in\mathcal{I}_\leq} \nu_{jk}(x,t)\nu_{j\ell}(x,t)[\Pi_{jk\ell}^{(i)} + \Pi_{jk\ell}^{(i)} + \Pi_{kj\ell}^{(i)}] \nonumber\\
&- \sum_{i\in\{\text{I},\text{II}\}} \sum_{j\in\mathcal{I}_>} \sum_{k\in\mathcal{I}_\leq} [\nu_{jk}(x,t)]^2[\Pi_{jjk}^{(i)} + \Pi_{jjk}^{(i)} + \Pi_{jkj}^{(i)} ].
%
\end{align}
Similar to (\ref{pijkellf}) we have that for $i\in\{\text{I},\text{II}\}$, $1\leq j,k,\ell\leq q$, $x\in\mathcal{S}\oplus\delta_0$ and $t\in[0,1]$,
\begin{align*}
\|\Pi_{jk\ell}^{(i)}(x,t;D_0,D)\|_F\leq d^{3/2} \|D\|_F\|D_0\|_F .
\end{align*}
Therefore by (\ref{weylgap}) (\ref{ttaylorexp}), (\ref{t1finalexp}), when $\varepsilon_n^{(2)}$ is small enough we have
\begin{align}\label{tdiffbound}
\sup_{x\in\mathcal{S}\oplus\delta_0}\|T(\nabla^2 \wh f(x);D_{n,2}^e(x)) - T(\nabla^2 f(x);\widebar D_{n,2}^e(x))\|_F\lesssim \varepsilon_n^{(2)} \varepsilon_{n,e}^{(2)}.
\end{align}
Also it is clear that $\sup_{x\in\mathcal{S}\oplus\delta_0}\|T(\nabla^2 \wh f(x);D_{n,2}^e(x)) \|_F\lesssim \varepsilon_{n,e}^{(2)}$. Then immediately (\ref{whmxoverbound}) follows by observing (\ref{mxtdn2}) and (\ref{hmxtdn2}) and using a telescoping argument. 
\end{proof}

{\bf Proof of Theorem~\ref{multiplierquantile}}
\begin{proof}
For $\alpha\in(0,1)$, recall that $t_{1-\alpha}$ and $t_{1-\alpha}^e$ are the $(1-\alpha)$-quantiles of $\omega_n^{-1}\wt Z$ and $\sup_{x\in\wh{\mathcal{S}}(\rho_n)} |\wh p^e (x) - \wh p(x)|$, respectively. We denote the event in (\ref{zetan2res}) by $E_3$. Then $\mathbb{P}(E_3)\geq 1-\kappa$. 

We will show that for any fixed $\alpha$, $t_{1-\alpha}^e\geq \omega_n^{-1}\bar\sigma_{0}$ for $n$ large enough on the event $E_3$. Note with $\zeta_{n}^{\gamma,\kappa,q}$ given in (\ref{zetan2def}), on the event $E_3$, with the constant $C$ given in Proposition~\ref{multiplierbootapp},
\begin{align*}
& \mathbb{P}\Big(\sup_{x\in\wh{\mathcal{S}}(\rho_n)} |\wh p^e(x) - \wh p(x)| \leq t_{1-\alpha + C\zeta_{n}^{\gamma,\kappa,q}} \;|\;\X_n\Big) \\
\geq & \mathbb{P}(\omega_n^{-1} \wt Z \leq t_{1-\alpha + C\zeta_{n}^{\gamma,\kappa,q}} )-  C\zeta_{n}^{\gamma,\kappa,q} = 1-\alpha.
\end{align*}
Therefore on the event $E_3$,
\begin{align}\label{tewttplus}
t_{1-\alpha}^e \leq t_{1-\alpha + C\zeta_{n}^{\gamma,\kappa,q}}.
\end{align}
Similarly, we can show that on the event $E_3$, 
\begin{align}\label{tewttminus}
t_{1-\alpha}^e \geq t_{1-\alpha - C\zeta_{n}^{\gamma,\kappa,q}}.
\end{align}
Using Lemmas~\ref{wtzhquantile} and \ref{sigmah2bounds} we have for $n$ large enough (and therefore $h$ small enough), 
\begin{align*}
t_{1-\alpha - C\zeta_{n}^{\gamma,\kappa,q}} \geq C_1 \omega_n^{-1} \sqrt{|\log h|} \geq\omega_n^{-1} \bar\sigma_{0},
\end{align*}
where the constant $C_1$ is given in Lemma~\ref{wtzhquantile}. Therefore $t_{1-\alpha}^e\geq \omega_n^{-1} \bar\sigma_{0}$ for $n$ large enough on the event $E_3$. Hence using Theorem~\ref{gaussianappanti} and (\ref{tewttplus}), we have for some constant $\wt C>0,$
\begin{align*}
\mathbb{P}\Big(  \sup_{x\in\mathcal{S}} \wh p(x) \leq t_{1-\alpha}^e\Big) 
\leq & \mathbb{P}(\omega_n^{-1}\wt Z \leq t_{1-\alpha}^e ) +  \wt C(\log n)\delta_{n}^+  \\
%
%
\leq & \mathbb{P}(\{\omega_n^{-1} \wt Z \leq t_{1-\alpha + C\zeta_{n}^{\gamma,\kappa,q}} \}\cap E_3)  + \wt C(\log n)\delta_{n}^+ + \mathbb{P}(E_3^\complement) \\
\leq & 1-\alpha + \wt C(\log n)\delta_{n}^+  + C\zeta_{n}^{\gamma,\kappa,q} +\kappa .
\end{align*}
Similarly, using Theorem~\ref{gaussianappanti} and (\ref{tewttminus}) we have 
%
\begin{align*}
\mathbb{P}(  \sup_{x\in\mathcal{S}} \wh p(x) \leq t_{1-\alpha}^e) 
\geq & \mathbb{P}(\omega_n^{-1}\wt Z \leq t_{1-\alpha}^e ) -  \wt C(\log n)\delta_{n}^+  \\
\geq & \mathbb{P}(\{\omega_n^{-1}\wt Z \leq t_{1-\alpha - C\zeta_{n}^{\gamma,\kappa,q}}\}\cap E_3 )   - \wt C(\log n)\delta_{n}^+ -\mathbb{P}(E_3^\complement)  \\
\geq & 1-\alpha - \wt C(\log n)\delta_{n}^+  - C\zeta_{n}^{\gamma,\kappa,q}  -\kappa.
\end{align*}
Hence overall
\begin{align}\label{overallapprox}
\Big| \mathbb{P}\Big(  \sup_{x\in\mathcal{S}}\wh p(x) \leq  t_{1-\alpha}^e \Big)  - (1-\alpha)\Big| %
\leq  \wt C(\log n)\delta_{n}^+  + C\zeta_{n}^{\gamma,\kappa,q} +\kappa.
%
\end{align}
Choose $\kappa=\sqrt{\gamma} = \gamma^{-(1+1/q)}(\log n)^{3/2}(\gamma_{n,h}^{(0)})^{1/2}$, that is, $\gamma = (\log n (\gamma_{n,h}^{(0)})^{1/3})^{\frac{1}{1+2/(3q)}}.$ 
So for any arbitrarily small positive $\eta$, by taking $q=2/(3\eta)$ we have
\begin{align*}
\zeta_{n}^{\gamma,\kappa,q}  = O\Big( (\sqrt{\log n} (\gamma_{n,h}^{(0)})^{1/6})^{1-\eta} +  (\log n)^{3/2}  h^{-1}\rho_n^{1/\beta^\prime}\Big), 
\end{align*}
which then leads to
\begin{align}\label{ridgenessboundmu}
&\Big| \mathbb{P}\Big(  \sup_{x\in\mathcal{S}}\wh p(x) \leq  t_{1-\alpha}^e \Big)  - (1-\alpha)\Big| \nonumber\\
= &O\Big((\log n)\delta_{n}^++(\sqrt{\log n} (\gamma_{n,h}^{(0)})^{1/6})^{1-\eta} +  (\log n)^{3/2}  h^{-1}\rho_n^{1/\beta^\prime}\Big) \nonumber\\
= &O(\tau_n).
\end{align}
Hence by using (\ref{ridgenessboundmu}) and (\ref{eigenvaluesneg}) we get
\begin{align*}
\mathbb{P}(\mathcal{S}\subset \wh{\mathcal{S}}(t_{1-\alpha}^e) ) = & \mathbb{P}\Big( \sup_{x\in\mathcal{S}} \wh p(x) \leq  t_{1-\alpha}^e \;\;\text{and}\;\; \sup_{x\in\mathcal{S}} \wh\lambda_{r+1}(x) < 0\Big) \\
= & 1-\alpha + O(\tau_n).
\end{align*}
This is \eqref{multiplierquantileres}. 

Next we show (\ref{confiregsize}) for case (a). Notice that
\begin{align}\label{ridgenessonci}
\sup_{x\in \wh{\mathcal{S}}(t_{1-\alpha}^e)} p(x) & \leq \sup_{x\in \wh{\mathcal{S}}(t_{1-\alpha}^e)} \wh p(x) + \sup_{x\in \wh{\mathcal{S}}(t_{1-\alpha}^e)} |\wh p(x) - p(x)| 
 \leq t_{1-\alpha}^e + \sup_{x\in \mathcal{H} } |\wh p(x) - p(x)|.
\end{align}
With the choice of $\kappa$ and $\gamma$ above and using (\ref{tewttplus}) we have with probability at least $1-\kappa=1-\frac{1}{2}n^{-C_0}$ for some $C_0>0$ that $t_{1-\alpha}^e \leq t(1-\alpha + C\zeta_{n}^{\gamma,\kappa,q})\lesssim \omega_n^{-1}\sqrt{|\log h|} \simeq \gamma_{n,h}^{(2)}$ by using Lemma~\ref{wtzhquantile} and assumption (H). 
%
From (\ref{whpxphdiffbound}) we know
\begin{align}
\label{hatpdiff}
\sup_{x\in\mathcal{H}}|\wh p(x) - p(x)|\lesssim \varepsilon_n^{(1)} + \varepsilon_n^{(2)}.
\end{align}
By Lemmas~\ref{biasbound} and \ref{talabound} (also see (\ref{originalepb})), and using assumption (H), $\sup_{x\in\mathcal{H}}|\wh p(x) - p(x)|\lesssim\gamma_{n,h}^{(2)} +h^\beta\simeq\gamma_{n,h}^{(2)}$ with probability at least $1-\frac{1}{2}n^{-C_0}$.
Hence it follows from (\ref{ridgenessonci}) that $\sup_{x\in \wh{\mathcal{S}}(t_{1-\alpha}^e)} p(x)\lesssim \gamma_{n,h}^{(2)}$ with probability at least $1-n^{-C_0}$. 
Also notice that by (\ref{lambdar1diff})
\begin{align*}
\sup_{x\in\wh{\mathcal{S}}(t_{1-\alpha}^e)} \lambda_{r+1}(x) \leq \sup_{x\in\wh{\mathcal{S}}(t_{1-\alpha}^e)} \wh \lambda_{r+1}(x) + \sup_{x\in\wh{\mathcal{S}}(t_{1-\alpha}^e)} |\wh \lambda_{r+1}(x) - \lambda_{r+1}(x)| \leq \varepsilon_n^{(2)}.
\end{align*}
We are then able to use an argument similar to the proof of Lemma~\ref{ridgeinclusionlemma} to get (\ref{confiregsize}) for case (a). Case (b) can also be proved similarly by using the similar argument as in the proof of Lemma~\ref{ridgeinclusionlemma} and the details are omitted.
\end{proof}

\subsection{Proofs for Section \ref{unknowninference}}

{\bf Proof of Theorem \ref{betaprimeest}}
\begin{proof}
Let $\eta_n= \sup_{x\in\mathcal{H}} |\wh p(x) - p(x)|$, and denote the event $\eta_n \leq r_n^{\beta^\prime}$ by $E_0$. Using \eqref{hatpdiff}, Lemmas~\ref{biasbound} and \ref{talabound}, we have that for any $\ell>0$ and $n$ large enough 
\begin{align}
\label{etanbound}
\mathbb{P}(E_0) \geq 1- n^{-\ell}.
\end{align}
Let $c= ((C_U+2)/C_L)^{1/\beta^\prime} + 1$. Below we consider $r_n < \min(\delta_0/c,1)$ so that $\mathcal{S}\oplus (cr_n) \subset \mathcal{H}$ by assumption (F2). The following derivation will be conditional on $E_0$. 

For any $x\in \mathcal{H}\backslash[\mathcal{S}\oplus (cr_n)]$, and $y\in\mathcal{B}(x,r_n)$, we have that $y\in \mathcal{H}\backslash[\mathcal{S}\oplus ((c-1)r_n)]$ and hence 
\begin{align}
\label{phatlower}
\wh p(y) &\geq p(y) - |p(y) - \wh p(y)| \nonumber\\%
&\geq C_L (c-1)^{\beta^\prime}r_n^{\beta^\prime} - \eta_n \nonumber\\
&= (C_U+2) r_n^{\beta^\prime} - \eta_n \nonumber\\
&\geq (C_U+1) r_n^{\beta^\prime}.
\end{align}
For $x_*\in \mathcal{S}$ in assumption (F3), and any $y\in\mathcal{B}(x_*,r_n)$, we have 
\begin{align}
\label{phatupper}
\wh p(y) \leq p(y) + |p(y) - \wh p(y)| \leq C_U r_n^{\beta^\prime} +\eta_n \le (C_U+1) r_n^{\beta^\prime}.
\end{align}
Denote $a_n(x)=\sup_{y\in\mathcal{B}(x,r_n)} p(y)$ and $\wh a_n(x)=\sup_{y\in\mathcal{B}(x,r_n)} \wh p(y)$ so that $R_n=\inf_{x\in\mathcal{H} } \wh a_n(x)$. It follows from \eqref{phatlower} and \eqref{phatupper} that $\wh a_n$ achieves its infimum at a point in $\mathcal{S}\oplus (cr_n)$, denoted by $x_\dagger$, and therefore we can write $R_n = \wh a_n(x_\dagger).$ Note that assumptions (F1) and (F2) imply that $p\in\Sigma(\beta,\wt L_0,\mathcal{H})$ for some constant $\wt L_0>0$ using \eqref{whpxphdiffbound}. Let $\wt x_\dagger\in\mathcal{S}$ be a point such that $\|x_\dagger-\wt x_\dagger\|\leq cr_n$. For any $y\in\mathcal{B}(x_\dagger,r_n)$ and $\wt y\in\mathcal{B}(\wt x_\dagger,r_n)$, we have 
\begin{align}
|\wh p(y) - p(\wt y)| \leq |\wh p(y) - p(y)| + |p(y) - p(\wt y)|  \leq \eta_n + \wt L \|y-\wt y\|^\beta \leq \eta_n + \wt L (c+2)^\beta r_n^\beta. 
\end{align}
Consequently, on the one hand,
\begin{align}
R_n &\ge a_n(\wt x_\dagger) - |a_n(\wt x_\dagger) - \wh a_n(x_\dagger)| \nonumber\\
&\ge a_n(\wt x_\dagger) - \sup_{y\in\mathcal{B}(x_\dagger,r_n)} \sup_{\wt y\in\mathcal{B}(\wt x_\dagger,r_n)} |\wh p(y) - p(\wt y)| \nonumber\\
&\ge C_Lr_n^{\beta^\prime} - [ \eta_n+ \wt L (c+2)^\beta r_n^\beta].
\end{align}
On the other hand,
\begin{align}
R_n \le \wh a_n(x_*) \le a_n(x_*) +\sup_{y\in\mathcal{B}(x_*,r_n)} |\wh p(y) - p(y) | \le C_Ur_n^{\beta^\prime} + \eta_n.
\end{align}
%
%
%
%
%
Notice that $\wh\beta^\prime = \log_{r_n}(R_n)$. Therefore,
\begin{align}
\log_{r_n}(C_L-r_n^{-\beta^\prime}\eta_n - \wt L (c+2)^\beta r_n^{\beta-\beta^\prime} ) \geq \wh\beta^\prime - \beta^\prime \geq \log_{r_n}(C_U+r_n^{-\beta^\prime}\eta_n ) 
\end{align}
%
%
By \eqref{etanbound} and the assumption $\beta>\beta^\prime$, we immediately arrive at the conclusion of this theorem.
\end{proof}

{\bf Proof of Theorem \ref{testab}}
\begin{proof}
Case (b): Theorem \ref{betaprimeest} implies that for any $\ell>0$ and $n$ large enough, with probability at least $1-n^{\ell}$, $1/t_n-1/\beta^\prime< - 1/(2(\beta^\prime)^2)<0$ and hence $\rho_n^{1/t_n} / \rho_n^{1/\beta^\prime} = \rho_n^{1/t_n-1/\beta^\prime}\rightarrow\infty$. On this event, it follows from Lemma \ref{ridgeinclusionlemma} that 
for all $x\in \wh{\mathcal{S}}(\rho_n)$, $\mathcal{B}(x,\rho_n^{1/t_n})$ contains at least one ridge point, denoted by $z\in\mathcal{S}$, and hence \[\inf_{y\in \mathcal{B}(x,\rho_n^{1/t_n})}\|\nabla \wh f(y)\| \leq  \|\nabla \wh f(z)\|=\|\nabla \wh f(z) - \nabla f (z)\|.\] Consequently $T_n \leq \sup_{x\in \mathcal{S}} \|\nabla \wh f(x) - \nabla f (x)\|=:D_n$. Following similar arguments for Theorem \ref{multiplierquantile}, we can show that $\mathbb{P}(D_n \leq \tau^e_{1-\alpha}) = 1 - \alpha+o(1)$, which then leads to $\mathbb{P}(T_n < \tau^e_{1-\alpha}) \ge 1 - \alpha+o(1)$.

Case (a): Suppose that there exists $x_0\in\mathcal{S}$ such that $\|\nabla f(x_0)\| \neq 0.$ Due to the continuity of $\nabla f$, there exist $r_1\in(0,\delta_0]$ and $c_0>0$ such that $\|\nabla f(x)\|\geq c_0$ for all $x\in\mathcal{B}(x_0,r_1)$. Using Lemma \ref{ridgeinclusionlemma}, for any $\ell>0$ and $n$ large enough, $x_0\in\wh{\mathcal{S}}(\rho_n)$ with probability at least $1-n^{\ell}$. Conditional on this, when $n$ is large enough that $\rho_n^{1/t_n}\leq r_1$, we have 
\begin{align}
\label{Tnupperbound}
T_n \ge \inf_{y\in \mathcal{B}(x_0,\rho_n^{1/t_n})}\|\nabla \wh f(y)\| \ge \inf_{y\in \mathcal{B}(x_0,\rho_n^{1/t_n})}\|\nabla f(y)\| - \varepsilon_n^{(1)} \geq c_0 - \varepsilon_n^{(1)} .
\end{align}
Note that $\mathbb{P}(\varepsilon_n^{(1)} \ge c_0/2)\to0$ by Lemmas~\ref{biasbound} and \ref{talabound}, and $\mathbb{P}(\tau^e_{1-\alpha}\ge c_0/2)\to0$ by following similar arguments for Lemma \ref{wtzhquantile} and Proposition~\ref{multiplierbootapp}. From \eqref{Tnupperbound} we then arrive at the conclusion of this theorem. 
\end{proof}

\end{document}